\documentclass[10pt, journal]{IEEEtran}

\usepackage{amsmath,mathtools,amssymb,graphicx,paralist,subfigure,pdfpages,cite,algorithm,algpseudocode,paralist}
\usepackage{verbatim}
\usepackage{algorithm}
\usepackage{amsthm}
\usepackage{verbatim}
\usepackage{algorithm}
\usepackage{algpseudocode}
\newtheorem{lem}{Lemma} 
\newtheorem{theorem}{Theorem}

\newtheorem{cor}{Corollary}
\newtheorem{rmk}{Remark}
\newtheorem{assump}{Assumption}

\newtheorem{corollary}{Corollary}
\newtheorem{proposition}{Proposition}
\newenvironment{P1}
  {\begin{proof}[Proof of Theorem~\ref{main_gtsaga}]}
  {\end{proof}}
\newenvironment{P2}
  {\begin{proof}[Proof of Theorem~\ref{main_gtsvrg}]}
  {\end{proof}}

\def\mc{\mathcal}
\def\mb{\mathbf}
\def\mbb{\mathbb}
\def\ra{\rightarrow}

\def\GTVR{\textbf{\texttt{GT-VR}}}
\def\GTSAGA{\textbf{\texttt{GT-SAGA}}}
\def\GTSVRG{\textbf{\texttt{GT-SVRG}}}

\def\SAGA{\textbf{\texttt{SAGA}}}
\def\SVRG{\textbf{\texttt{SVRG}}}
\def\DSGD{\textbf{\texttt{DSGD}}}

\IEEEoverridecommandlockouts

\def\mbb{\mathbb}
\def\mb{\mathbf}
\def\mc{\mathcal}

\def\wt{\widetilde}
\def\ol{\overline}
\def\ul{\underline}

\def\bds{\boldsymbol}
\newcommand{\mn}[1]{{\left\vert\kern-0.25ex\left\vert\kern-0.25ex\left\vert\kern0.3ex #1 
		\kern0.3ex\right\vert\kern-0.25ex\right\vert\kern-0.25ex\right\vert}}

\IEEEoverridecommandlockouts
\begin{document}
\title{Variance-Reduced Decentralized Stochastic Optimization with Accelerated Convergence}
\author{Ran Xin$^\dagger$,  Usman A. Khan$^\ddagger$, and Soummya Kar$^\dagger$
		\\$^\dagger$Carnegie Mellon University, Pittsburgh, PA \hspace{1cm} $^\ddagger$Tufts University, Medford, MA
		\thanks{RX and SK are with the Electrical and Computer Engineering (ECE) Dept. at Carnegie Mellon University, \texttt{\{ranx,soummyak\}@andrew.cmu.edu}. UAK is with the ECE Dept. at Tufts University, \texttt{khan@ece.tufts.edu}. The work of SK and RX has been partially supported by NSF under award \#1513936. The work of UAK has been partially supported by NSF under awards \#1350264, \#1903972, and \#1935555. 
		}
	}
\maketitle

\begin{abstract}
This paper describes a novel algorithmic framework to minimize a finite-sum of functions available over a network of nodes. The proposed framework, that we call~\GTVR, is stochastic and decentralized, and thus is particularly suitable for problems where large-scale, potentially private data, cannot be collected or processed at a centralized server. The \GTVR~framework leads to a family of algorithms with two key ingredients: (i) \textit{local variance reduction}, that enables estimating the local batch gradients from arbitrarily drawn samples of local data; and, (ii) \textit{global gradient tracking}, which fuses the gradient information across the nodes. Naturally, combining different variance reduction and gradient tracking techniques leads to different algorithms of interest with valuable practical tradeoffs and design considerations. 

Our focus in this paper is on two instantiations of the~$\GTVR$ framework, namely~\textbf{\texttt{GT-SAGA}} and~\textbf{\texttt{GT-SVRG}}, that, similar to their centralized counterparts (\SAGA~and~\SVRG), exhibit a compromise between space and time. We show that both~\textbf{\texttt{GT-SAGA}} and~\textbf{\texttt{GT-SVRG}} achieve accelerated linear convergence for smooth and strongly convex problems and further describe the regimes in which they achieve non-asymptotic, network-independent linear convergence rates that are faster with respect to the existing decentralized first-order schemes. Moreover, we show that both algorithms achieve a linear speedup in such regimes, in that, the total number of gradient computations required at each node is reduced by a factor of $1/n$, where $n$ is the number of nodes, compared to their centralized counterparts that process all data at a single node. Extensive simulations illustrate the convergence behavior of the corresponding algorithms. 

\end{abstract}

\section{Introduction}\label{intro}
In this paper, we consider decentralized finite-sum minimization problems that take the following form:
\begin{align*}
\mbox{P1}:
\quad\min_{\mb{x}\in\mathbb{R}^p}F(\mb{x})\triangleq\frac{1}{n}\sum_{i=1}^{n}f_i(\mb{x}),
\quad
f_i(\mb{x}) \triangleq \frac{1}{m_i}\sum_{j=1}^{m_i} f_{i,j}(\mb{x}),
\end{align*}	
where each cost function~${f_i:\mbb R^p\ra\mbb R}$ is private to a node~$i$, in a~network of~$n$ nodes, and is further subdivided into an average of~$m_i$ component functions~$\{f_{i,j}\}_{j=1}^{m_i}$. This formulation has found tremendous interest over the past decade and has been studied extensively by the signal processing, control, and machine learning communities~\cite{DGD_tsitsiklis,DGD_nedich,diffusion_sayed,cons+inov_kar}. When the dataset is large-scale and further contains private information, it is often not feasible to communicate and process the entire dataset at a central location. Decentralized stochastic gradient methods thus~are preferable as they not only benefit from local (short-range) communication but also exhibit low computation complexity by sampling and processing small subsets of data at each node~$i$, instead of the entire local batch of~$m_i$ functions.

Decentralized stochastic gradient descent (DSGD) was introduced in~\cite{DSGD_nedich,diffusion_sayed,cons+inov_kar}, which combines network fusion with local stochastic gradients and has been popular in various decentralized learning tasks. However, the performance of DSGD~is mainly adversely impacted by two components: (i) the variance of the local stochastic gradients at each node; and, (ii) the dissimilarity between the datasets and local functions  across the nodes. In this paper, we propose a novel algorithmic framework, namely~\GTVR, that systematically addresses both of these aspects of~DSGD~by building an estimate of the global descend direction~$-\nabla F$ locally at each node based on local stochastic gradients. In particular, the \GTVR~framework leads to a family of algorithms with two key ingredients: (i) \textit{local variance reduction}, that estimates the local batch gradients~${\sum_j\nabla f_{i,j}}$ from arbitrarily drawn samples of local data; and, (ii) \textit{global gradient tracking}, which uses the aforementioned local batch gradient estimates and fuses them across the nodes to track the global batch gradient~${\sum_i \nabla f_i}$. Naturally, existing methods for variance reduction, such as SAG~\cite{SAG}, SAGA~\cite{SAGA}, SVRG~\cite{SVRG}, SARAH~\cite{SARAH}, and for gradient tracking, such as dynamic average consensus~\cite{DAC,NEXT,harnessing,DIGing} and dynamic average diffusion~\cite{dav_diff}, are all valid choices for the two components in~\GTVR~and lead to various design choices and practical trade-offs.

This paper focuses on smooth and strongly convex problems, where simple schemes, such as~\SAGA~and~\SVRG,~are shown to obtain linear convergence and strong performance. These two methods are extensively studied in the centralized settings and exhibit a compromise between space and time. Specifically, \SAGA~in practice demonstrates faster convergence compared with~\SVRG~\cite{SAGA,SPM_XKK}, however at the expense of additional storage requirements. Consequently, we consider the following two instantiations of the~\GTVR~framework: (i)~\GTSAGA, which is an incremental gradient method that requires~$\mc O(pm_i)$ storage cost at each node~$i$; and, (ii)~\GTSVRG, which is a hybrid gradient method that does not require additional storage but computes local batch gradients periodically, which leads to stringent requirements on network synchronization and may add latency to the overall implementation.

\emph{Related work:} Significant progress has been made recently towards decentralized first-order gradient methods. Examples include EXTRA~\cite{EXTRA}, Exact-Diffusion~\cite{Exact_Diffusion}, methods based on gradient-tracking~\cite{GT_CDC,NEXT,harnessing,DIGing,proxDGT,AGT,Network-DANE,AB,push-pull} and primal-dual methods~\cite{GT_jakovetic,GT_EXTRA_NIPS,DLM}; these full gradient methods, based on certain bias-correction principles, achieve linear convergence to the optimal solution for smooth and strongly convex problems and improve upon the well-known~DGD~\cite{DGD_nedich}, where a constant step-size leads to linear but inexact convergence. Several stochastic variants of EXTRA, Exact-Diffusion, and gradient tracking methods have been recently studied in~\cite{DSGD_NIPS,D2,DSGT_Xin,SED,xin2020improved,DSGD_vlaski_2,DSGD_Pu,DSGT_Pu}; these methods, due to the non-diminishing variance of the local stochastic gradients, converge sub-linearly to the optimal solution with decaying step-sizes and outperform their deterministic counterparts when local data batches are large and low-precision solutions suffice~\cite{SPM_XKK}. Exact linear convergence to the optimal solution has been obtained with the help of variance reduction where existing
decentralized stochastic methods include~\cite{DSA,DAVRG,DSBA,edge_DSA,ADFS,Network-DANE}. The proposed~\GTVR~framework leads to accelerated convergence over the related stochastic methods; a detailed comparison will be revisited in~Sections~\ref{results} and~\ref{simulation}.


\textit{Main contributions:} We enlist the main contributions of this paper as follows: \begin{enumerate}[(i)]
\item We describe~\GTVR, \textit{a novel algorithmic framework} to minimize a finite sum of functions over a decentralized network of nodes. 
\item Focusing on two particular instantiations of~\GTVR, \GTSAGA~and \GTSVRG, we show how different combinations of variance reduction and gradient tracking potentially lead to valuable practical considerations in terms of storage, computation, and communication tradeoffs. 
\item We show that both~\textbf{\texttt{GT-SAGA}} and~\textbf{\texttt{GT-SVRG}} achieve accelerated linear convergence to the optimal solution for smooth and strongly convex problems. 
\item We characterize the regimes in which~\textbf{\texttt{GT-SAGA}}~and \textbf{\texttt{GT-SVRG}} achieve non-asymptotic, network-independent convergence rates and exhibit a linear speedup, in that, the total number of gradient computations at each node is reduced by a factor of~$1/n$ compared to their centralized counterparts that process all data at a single node.
\end{enumerate}
To the best of our knowledge,~\GTSAGA~and~\GTSVRG~are the first decentralized stochastic methods that show \textit{provable} \textit{network-independent linear convergence} and \textit{linear speedup} without requiring the expensive computation of dual gradients or proximal mappings of the cost functions.

\textit{Outline of the paper: }Section~\ref{results} develops the class of decentralized VR algorithms proposed in this paper while Section~\ref{contribution} presents the main convergence results and a detailed comparison with the state-of-the-art. Section~\ref{simulation} provides extensive numerical simulations to illustrate the convergence behavior of the proposed methods. Section~\ref{general} presents a unified approach to cast and analyze the proposed algorithms. Sections~\ref{sGTSAGA} and~\ref{sGTSVRG} contain the convergence analysis for \textbf{\texttt{GT-SAGA}} and \textbf{\texttt{GT-SVRG}}, respectively, and Section~\ref{conclusions} concludes the paper.

\textit{Basic notation: }We use lowercase bold letters to denote vectors and~$\left\|\:\cdot\:\right\|$ to denote the Euclidean norm of a vector. The matrix~$I_d$ is the~${d\times d}$ identity, and~$\mb{1}_d$ (resp.~$\mb 0_d$) is the~$d$-dimensional column vector of all ones (resp. zeros). For two matrices~${X,Y\in\mathbb{R}^{d\times d}}$,~${X\otimes Y}$ denotes their Kronecker product. The spectral radius of a matrix~$X$ is denoted by~$\rho(X)$, while its spectral norm is denoted by~$\mn X$. The weighted infinity norm of~$\mb{x}= [x_1,\cdots,x_d]^\top$ given a positive weight vector~$\mb w = [w_1,\cdots,w_d]^\top$ is defined as~$\left\|\mb{x}\right\|_\infty^{\mb w} = \max_i |x_i|/w_{i}$ and $\mn{\cdot}_{\infty}^{\mb x}$ is the matrix norm induced by~${\left\|\:\cdot\:\right\|_\infty^{\mb w}}$. 

\section{Algorithm Development}\label{results}
In this section, we systematically build the~proposed \textbf{\texttt{GT-VR}} framework and describe its two instantiations, \textbf{\texttt{GT-SAGA}} and \textbf{\texttt{GT-SVRG}}. To this aim, we consider~\DSGD~\cite{DSGD_nedich,diffusion_sayed,cons+inov_kar}, a well-know decentralized version of stochastic gradient descent, and its convergence guarantee for smooth and strongly convex problems as follows. Let~$\mb x^*$ denote the unique minimizer of Problem~P1 and~$\mb x_i^k\in\mbb R^p$ denote the estimate of~$\mb x^*$ at node~$i$ and iteration~$k$ of~\DSGD. The update of~\DSGD~is given by
\begin{align}\label{dsgd_eq}
\mb x_i^{k+1} = \textstyle\sum_{r=1}^n \ul{w}_{ir} \mb x_r^k - \alpha \cdot \nabla f_{i,s_i^k}(\mb x_i^k),\qquad k\geq0,
\end{align}
where the matrix~${\ul W=\{\ul{w}_{ir}\}\in\mbb R^{n\times n}}$ collects the weights that each node assigns to its neighbors and the index~$s_i^k$ is chosen uniformly at random from the set~$\{1,\ldots,m_i\}$ at each iteration~$k$. Assuming bounded variance of~$\nabla f_{i,s_i^k}(\mb x_i^k)$, i.e.,~$\mathbb{E}[\|\nabla f_{i,s^k_i}(\mb x^k_i)-\nabla f_i(\mb x_i^k)\|^2 ~|\:\mb x_i^k] \leq\nu^2,\forall i,k,$ 
and cost functions to be smooth and strongly convex, it can be shown that with an appropriate constant step-size~$\alpha$ the mean-squared error~$\mathbb{E}[\|\mb x^k_i-\mb x^*\|^2]$, at each node~$i$, decays linearly up to a steady state error such that~\cite{SED} 
\begin{align}
\label{DSGD_convergence}
\limsup_{k\rightarrow\infty}\frac{1}{n}\sum_{i=1}^{n}\mathbb{E}&\left[\big\|\mb x^k_i-\mb x^*\big\|^2\right] \nonumber\\
=&~\mc{O}\Big(\frac{\alpha\nu^2}{n\mu}
+ \frac{\alpha^2\kappa^2\nu^2}{1-\sigma}
+ \frac{\alpha^2\kappa^2\zeta^2} {(1-\sigma)^2}\Big),
\end{align}
where~${\zeta^2 := \frac{1}{n}\sum_{i=1}^{n}\left\|\nabla f_i\left(\mb x^*\right)\right\|_2^2}$ and~$(1-\sigma)$ is the spectral gap of the weight matrix~$\ul W$. This steady-state error, due to the presence of~$\nu^2$ and~$\zeta^2$, can be eliminated with the help of decaying step-size $\alpha_k = \mc{O}(1/k)$; however, the convergence rate becomes sub-linear~$\mc{O}(1/k)$~\cite{DSGD_Pu}. In other words, there is an inherent rate/accuracy trade-off in the performance of~\DSGD. The proposed \textbf{\texttt{GT-VR}} framework, built on global gradient tracking and local variance reduction, completely removes the steady-state error of~\DSGD~and achieves fast convergence with a constant step-size to the exact solution.

\vspace{-0.4cm}
\subsection{The~\GTVR~framework}
The proposed~\GTVR~framework combines two well-known techniques from recent centralized and decentralized optimization literature to systematically eliminate the steady-state error of~\DSGD~and as a consequence recovers linear convergence to the exact solution. The framework has two key ingredients:

(i) \textit{Local Variance Reduction:} \GTVR~removes the performance limitation due to the variance~$\nu^2$ of the stochastic gradients by asymptotically estimating the local batch gradient~$\nabla f_i$, at each node~$i$, based on randomly drawn samples from the local dataset. Many variance reduction schemes, e.g.,~\cite{SAG,SAGA,SVRG,SARAH}, are applicable here and a suitable one can be chosen according to the application of interest and problem specifications. 

(ii) \textit{Global Gradient Tracking:} The other error source~$\zeta^2$ is due to the fact that~$\nabla f_i(\mb{x}^*) \neq \mb 0_p,\forall i$, in general, because of the difference between the local and global cost functions. This issue is addressed with the help of gradient tracking techniques~\cite{DAC,dav_diff,NEXT,DIGing,harnessing} that properly fuse the local batch gradient estimates (obtained from the local variance reduction procedures described above) to track the global batch gradient.

Our focus in this paper is on smooth and strongly convex problems for which the  variance reduction methods \textbf{\texttt{SAGA}}~\cite{SAGA} and \textbf{\texttt{SVRG}}~\cite{SVRG}, in centralized settings, are shown to achieve strong practical performance and theoretical guarantees. These two methods contrast each other, in that, they can be viewed as a compromise between space and time~\cite{SAGA}, where \SAGA~requires additional storage but, in practice, demonstrates faster convergence as compared to~\SVRG,~where additional storage is not required. Additionally, the two methods are build upon different variance-reduction principles, i.e.,~\SAGA~is a randomized incremental gradient method, whereas~\SVRG~is a hybrid gradient method that evaluates batch gradients periodically in addition to stochastic gradient computations at each iteration, as will be detailed further. We thus explicitly focus on these two methods in this paper, formally described next.

\subsubsection{\GTSAGA} Algorithm~\ref{GT-SAGA} formally describes the \SAGA-based implementation of~\GTVR. 
To implement the gradient estimator, each node~$i$ maintains a table of component gradients $\{\nabla f_{i,j}(\mb{z}_{i,j}^k)\}_{j=1}^{m_i}$, where $\mb{z}_{i,j}^k$ is the most recent iterate at which the component gradient~$\nabla f_{i,j}$ was evaluated up to iteration~$k$. At each iteration~$k$, each node~$i$ samples an index~$s_i^k$ uniformly at random from the local indices~$\{1,\cdots,m_i\}$ and computes its local gradient estimator as
\begin{align*}
\mb{g}_{i}^{k} =&~\nabla f_{i,s_i^{k}}\big(\mb{x}_{i}^{k}\big) - \nabla f_{i,s_i^{k}}\big(\mb{z}_{i,s_i^{k}}^{k}\big) 
+\textstyle{\frac{1}{m_i}\sum_{j=1}^{m_i}}\nabla f_{i,j}\big(\mb{z}_{i,j}^{k}\big).
\end{align*}
After~$\mb{g}_i^{k}$ is computed, the~$s_i^{k}$-th
element in the gradient table is replaced by~$\nabla f_{i,s_i^k}(\mb{x}_i^k)$, while other entries remain unchanged. The local estimators~$\mb{g}_i^k$'s are then fused over the network to compute~$\mb y_i^{k}$, which tracks the global batch gradient~$\nabla F$ at each node~$i$, and is used as the descent direction to update the local estimate~$\mb x_i^k$ of the optimal solution.  
Clearly, each local estimator~$\mb g_i^{k}$ approximates the local batch gradient~$\nabla f_i$ in an incremental manner via the average of the past component gradients in the table. This implementation procedure results in a storage cost of~$\mc{O}(pm_{i})$ at each node~$i$, which can be reduced to~$\mc O(m_i)$ for certain structured problems~\cite{SAG,SAGA}.

\begin{algorithm}
\caption{\textbf{\texttt{GT-SAGA}} at each node~$i$}
\label{GT-SAGA}
	\begin{algorithmic}[1]
		\Require{$\mb{x}_i^0$;~${\mb{z}_{i,j}^1=\mb{x}_i^0},~{\forall j\in\{1,\cdots,m_i\}};$ $\alpha$; $\{\ul{w}_{ir}\}_{r=1}^n$; $\mb{y}_i^0=\mb{g}_i^0=\nabla f_i(\mb{x}_i^0)$.}
		\For{{$k= 0,1,2,\cdots$}}
		\State{Update the local estimate of the solution:$$\mb{x}_{i}^{k+1} = \textstyle{\sum_{r=1}^{n}}\ul{w}_{ir}\mb{x}_{r}^{k} - \alpha\mb{y}_{i}^{k};$$}
		\State{Select~$s_{i}^{k+1}$ uniformly at random from~$\{1,\cdots,m_i\}$;}
		\State{Update the local gradient estimator:\begin{align*}\mb{g}_{i}^{k+1} =&~\nabla f_{i,s_i^{k+1}}\big(\mb{x}_{i}^{k+1}\big) - \nabla f_{i,s_i^{k+1}}\big(\mb{z}_{i,s_i^{k+1}}^{k+1}\big) \nonumber\\
		&+\textstyle{\frac{1}{m_i}\sum_{j=1}^{m_i}}\nabla f_{i,j}\big(\mb{z}_{i,j}^{k+1}\big);\end{align*}}
		\State{If~$j = s_i^{k+1}$, then~$\mb{z}_{i,j}^{k+2} = \mb{x}_{i,j}^{k+1}$; else~$\mb{z}_{i,j}^{k+2} = \mb{z}_{i,j}^{k+1}$.}
		\State{Update the local gradient tracker:$$\mb{y}_{i}^{k+1} = \textstyle{\sum_{r=1}^{n}}\ul{w}_{ir}\mb{y}_{r}^{k} + \mb{g}_i^{k+1} - \mb{g}_i^{k};$$}
		\EndFor
	\end{algorithmic}
\end{algorithm}

\subsubsection{\GTSVRG} Algorithm~\ref{GT-SVRG} formally describes the~\SVRG-based implementation of~\GTVR. In  contrast to \textbf{\texttt{GT-SAGA}} that incrementally approximates the local batch gradients via past component gradients, \textbf{\texttt{GT-SVRG}} achieves variance reduction by evaluating the local batch gradients~$\nabla f_i$'s \textit{periodically}. \textbf{\texttt{GT-SVRG}} may be interpreted as a ``double loop" method, where each node~$i$, at every outer loop update~$\{\mb{x}^{tT}_i\}_{t\geq0}$, calculates a local full gradient~$\nabla f_i(\mb{x}^{tT}_i)$ that is retained in the subsequent inner loop iterations to update the local gradient estimator~$\mb{v}_i^k$, i.e., for~$k\in[tT,(t+1)T-1$],
$$
\mb{v}_{i}^{k} = \nabla f_{i,s_i^{k}}\big(\mb{x}_{i}^{k}\big) - \nabla f_{i,s_i^{k}}\big(\mb{x}_{i}^{tT}\big) + \nabla f_i\big(\mb{x}_{i}^{tT}\big).
$$
Clearly, \textbf{\texttt{GT-SVRG}} eliminates the requirement of storing the most recent component gradients at each node and thus has a favorable storage cost compared with \textbf{\texttt{GT-SAGA}}. However, this advantage comes at the expense of evaluating two stochastic gradients~$\nabla f_{i,s_i^{k}}\big(\mb{x}_{i}^{k}\big)$ and~$\nabla f_{i,s_i^{k}}\big(\mb{x}_{i}^{tT}\big)$ at every iteration, in addition to calculating the local batch gradients~$\nabla f_i$'s every~$T$ iterations. See Remarks~\ref{R1} and~\ref{R2} for additional discussion. 
\begin{algorithm}[!h]
	\caption{\textbf{\texttt{GT-SVRG}} at each node~$i$}
	\label{GT-SVRG}
	\begin{algorithmic}[1]
		\Require{$\mb{x}_i^0$;~$\bds\tau_i^0=\mb{x}_i^0$;~$\alpha$;~$\{\ul{w}_{ir}\}_{r=1}^n$;~$T$;~$\mb{y}_i^0=\mb{v}_i^0=\nabla f_i(\mb{x}_i^0)$.}
		\For{{$k= 0,1,2,\cdots$}}
		\State{Update the local estimate of the solution:$$\mb{x}_{i}^{k+1} = \textstyle{\sum_{r=1}^{n}}\ul{w}_{ir}\mb{x}_{r}^{k} - \alpha\mb{y}_{i}^{k};$$}
		\State{Select~$s_{i}^{k+1}$ uniformly at random from~$\{1,\cdots,m_i\}$;}
		\State{If~${\bmod (k+1,T) = 0}$, then~${\bds\tau_i^{k+1} = \mb{x}_i^{k+1}}$;
		
	else $\bds\tau_i^{k+1} = \bds\tau_i^{k}$.}
		\State{Update the local stochastic gradient estimator:\begin{align*}\mb{v}_{i}^{k+1} =&~\nabla f_{i,s_i^{k+1}}\big(\mb{x}_{i}^{k+1}\big) - \nabla f_{i,s_i^{k+1}}\big(\bds\tau_i^{k+1}\big) \nonumber\\
		&+ \nabla f_i\big(\bds\tau_i^{k+1}\big);\end{align*}}
		\State{Update the local gradient tracker:$$\mb{y}_{i}^{k+1} = \textstyle{\sum_{r=1}^{n}}\ul{w}_{ir}\mb{y}_{r}^{k} + \mb{v}_i^{k+1} - \mb{v}_i^{k};$$}
		\EndFor
	\end{algorithmic}
\end{algorithm}      

\section{Main Results}\label{contribution}
The convergence results for \textbf{\texttt{GT-SAGA}} and \textbf{\texttt{GT-SVRG}} are established under the following assumptions.
\begin{assump}\label{sc}
The global cost function~$F$ is~$\mu$-strongly convex, i.e.,~$\forall\mb{x}, \mb{y}\in\mbb{R}^p$ and for some~$\mu>0$, we have
	\begin{equation*}
	F(\mb{y})\geq F(\mb{x})+ \big\langle\nabla F(\mb{x}), \mb{y}-\mb{x}\big\rangle+\frac{\mu}{2}\|\mb{x}-\mb{y}\|^2.
	\end{equation*}
\end{assump}
\noindent We note that under Assumption 1, the global cost function~$F$ has a unique minimizer, denoted as~$\mb{x}^*$.
\begin{assump}\label{smooth}
Each local cost function~$f_{i,j}$ is~$L$-smooth, i.e.,~$\forall\mb{x}, \mb{y}\in\mbb{R}^p$ and for some~$L>0$, we have
	\begin{equation*}
	\qquad\|\mb{\nabla} f_{i,j}(\mb{x})-\mb{\nabla} f_{i,j}(\mb{y})\|\leq L\|\mb{x}-\mb{y}\|.
	\end{equation*}
\end{assump}
\noindent Clearly, under Assumption~\ref{smooth}, the global cost~$F$ is also~$L$-smooth and~${L\geq\mu}$. We use~${Q \coloneqq L/\mu}$ to denote the condition number of the global cost~$F$. 

\begin{assump}\label{connect}
The weight matrix~${\ul W = \{\ul{w}_{ir}\}}$ associated with the network~$\mc{G}$ is primitive and  doubly stochastic.
\end{assump}

\noindent Assumption~\ref{connect} is not only restricted to undirected graphs and is further satisfied by the class of strongly-connected directed graphs that admit doubly stochastic weights. This assumption implies that the second largest singular value $\sigma$ of~$\ul W$ is less than~$1$, i.e,~$\sigma=\mn {\ul{W}-\frac{1}{n}\mb{1}_n\mb{1}_n^\top}<1$~\cite{matrix_analysis}. Note that although we focus on the basic case of static networks which appear, for instance, in data centers, the convergence analysis provided here can be possibly extended to the more general case of time-varying dynamic networks following the methodology in~\cite{TV-AB}.

We denote~$M\coloneqq \max_{i}m_i$ and~$m\coloneqq \min_{i}m_i$, where~$m_i$ is the number of local component functions at node~$i$. The main convergence results of \textbf{\texttt{GT-SAGA}} and~\textbf{\texttt{GT-SVRG}} are summarized respectively in the following theorems. 

\begin{theorem}[Mean-square convergence of \textbf{\texttt{GT-SAGA}}]\label{main_gtsaga}
Let Assumptions~\ref{sc},~\ref{smooth}, and~\ref{connect} hold. If the step-size~$\alpha$ in~\textbf{\texttt{GT-SAGA}} is such that~
$$\alpha=\min\left\{\mc{O}\left(\tfrac{1}{\mu M}\right),\mc{O}\left(\tfrac{m}{M}\tfrac{(1-\sigma)^2}{LQ}\right)\right\},$$
then we have:~$\forall k\geq0,~\forall i\in\{1,\cdots,n\}$, and for some~$c>0$,
\begin{align*}
\mathbb{E}\left[\left\|\mb{x}_i^k-\mb{x}^*\right\|^2\right] \leq c\left(1-\min\left\{\mc{O}\left(\tfrac{1}{M}\right),\mc{O}\left(\tfrac{m}{M}\tfrac{(1-\sigma)^2}{Q^2}\right)\right\}\right)^k\!.
\end{align*}
\textbf{\texttt{GT-SAGA}} thus achieves an~$\epsilon$-optimal solution
of~$\mb x^*$ in $$\mc{O}\left(\max\left\{M,\tfrac{M}{m}\tfrac{Q^2}{(1-\sigma)^2}\right\}\log\tfrac{1}{\epsilon}\right)$$ component gradient computations (iterations) at each node.
\end{theorem}
\begin{theorem}[Mean-square convergence of \textbf{\texttt{GT-SVRG}}]\label{main_gtsvrg}
Let Assumptions~\ref{sc},~\ref{smooth}, and~\ref{connect} hold. If the step-size~$\alpha$ and the length~$T$ of the inner loop are such~that
$$\alpha = \mc{O}\left(\tfrac{(1-\sigma)^2}{LQ}\right), \qquad T =\mc{O}\left(\tfrac{Q^2\log Q}{(1-\sigma)^2}\right),$$
then we have:~$\forall t\geq0,~\forall i\in\{1,\cdots,n\}$, and for some~$\ul{c}>0$,
\begin{align*}
\mathbb{E}\left[\left\|\mb{x}_i^{tT}-\mb{x}^*\right\|^2\right] \leq \ul{c}\cdot 0.7^t
\end{align*}
\textbf{\texttt{GT-SVRG}} thus achieves an~$\epsilon$-optimal solution of~$\mb x^*$ in $$\mc{O}\left(\left(M+\tfrac{Q^2\log Q}{(1-\sigma^2)^2}\right)\log\tfrac{1}{\epsilon}\right)$$ component gradient computations at each node.
\end{theorem}
Theorems~\ref{main_gtsaga} and~\ref{main_gtsvrg} lead to the following linear convergence rates for \textbf{\texttt{GT-SAGA}} and \textbf{\texttt{GT-SVRG}} on almost every sample path, following directly from Chebyshev's inequality and the Borel-Cantelli lemma; see Lemma~\ref{BC} for details.
\begin{cor}[Almost sure convergence of \textbf{\texttt{GT-SAGA}}]\label{as_gtsaga}
Let Assumptions~\ref{sc},~\ref{smooth} and~\ref{connect} hold. For the choice of the step-size~$\alpha$ in Theorem~\ref{main_gtsaga}, we have:~$\forall i\in\{1,\cdots,n\}$,
\begin{align*}
\mathbb{P}\left(\lim_{k\rightarrow\infty}\gamma_g^{-k}\left\|\mb{x}_i^k-\mb{x}^*\right\|^2 = 0\right) = 1,
\end{align*}
where~$\gamma_g = 1-\min\left\{\mc{O}\left(M^{-1}\right),\mc{O}\left(mM^{-1}(1-\sigma)^2 Q^{-2}\right)\right\}$.
\end{cor} 
\begin{cor}[Almost sure convergence of \textbf{\texttt{GT-SVRG}}]\label{as_gtsvrg}
Let Assumptions~\ref{sc},~\ref{smooth} and~\ref{connect} hold. For the choice of the step-size~$\alpha$ and the length~$T$ of the inner loop in Theorem~\ref{main_gtsvrg}, we have:~$\forall i\in\{1,\cdots,n\}$, 
\begin{align*}
\mathbb{P}\left(\lim_{t\rightarrow\infty}(0.7+\delta)^{-t}\left\|\mb{x}_i^{tT}-\mb{x}^*\right\|^2 = 0\right) = 1,
\end{align*} 
where~$\delta>0$ is an arbitrary small constant.
\end{cor}
We discuss some salient features of the proposed algorithms next and compare them with the state-of-the-art. 


\begin{rmk}[\textbf{Big data regime}]\label{R1}
When each node~has~a~large dataset such that ${M\approx m \gg Q^2(1-\sigma)^{-2}}$, we note that both \textbf{\texttt{GT-SAGA}} and \textbf{\texttt{GT-SVRG}}, achieve an $\epsilon$-optimal solution with a network-independent component gradient computation complexity of ${\mc{O}(M\log\frac{1}{\epsilon})}$ at each node; in contrast, centralized \SAGA~and~\SVRG, that process all data on a single node, require~${\mc{O}((nM+Q)\log\frac{1}{\epsilon})\approx\mc{O}(nM\log\frac{1}{\epsilon})}$ component gradient computations~\cite{SAGA,SVRG}. \textbf{\texttt{GT-SAGA}} and \textbf{\texttt{GT-SVRG}} therefore achieve a non-asymptotic, linear speedup in this big data regime, i.e., the number of component gradient computations required per node is reduced by a factor of~${1/n}$ compared with their centralized counterparts\footnote{We emphasize that linear speedup, although desirable and somewhat plausible, is not necessarily achieved for decentralized methods in general. 
In other words, the advantage of parallelizing an algorithm over~$n$ nodes may not naturally result into a performance improvement of~$n$.
}.
\end{rmk}

\begin{rmk}[\textbf{\textbf{\texttt{GT-SAGA}}  versus \textbf{\texttt{GT-SVRG}}}]\label{R2} 
It can be observed from Theorems~\ref{main_gtsaga} and~\ref{main_gtsvrg} that when data samples are unevenly distributed across the nodes, i.e.,~$\tfrac{M}{m}\gg1$, \textbf{\texttt{GT-SVRG}} achieves a lower gradient computation complexity than \textbf{\texttt{GT-SAGA}}. However, an uneven data distribution may adversely impact the practical implementation of \textbf{\texttt{GT-SVRG}}. This is because 
\textbf{\texttt{GT-SVRG}} requires a highly synchronized communication network as all nodes need to evaluate their local batch gradients every~$T$ iterations and cannot proceed to the next inner loop until all nodes complete this local computation. As a result, the nodes with smaller datasets have a relatively long idle time at the end of each inner loop that leads to an increase in overall wall-clock time. 
Indeed, the inherent trade-off between \textbf{\texttt{GT-SAGA}} and \textbf{\texttt{GT-SVRG}} is the network synchrony versus the gradient storage. For structured problems, where the component gradients can be stored efficiently, \textbf{\texttt{GT-SAGA}} may be preferred due to its flexibility of implementation and less dependence on network synchronization. Conversely, if the problem of interest is large-scale, i.e.,~$m$ is very large, and storing all component gradients is not feasible, \textbf{\texttt{GT-SVRG}} may become a more appropriate choice.  
\end{rmk}

\begin{rmk}[\textbf{Communication complexities}]\label{R8}
Note that since \textbf{\texttt{GT-SAGA}} incurs~$\mc{O}(1)$ communication round per node at each iteration, its total communication complexity is the same as its iteration complexity, i.e., $\mc{O}\big(\max\big\{M,\tfrac{M}{m}\tfrac{Q^2}{(1-\sigma)^2}\big\}\log\tfrac{1}{\epsilon}\big)$. For \textbf{\texttt{GT-SVRG}}, we note that a total number of $\mc{O}(\log\frac{1}{\epsilon})$ outer-loop iterations are required, where each outer-loop iteration incurs $T = \mc{O}\left(Q^2(1-\sigma)^{-2}\log Q\right)$ rounds of communication per node, leading to a total communication complexity of $\mc{O}\left(Q^2(1-\sigma)^{-2}\log Q\log\frac{1}{\epsilon} \right)$. Clearly, in a big data regime where each node has a large dataset, \textbf{\texttt{GT-SVRG}} achieves a lower communication complexity than \textbf{\texttt{GT-SAGA}}.
\end{rmk}

\begin{rmk}[\textbf{Comparison with Related Work}]
Existing decentralized variance-reduced (VR) gradient methods include: DSA~\cite{DSA} that integrates EXTRA~\cite{EXTRA} with SAGA~\cite{SAGA} and was the first decentralized VR method; DAVRG that combines Exact Diffusion~\cite{Exact_Diffusion} and AVRG~\cite{AVRG_ICASSP}; DSBA~\cite{DSBA} that uses proximal mapping~\cite{point-SAGA} to accelerate DSA; Ref.~\cite{edge_DSA} that applies edge-based method~\cite{edge_AL} to DSA; and ADFS~\cite{ADFS} that is a decentralized version of the accelerated randomized proximal coordinate gradient method~\cite{APCG} based on the dual of Problem P1. 
Both~\textbf{\texttt{GT-SAGA}} and~\textbf{\texttt{GT-SVRG}} improve upon the convergence rates in terms of the joint dependence on~$Q$ and~$m$ for these methods, especially in the ``big data" scenarios where~$m$ is very large, with the exception of DSBA and ADFS. We note that DSBA~\cite{DSBA} and ADFS~\cite{ADFS}, both achieve better a gradient computation complexity albeit at the expense of computing the proximal mapping of a component function at each iteration that is in general very expensive. Another recent work~\cite{Network-DANE} considers gradient tracking and variance reduction and proposes a decentralized SVRG type algorithm. 
However, the convergence of the decentralized SVRG in~\cite{Network-DANE} is only established when the local functions are sufficiently similar.
In contrast, \textbf{\texttt{GT-SAGA}} and \textbf{\texttt{GT-SVRG}} proposed in this paper 
achieve accelerated linear convergence for arbitrary local functions and are robust to the heterogeneity of local functions and data distributions.
Finally, we emphasize that all existing decentralized VR methods require symmetric weights and thus undirected networks. In contrast,~\textbf{\texttt{GT-SAGA}} and~\textbf{\texttt{GT-SVRG}} only require doubly stochastic weights and therefore can be implemented over directed graphs that admit doubly stochastic weights~\cite{weight_balance_digraph}, providing a more flexible topology design.
\end{rmk}


\begin{figure*}
\centering
\includegraphics[width=2.15in]{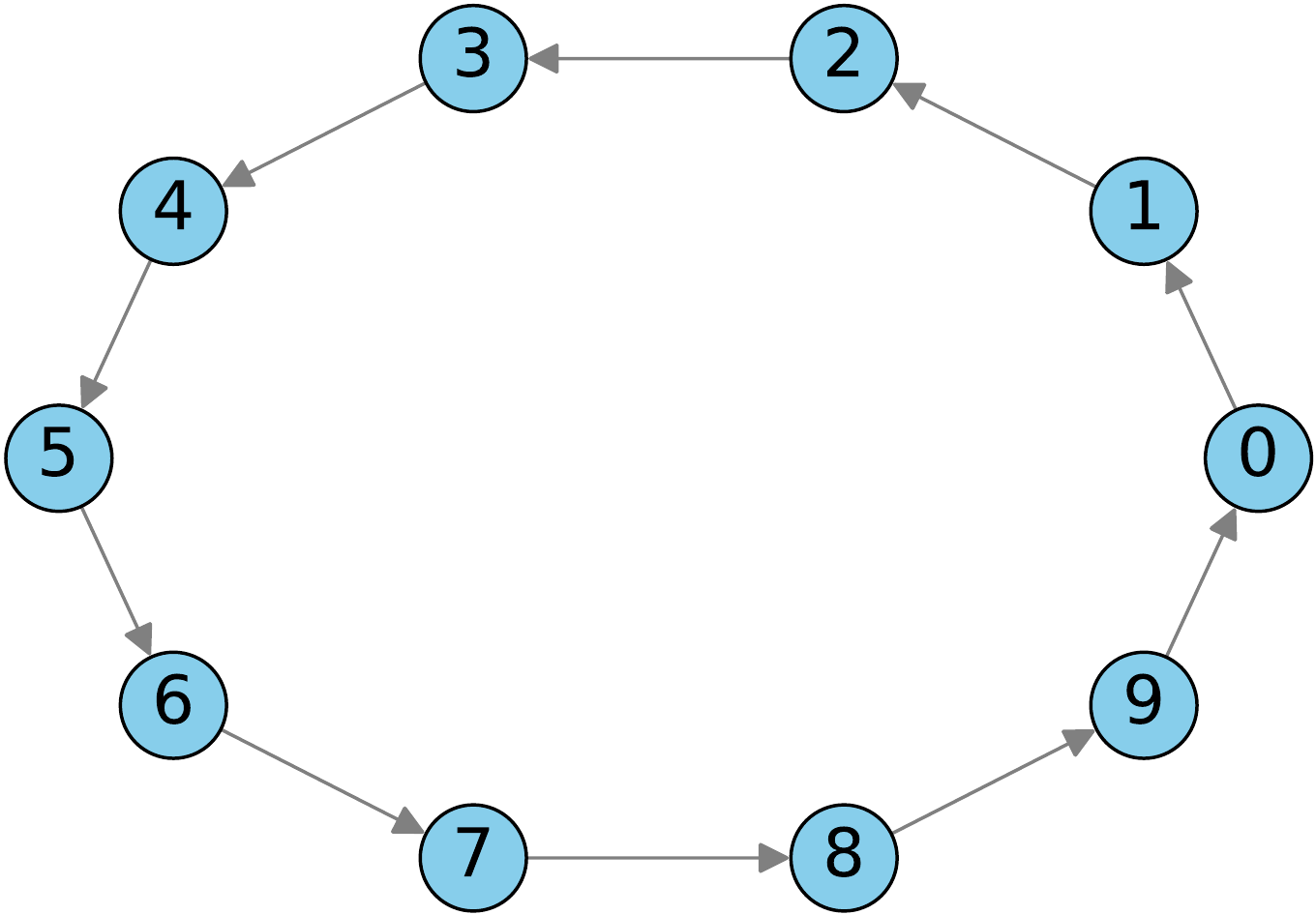}\qquad
\includegraphics[width=2.15in]{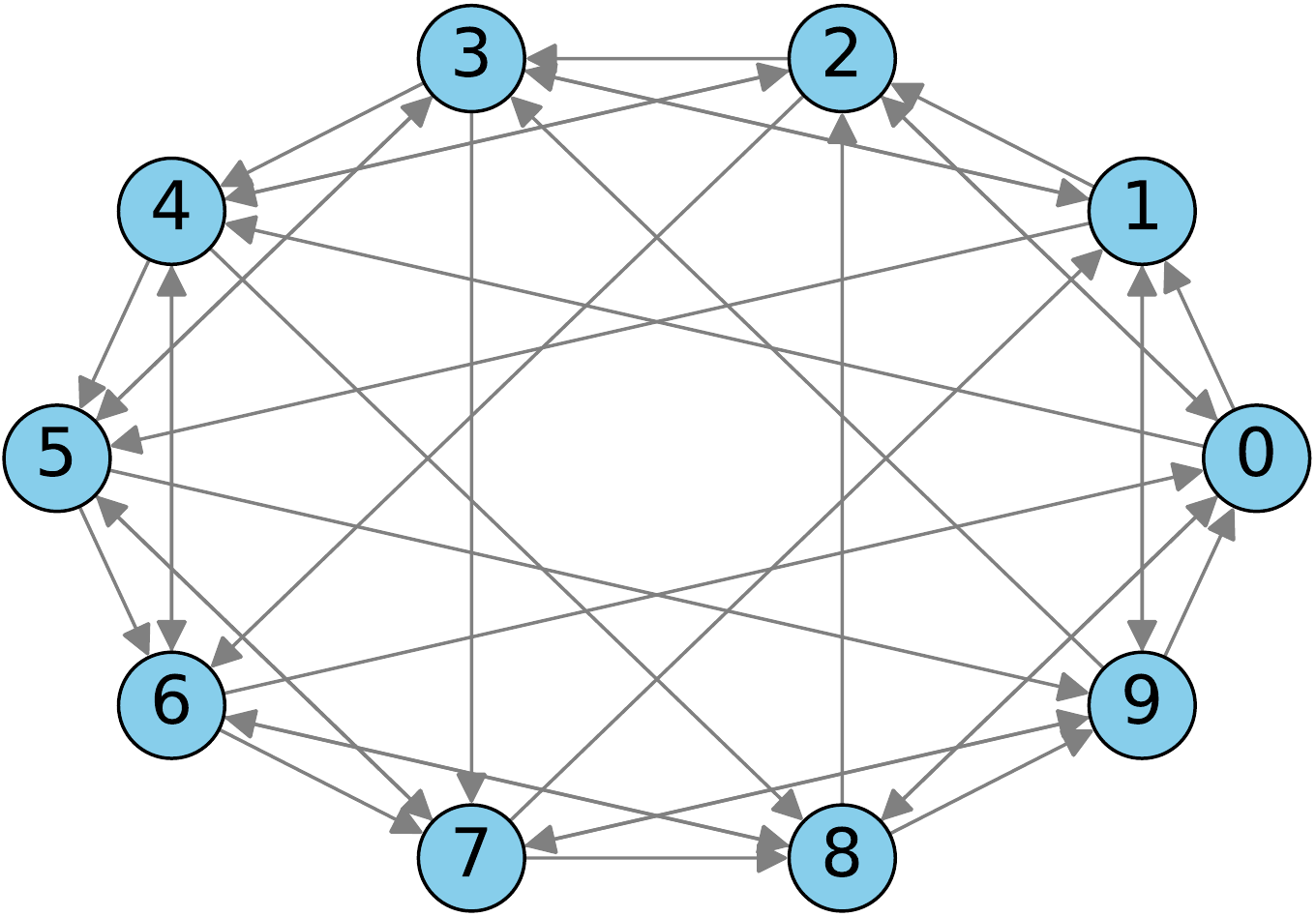}\qquad
\includegraphics[width=2.15in]{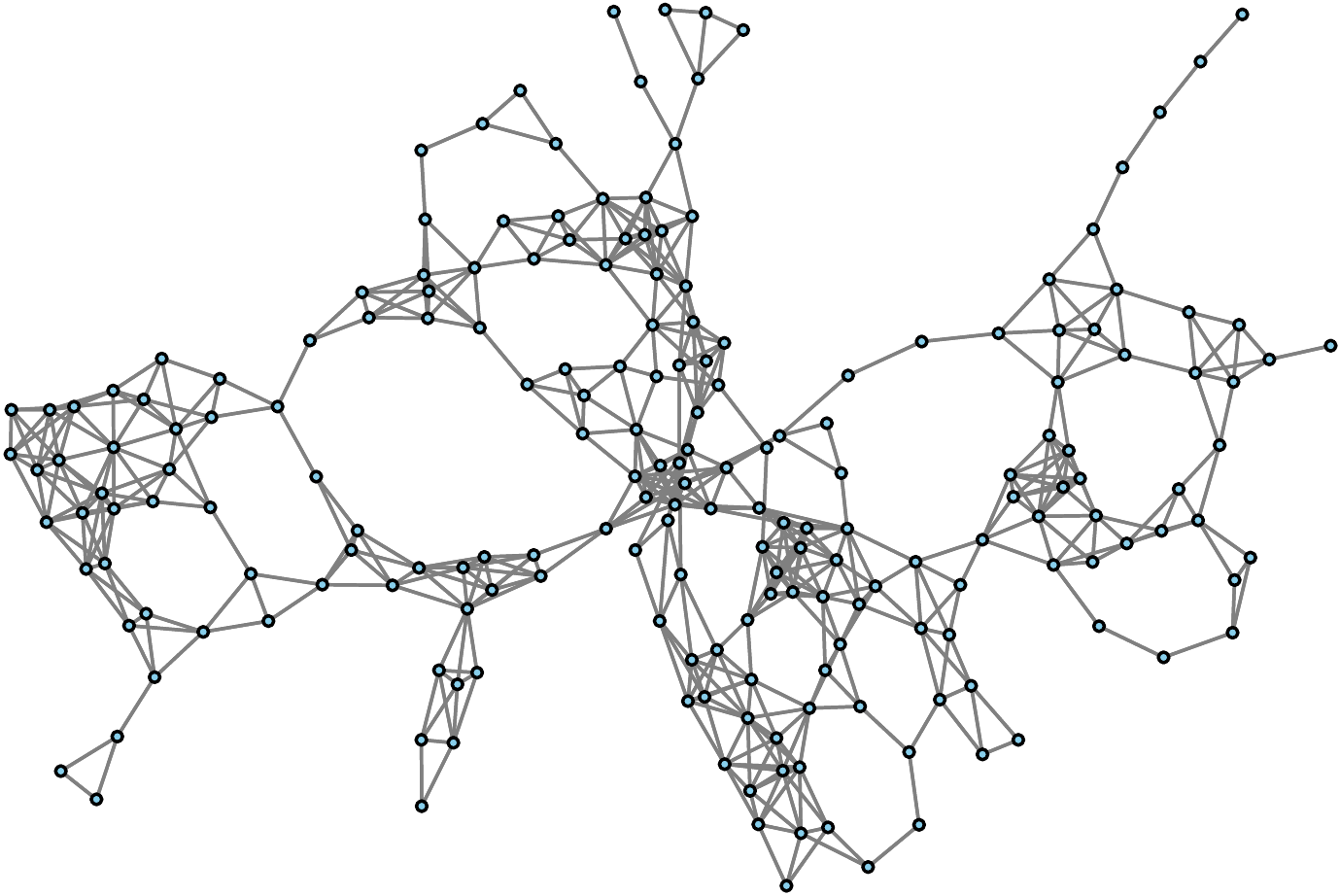}
\caption{The directed ring graph with~$10$ nodes, directed exponential graph with~$10$ nodes, and an undirected geometric graph with~$200$ nodes.}
\label{graphs}
\end{figure*}

\section{Numerical Experiments} \label{simulation}
In this section, we numerically demonstrate the convergence behavior of \textbf{\texttt{GT-SAGA}} and \textbf{\texttt{GT-SVRG}} under different regimes of interest and compare their performances with the-state-of-the-art decentralized stochastic first-order algorithms under different graph topologies and datasets. We consider a decentralized training problem where a network of~$n$ nodes with~$m$ data samples locally at each node cooperatively finds a regularized logistic regression model for binary classification:
\begin{align*}
F(\mb{x}) 
= \frac{1}{n}\sum_{i=1}^{n}\frac{1}{m}\sum_{j=1}^{m}\log\left[1+e^{-(\mb{x}^\top\bds{\theta}_{ij})\xi_{ij}}\right]+\frac{\lambda}{2}\|\mb{x}\|_2^2, \nonumber
\end{align*}
where~${\bds{\theta}_{ij}\in\mathbb{R}^{p}}$ denotes the feature vector of the~$j$-th data sample at the~$i$-th node,~${\xi_{ij}\in\{-1,+1\}}$ is the corresponding binary label, and~$\lambda$ is a regularization parameter to prevent overfitting of the training data. The datasets in question are summarized in Table~\ref{datasets} and all feature vectors are normalized to be unit vectors, i.e.,~$\|\bds{\theta}_{ij}\| = 1, \forall i,j$. The graph topologies under considerations, shown in Fig~\ref{graphs}, are directed ring graphs, directed exponential graphs, and undirected nearest-neighbor geometric graphs, all with self loops. We note that the directed ring graph has the weakest connectivity among all strongly-connected graphs; directed exponential graphs, where each node sends information to the nodes~$2^0,2^1,2^2,\cdots$ hops away, are sparse yet well-connected and therefore are often preferable when one has the freedom to design the graph topology; undirected nearest-neighbor geometric graphs, where two nodes are connected if they are in physical vicinity, are weakly-connected and often arise in ad hoc settings such as robotics swarms, IoTs, and edge computing networks. The doubly stochastic weights for directed ring and exponential graphs are chosen as uniform weights, while the weights for geometric graphs are generated by the Metropolis rule~\cite{tutorial_nedich}. The parameters of all algorithms in all cases are manually tuned for best performance. We characterize the performance of the decentralized optimization methods in question in terms of the optimality gap~${\frac{1}{n}\sum_{i=1}^n\left(F(\mb{x}_k^i)-F(\mb{x}^*)\right)}$ and model accuracy on the test data sets over epochs, where we assume that each node possesses the same number~$m$ of data samples and one epoch represents~$m$ gradient computations per node.

\begin{table}[!ht]
\caption{Summary of datasets used in numerical experiments. All datasets are available in LIBSVM~\cite{LIBSVM}.}
\begin{center}
\begin{tabular}{|c|c|c|c|}
\hline
\textbf{Dataset} & \textbf{train} ($N = nm$)  & \textbf{dimension} ($p$) & \textbf{test}\\ \hline
Fashion-MNIST & $10,\!000$ & $784$ & $4,\!000$  \\ \hline
Covertype & $400,\!000$ & $54$ & $181,\!012$ \\ \hline
CIFAR-10 & $10,\!000$ & $3,\!072$ & $2,\!000$ \\ \hline
Higgs &  $90,\!000$ & $28$ & $8,\!050$   \\ \hline
a9a & $32,\!560$  & $123$ & $16,\!282$ \\ \hline
w8a & $49,\!740$ & $300$ & $14,\!960$ \\ \hline
\end{tabular}
\end{center}
\label{datasets}
\end{table}

\subsection{Big data regime: Non-asymptotic network-independence convergence and linear speedup}
In this subsection, we demonstrate the convergence behavior of \textbf{\texttt{GT-SAGA}} and \textbf{\texttt{GT-SVRG}} in the big data regime, i.e.,~${m\approx Q^2(1-\sigma)^{-2}}$. To this aim, we choose~$500,\!000$ training samples from the Covertype dataset, equally distributed in a network of~${n=10}$ nodes such that each node has~${m = 50,\!000}$ data samples and set the regularization parameter as~${\lambda = 0.01}$ that leads to~${Q \approx 25}$, where~$Q$ is the condition number of~$F$. We test the performance of \textbf{\texttt{GT-SAGA}} and \textbf{\texttt{GT-SVRG}} over different graph topologies, i.e., the directed ring, the directed exponential, and the complete graph with~$10$ nodes; the second largest singular eigenvalues of the weight matrices associated with these three graphs are~${\sigma = 0.951,0.6,0}$, respectively. It can be verified that the big data condition holds for the optimization problem defined on these three graphs. The experimental results are shown in Fig.~\ref{speedup} (left and middle) and we observe that, in this big data regime, the convergence rates of \textbf{\texttt{GT-SAGA}} and \textbf{\texttt{GT-SVRG}} are not affected by the network topology.  We next illustrate the speedup of \textbf{\texttt{GT-SAGA}} and \textbf{\texttt{GT-SVRG}} compared with their centralized counterparts. The speedup is characterized as the ratio of the number of component gradient computations required for centralized \SAGA~and \SVRG~that execute on \emph{a single node} over the number of component gradient computations required \emph{at each node} for \textbf{\texttt{GT-SAGA}} and \textbf{\texttt{GT-SVRG}} to achieve the optimality gap of~$10^{-13}$. It can be observed in Fig~\ref{speedup} (right) that linear speedup is achieved for both methods.

\begin{figure*}
\centering
\includegraphics[width=2.25in]{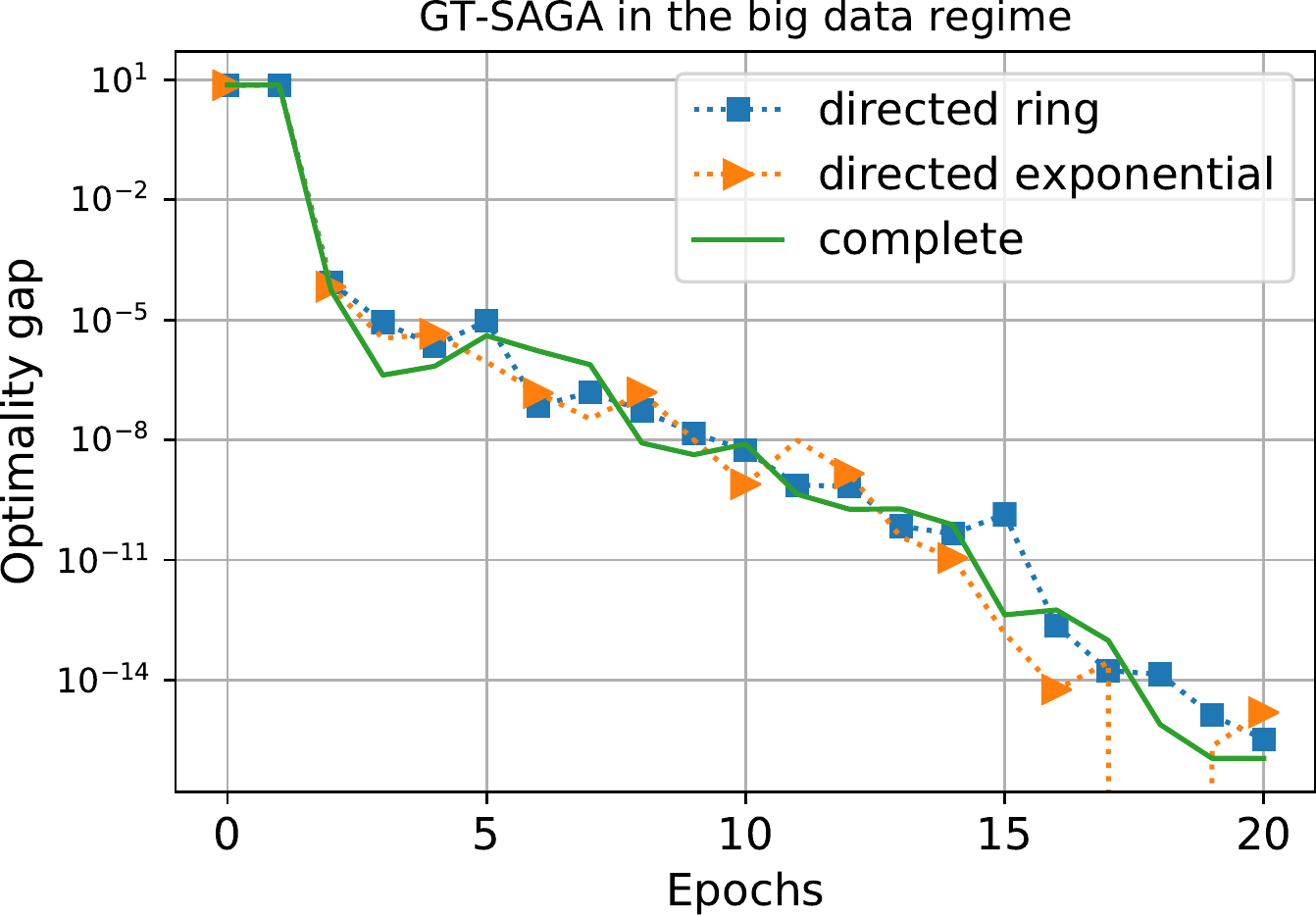}
\includegraphics[width=2.25in]{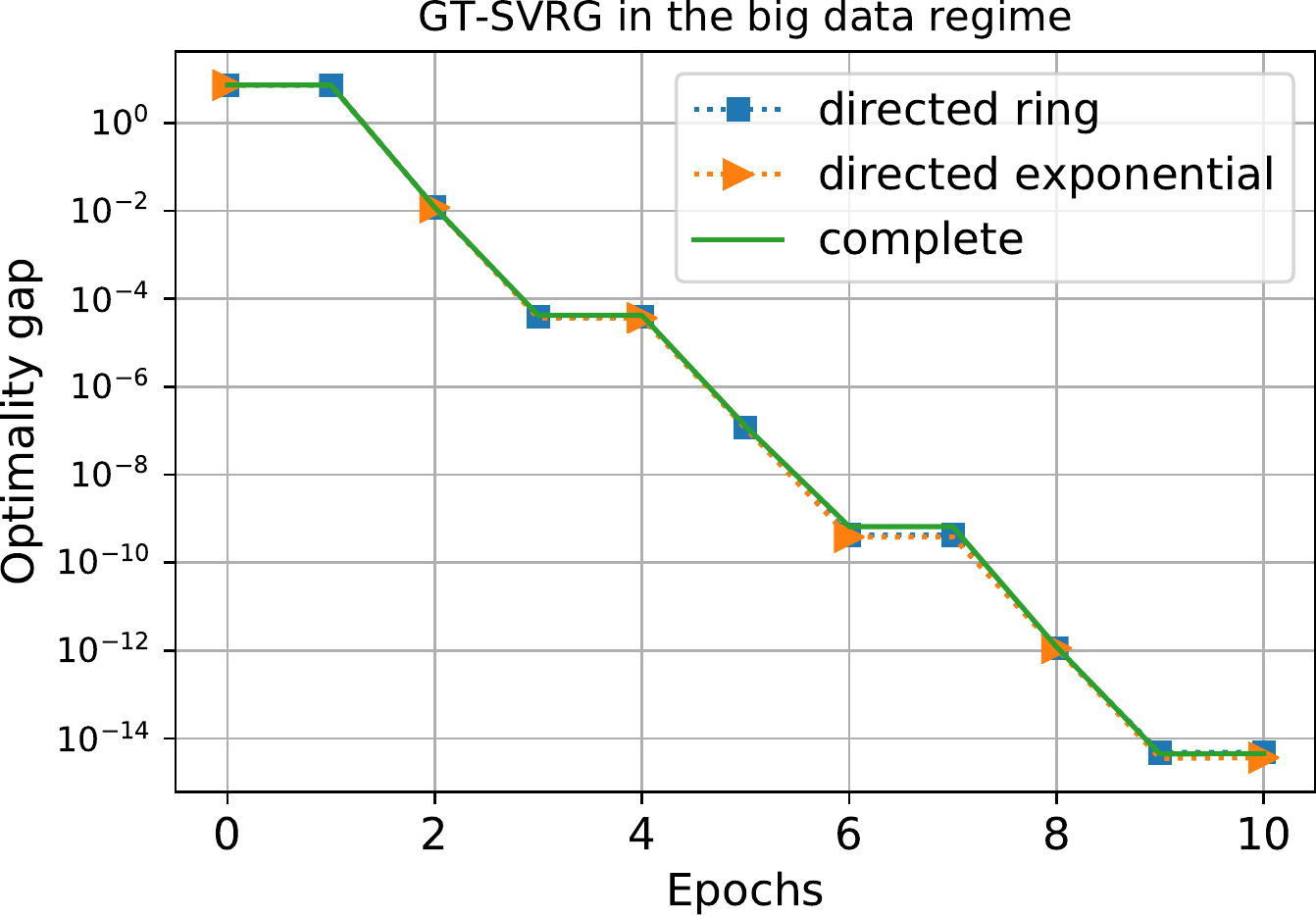}
\includegraphics[width=2.18in]{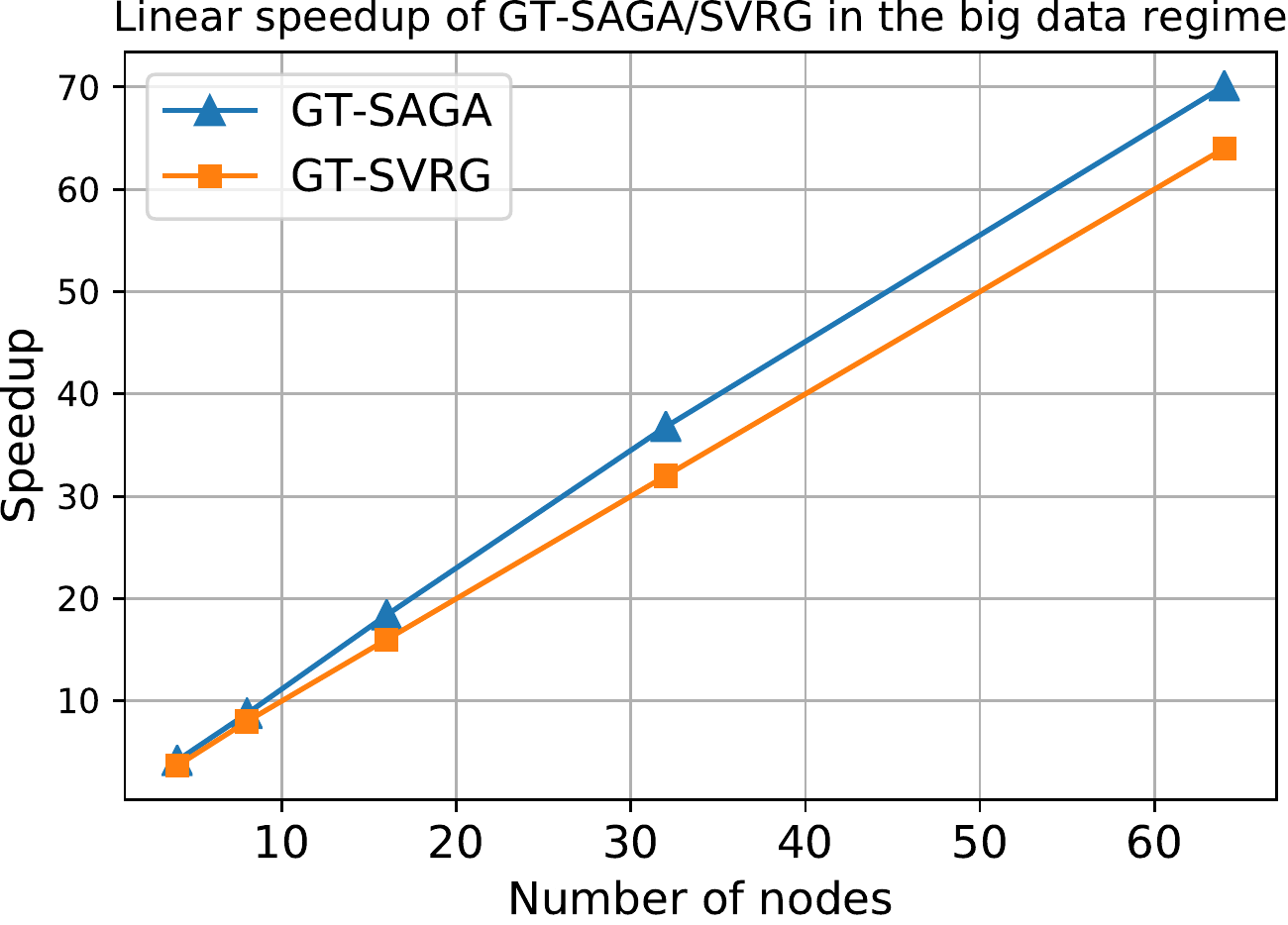}
\caption{The convergence behavior of \textbf{\texttt{GT-SAGA}} and \textbf{\texttt{GT-SVRG}} in the big data regime: (Left and Middle) Non-asymptotic, network-independent convergence; (Right) Linear speedup with respect to centralized \textbf{\texttt{SAGA}} and \textbf{\texttt{SVRG}} that process all data on a single node.}
\label{speedup}
\end{figure*}

\begin{figure*}
\centering
\includegraphics[width=2.25in]{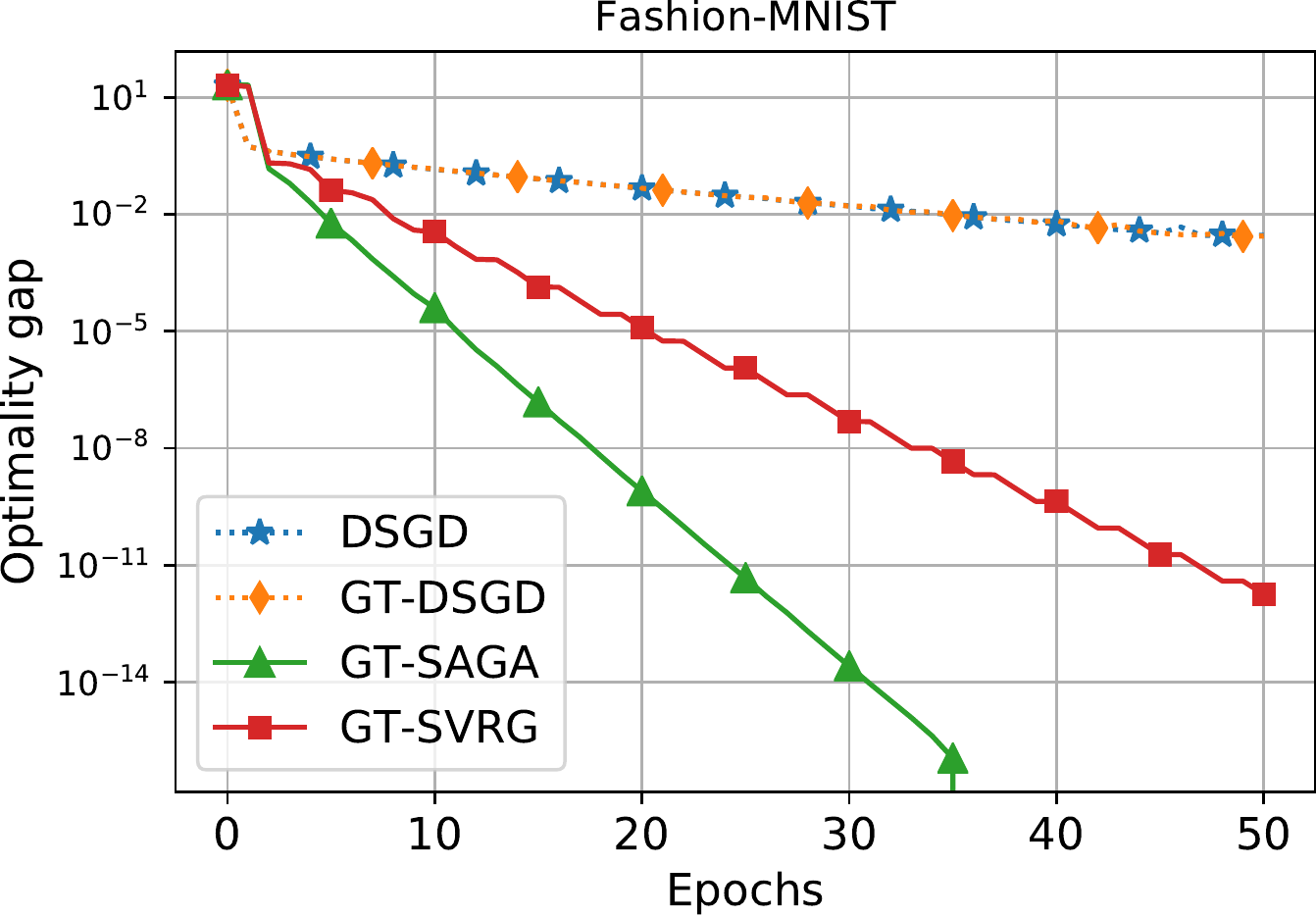}
\includegraphics[width=2.25in]{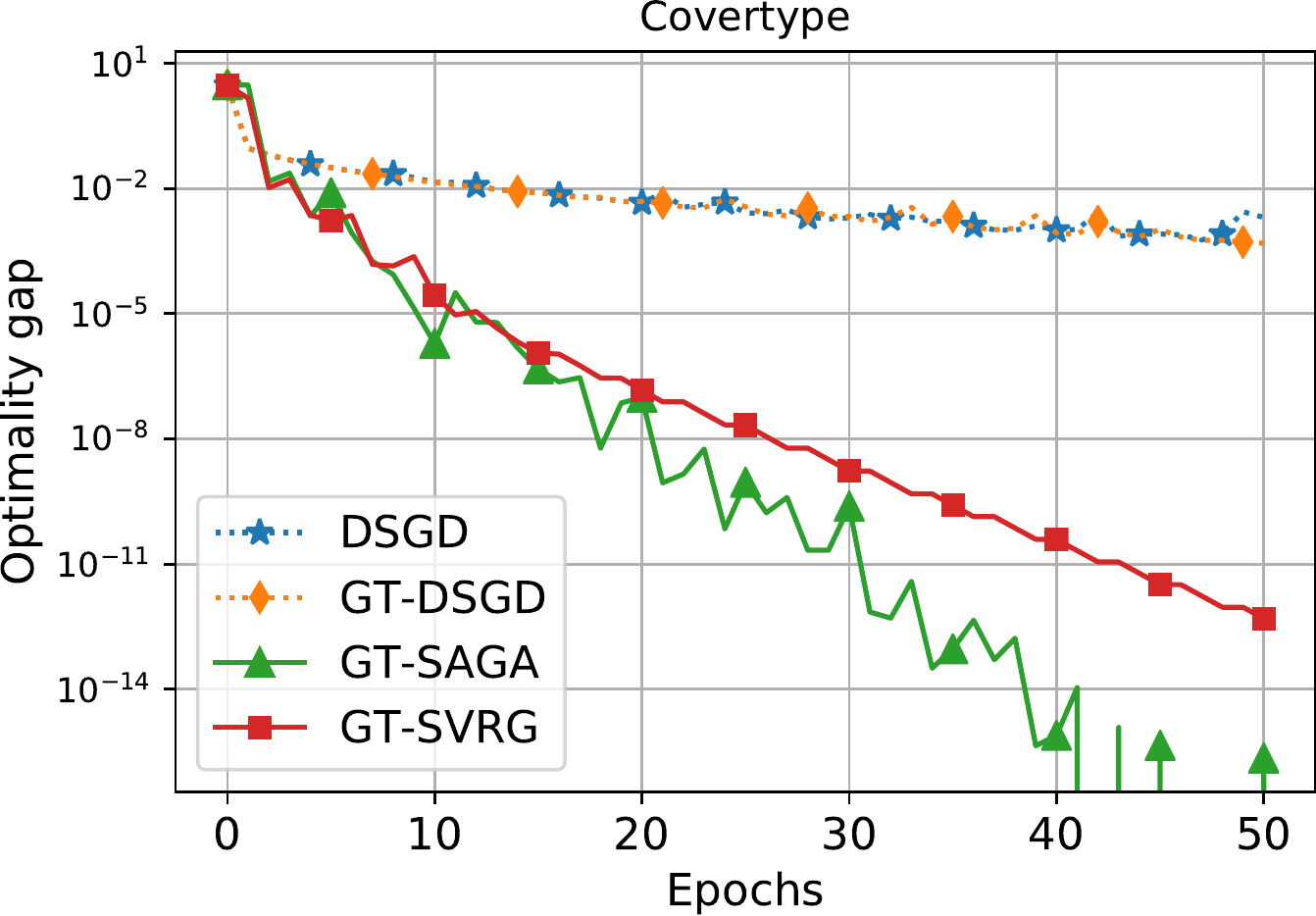}
\includegraphics[width=2.25in]{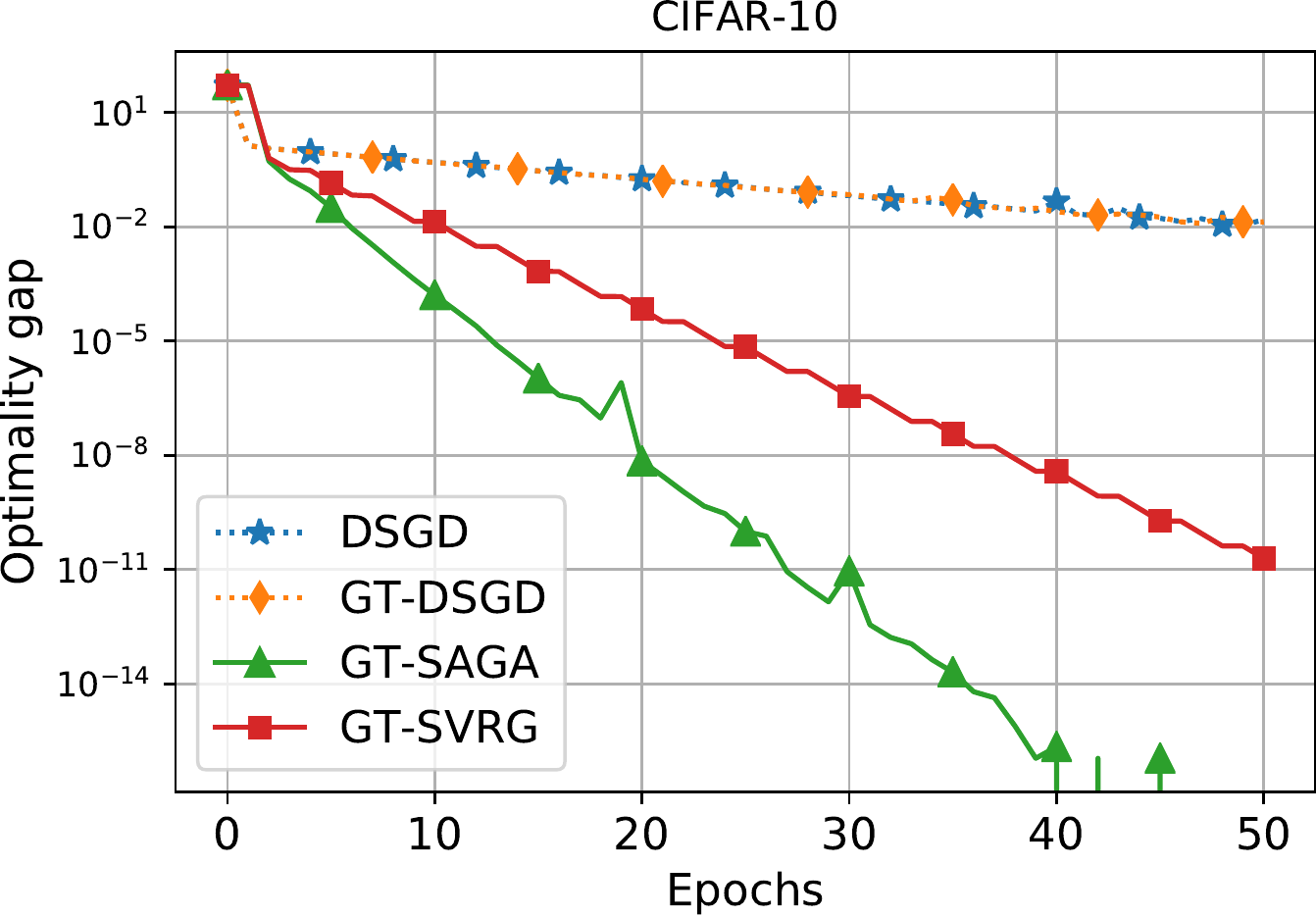}\\
\vspace{0.2cm}
\includegraphics[width=2.25in]{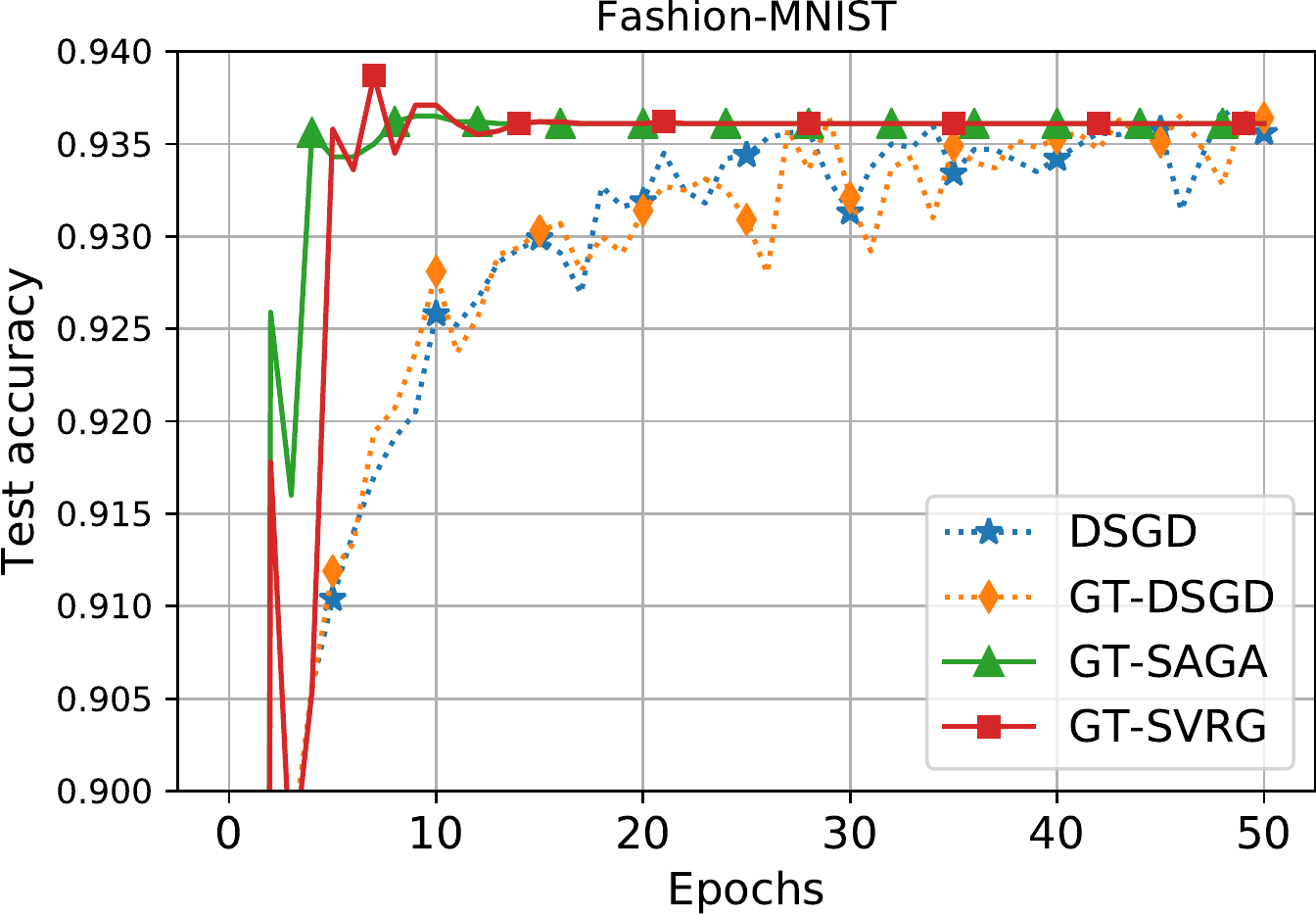}
\includegraphics[width=2.25in]{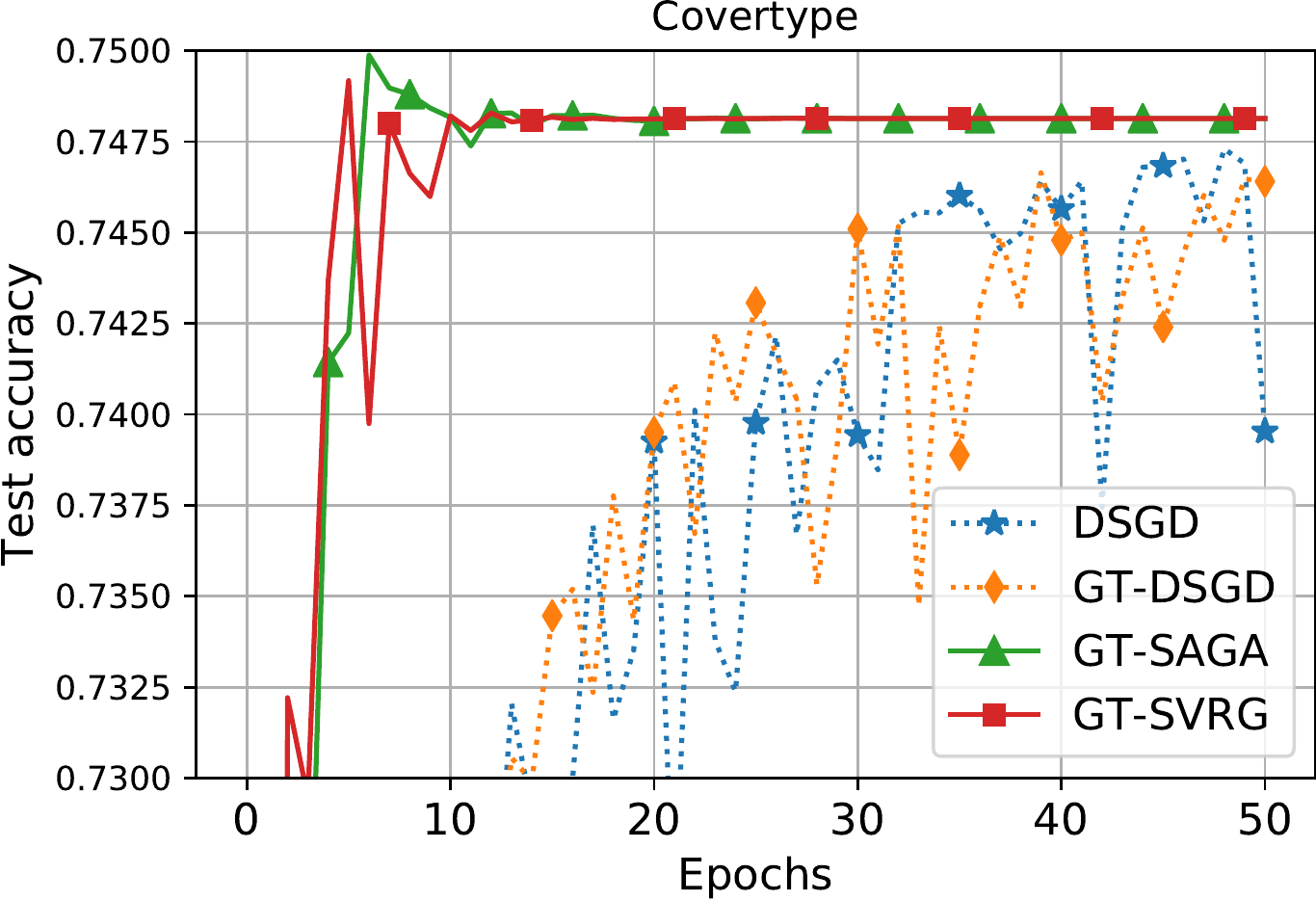}
\includegraphics[width=2.25in]{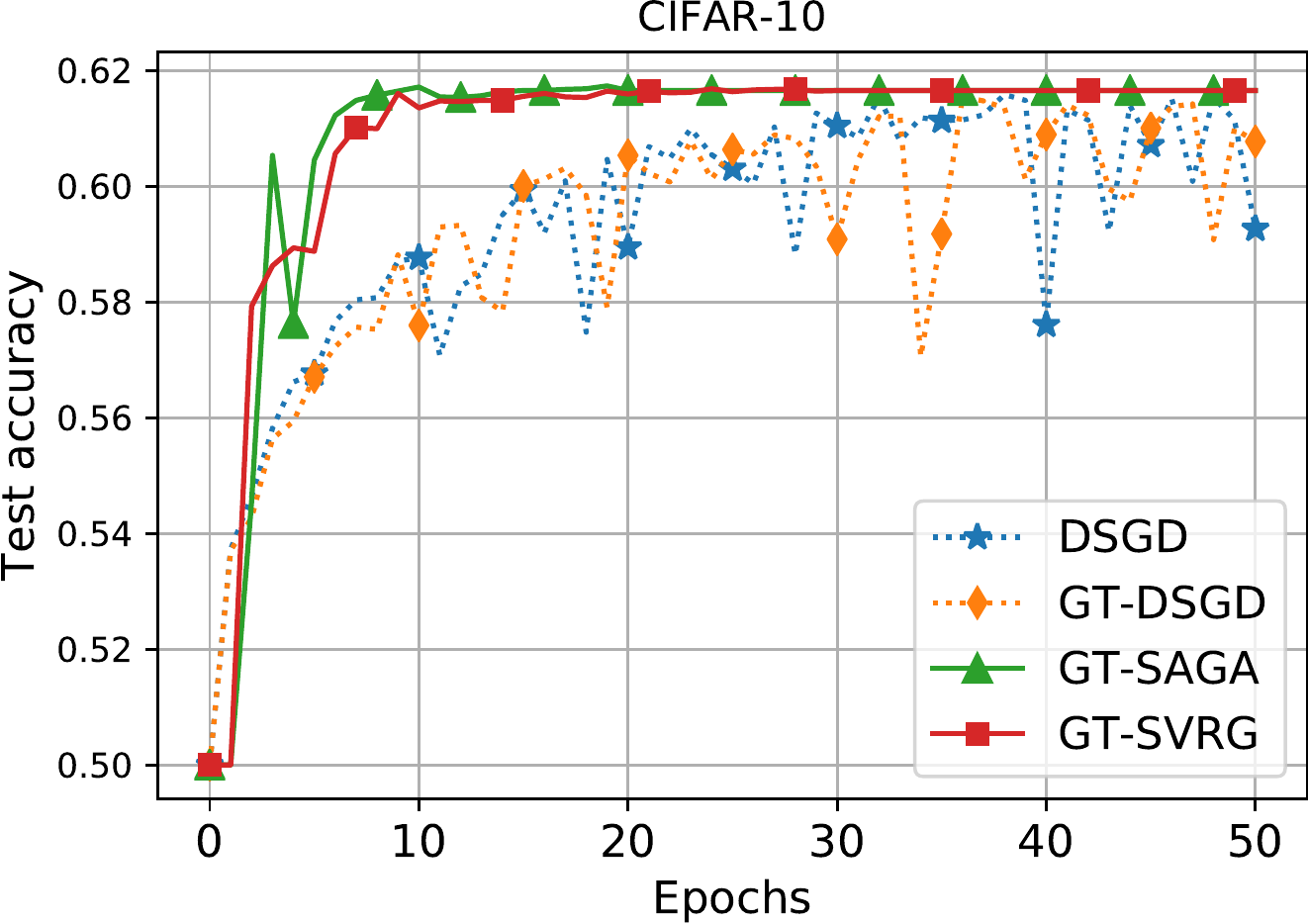}
\caption{Performance comparison of~\textbf{\texttt{GT-SAGA}} and~\textbf{\texttt{GT-SVRG}} with~\textbf{\texttt{DSGD}} and~\textbf{\texttt{GT-DSGD}} on the directed exponential graph with~${n=10}$ nodes over the Fashion-MNIST, Covertype, and CIFAR-10 datasets. The top row shows the optimality gap, while the bottom row shows the corresponding test accuracy.}
\label{dir_all_1}
\end{figure*}

\begin{figure*}
\centering
\includegraphics[width=2.25in]{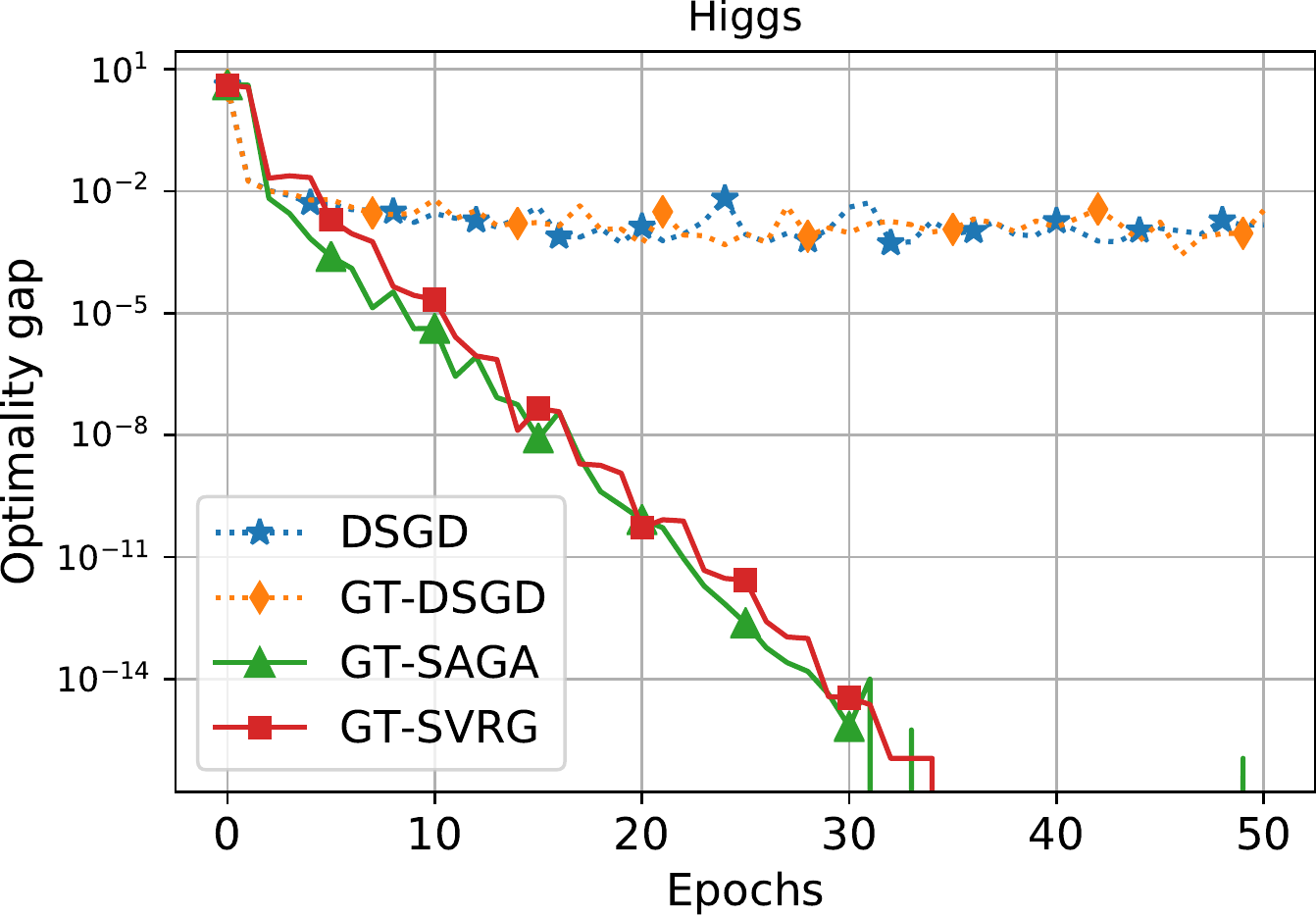}
\includegraphics[width=2.25in]{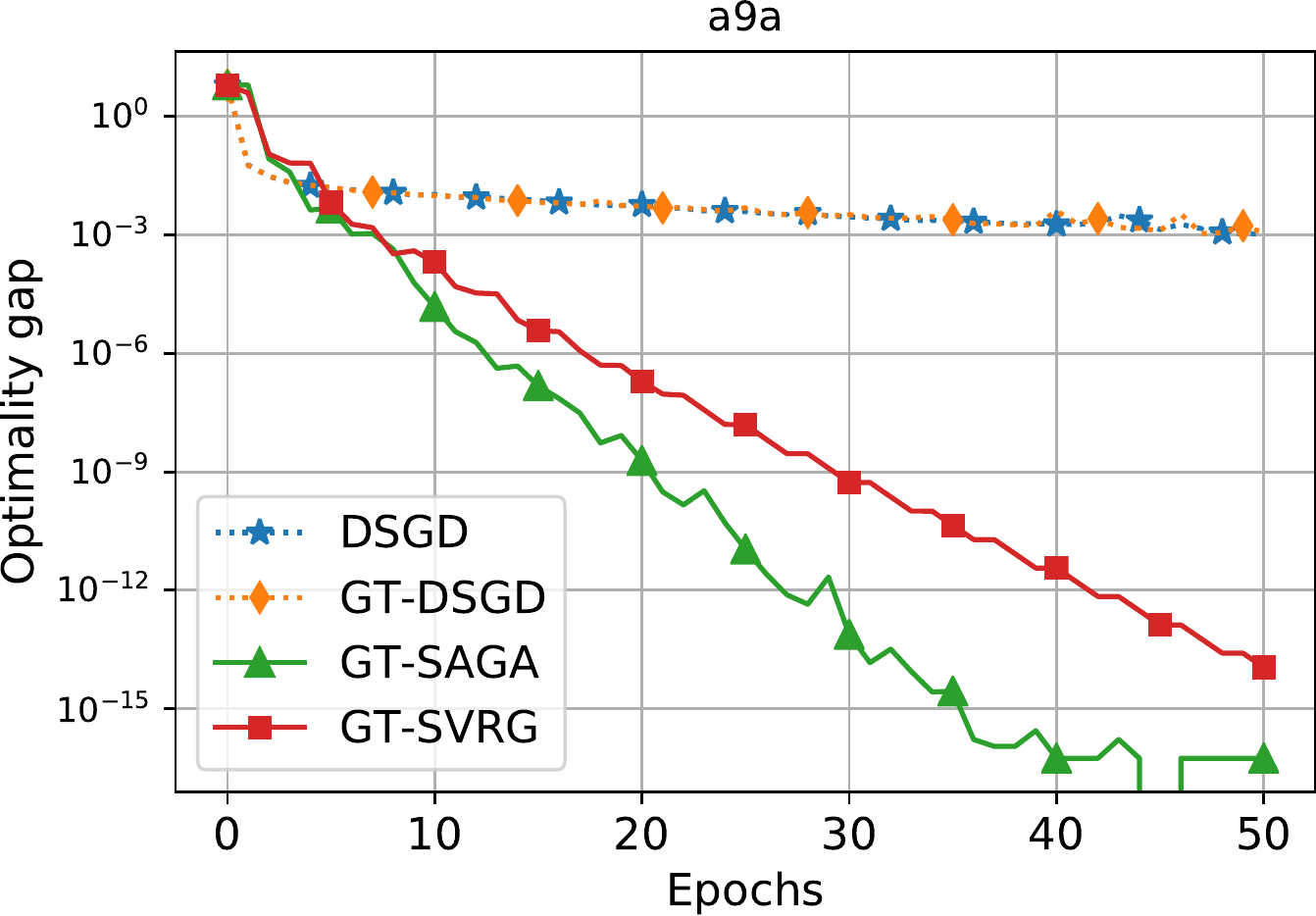}
\includegraphics[width=2.25in]{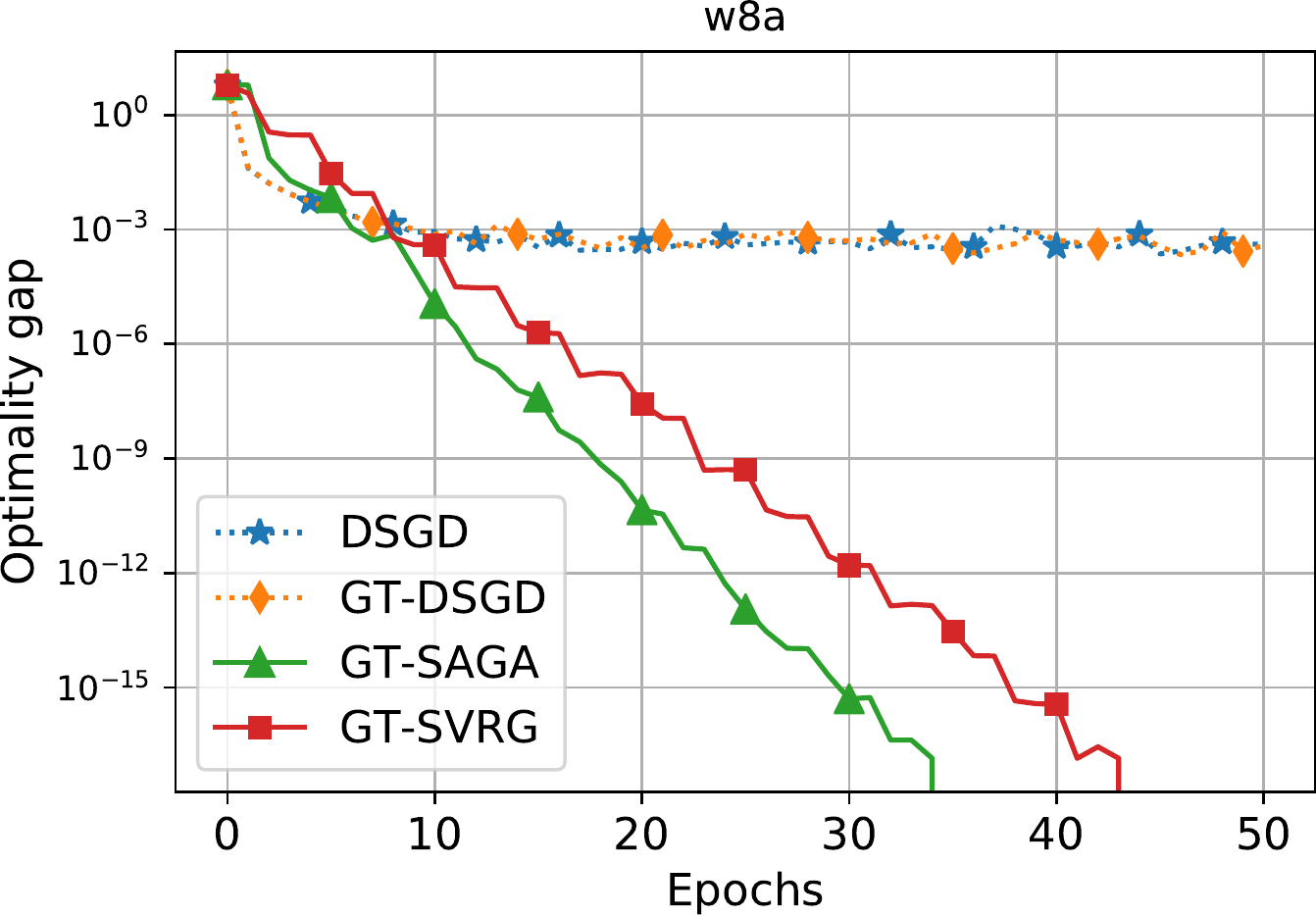}\\
\vspace{0.2cm}
\includegraphics[width=2.25in]{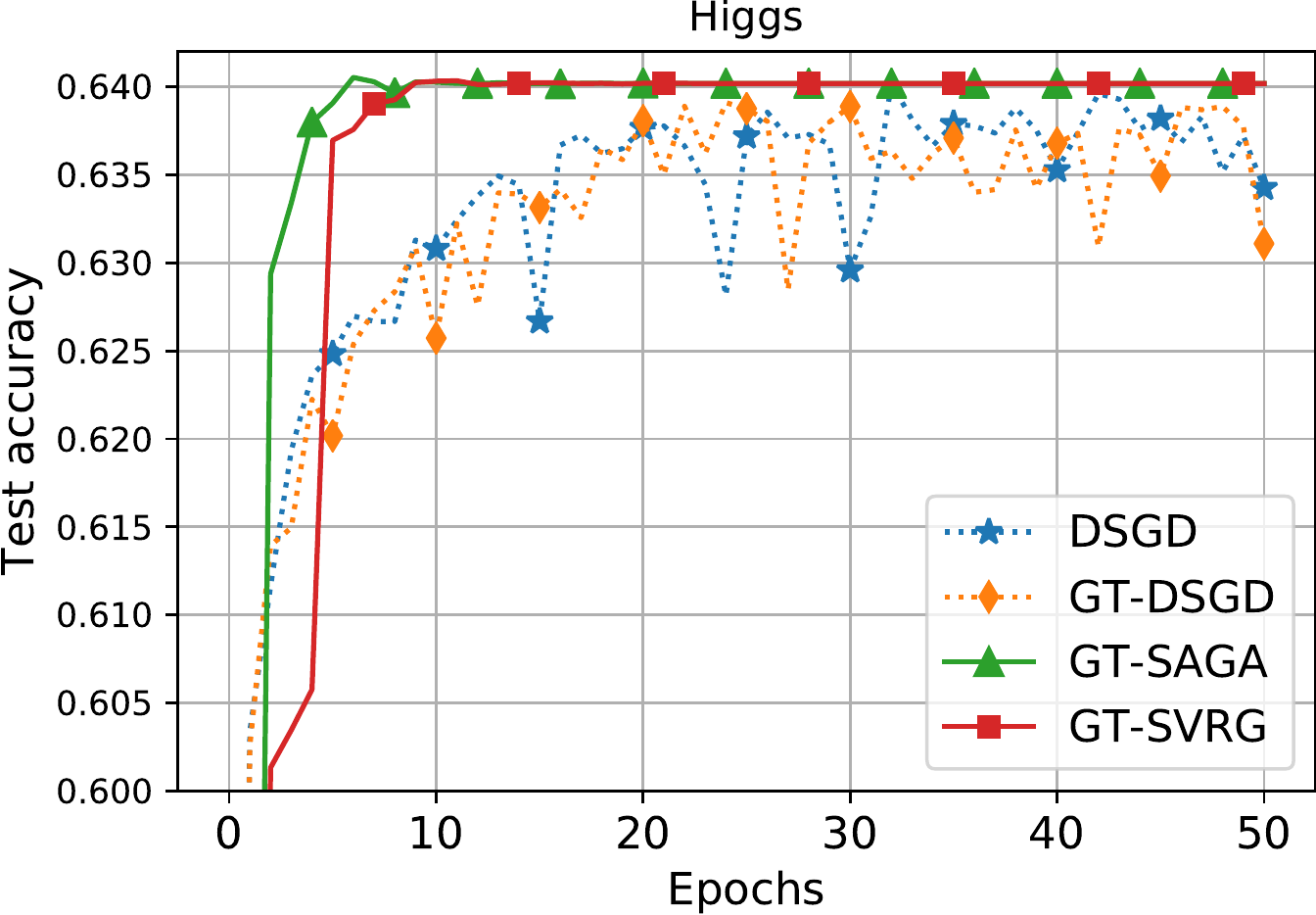}
\includegraphics[width=2.25in]{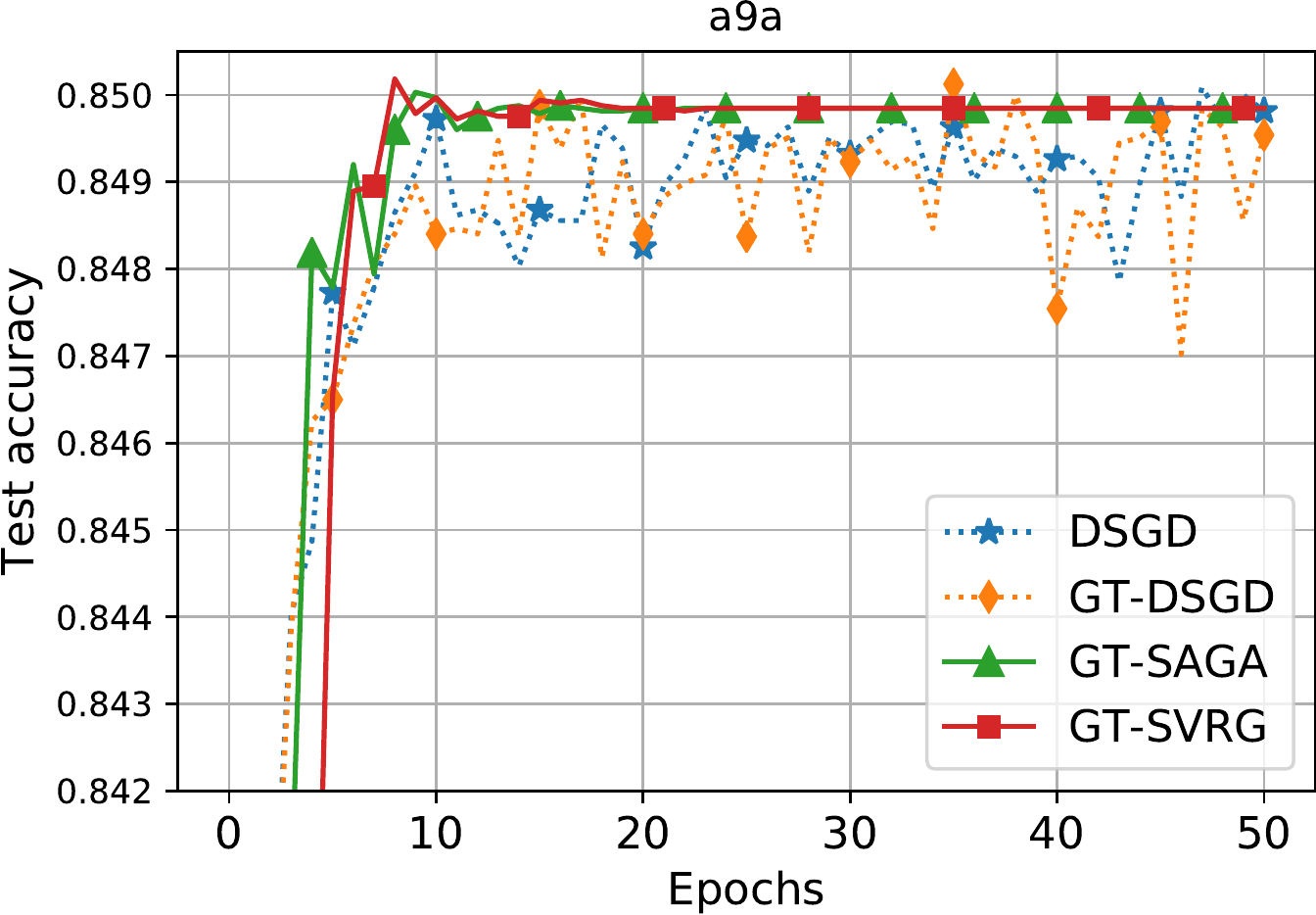}
\includegraphics[width=2.25in]{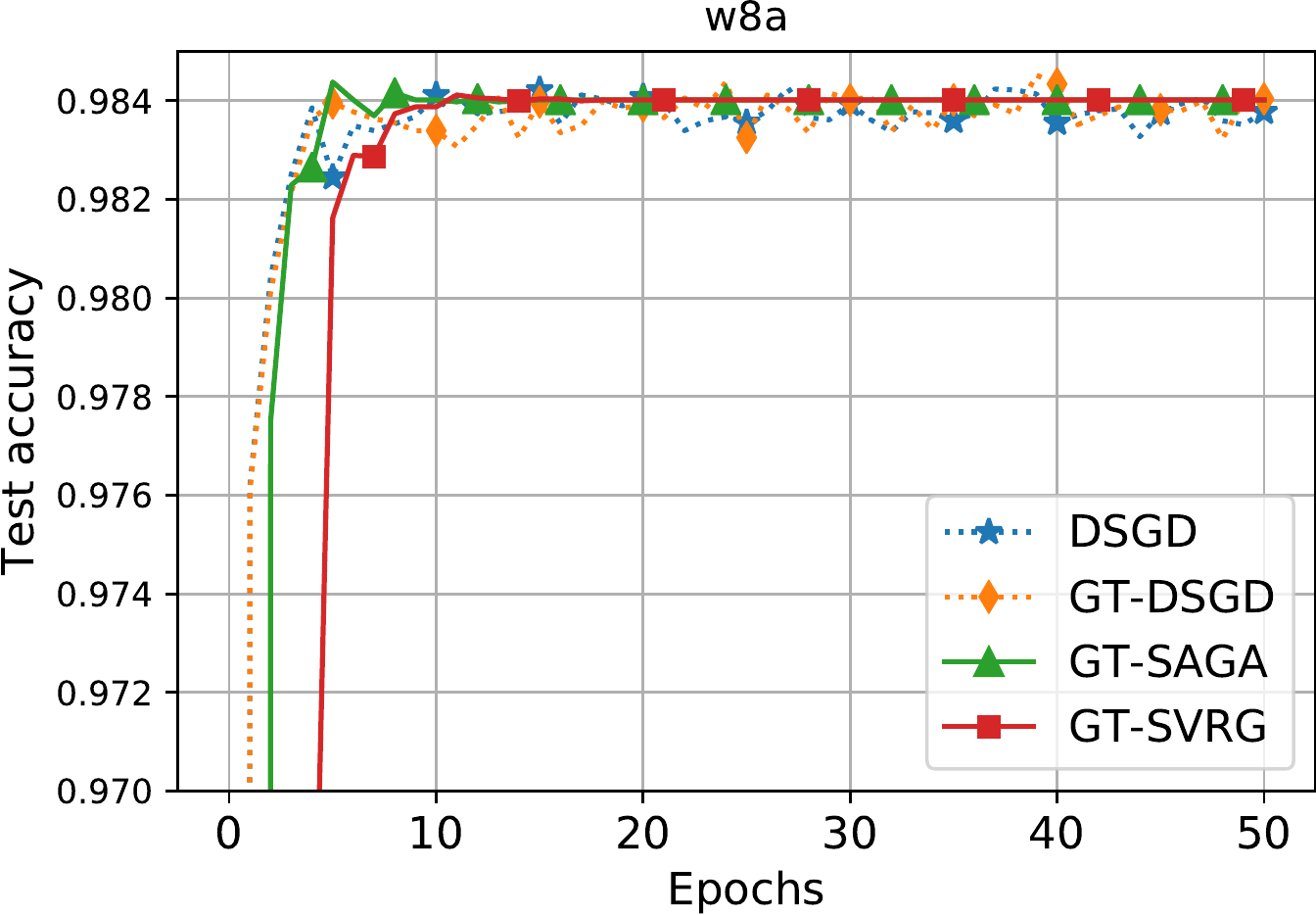}
\caption{Performance comparison of~\textbf{\texttt{GT-SAGA}} and~\textbf{\texttt{GT-SVRG}} with~\textbf{\texttt{DSGD}} and~\textbf{\texttt{GT-DSGD}} on the directed exponential graph with~${n=10}$ nodes over the Higgs, a9a, and w8a datasets. The top row presents the optimality gap, while the bottom row presents the corresponding test accuracy.}
\label{dir_all_2}
\end{figure*}

\begin{figure*}
\centering
\includegraphics[width=2.3in]{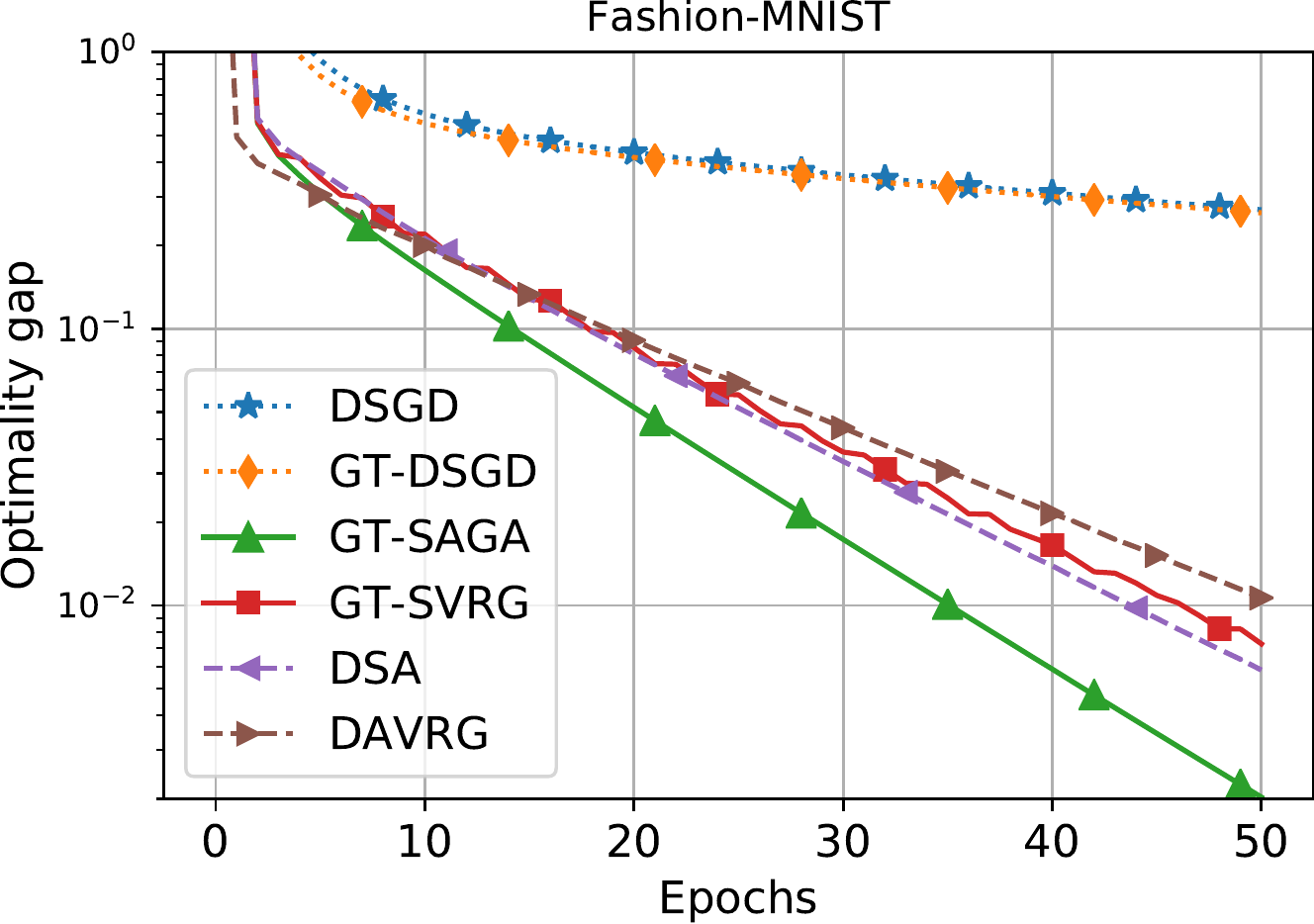}
\includegraphics[width=2.3in]{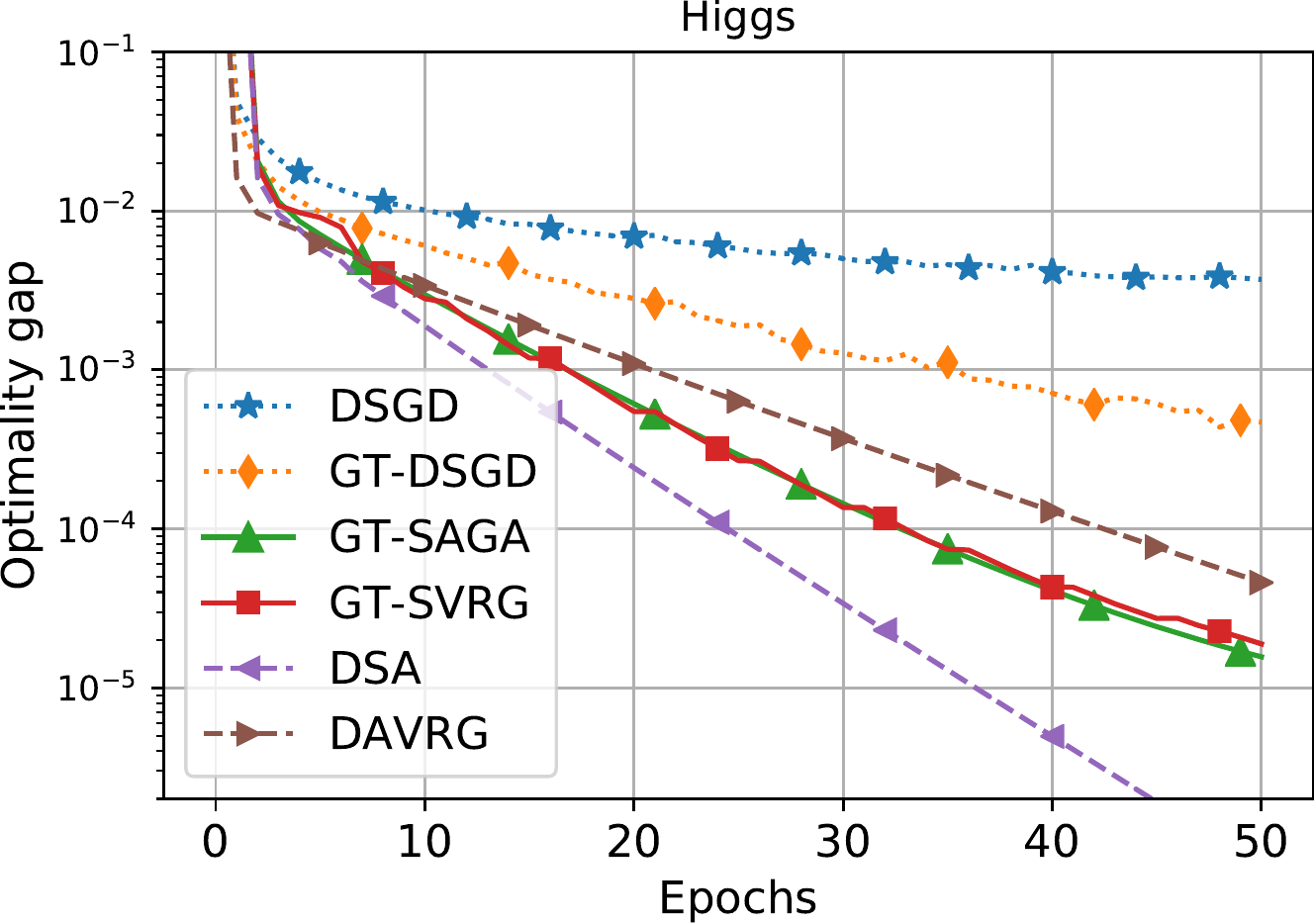}
\includegraphics[width=2.3in]{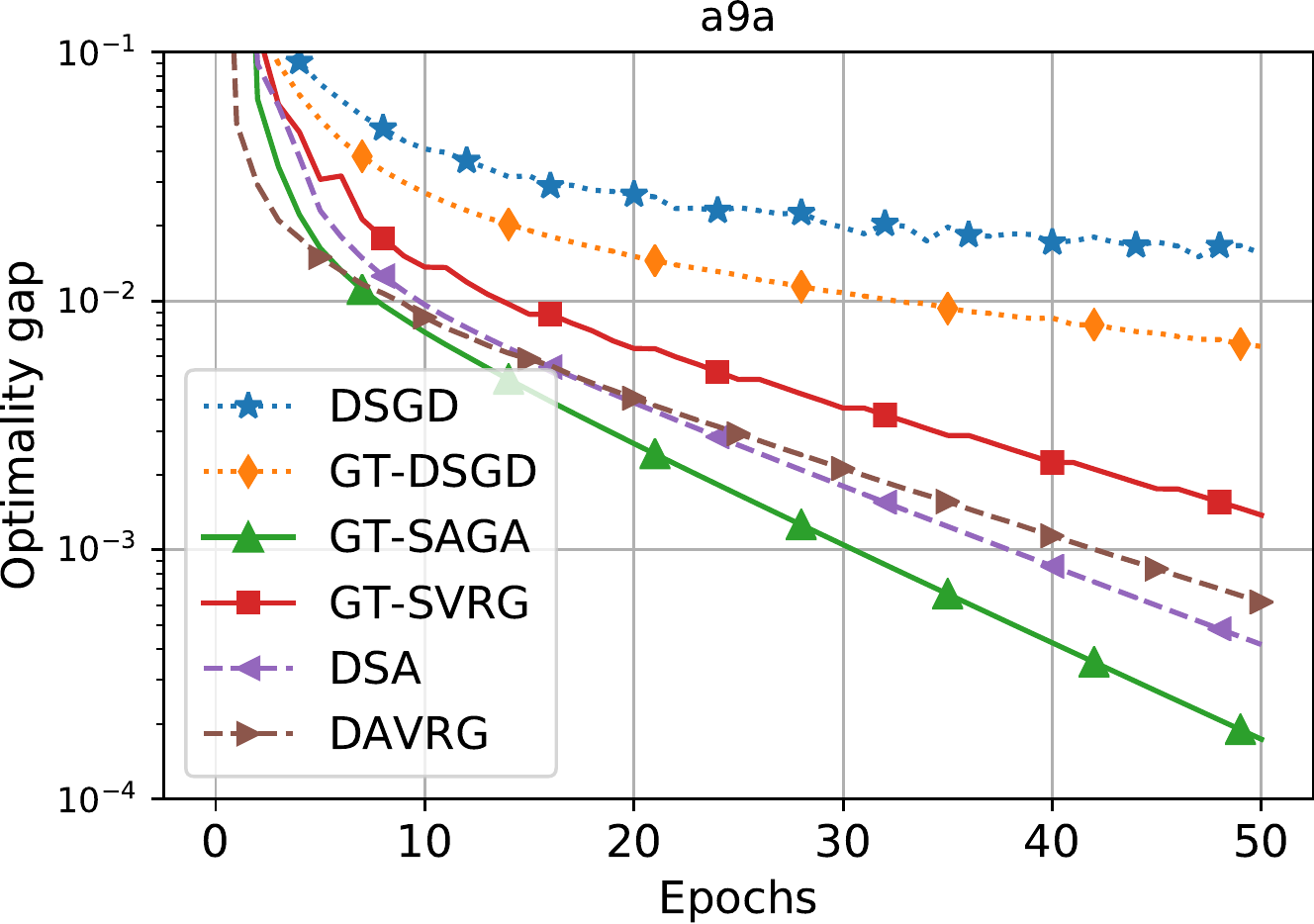}
\\
\vspace{0.2cm}
\includegraphics[width=2.3in]{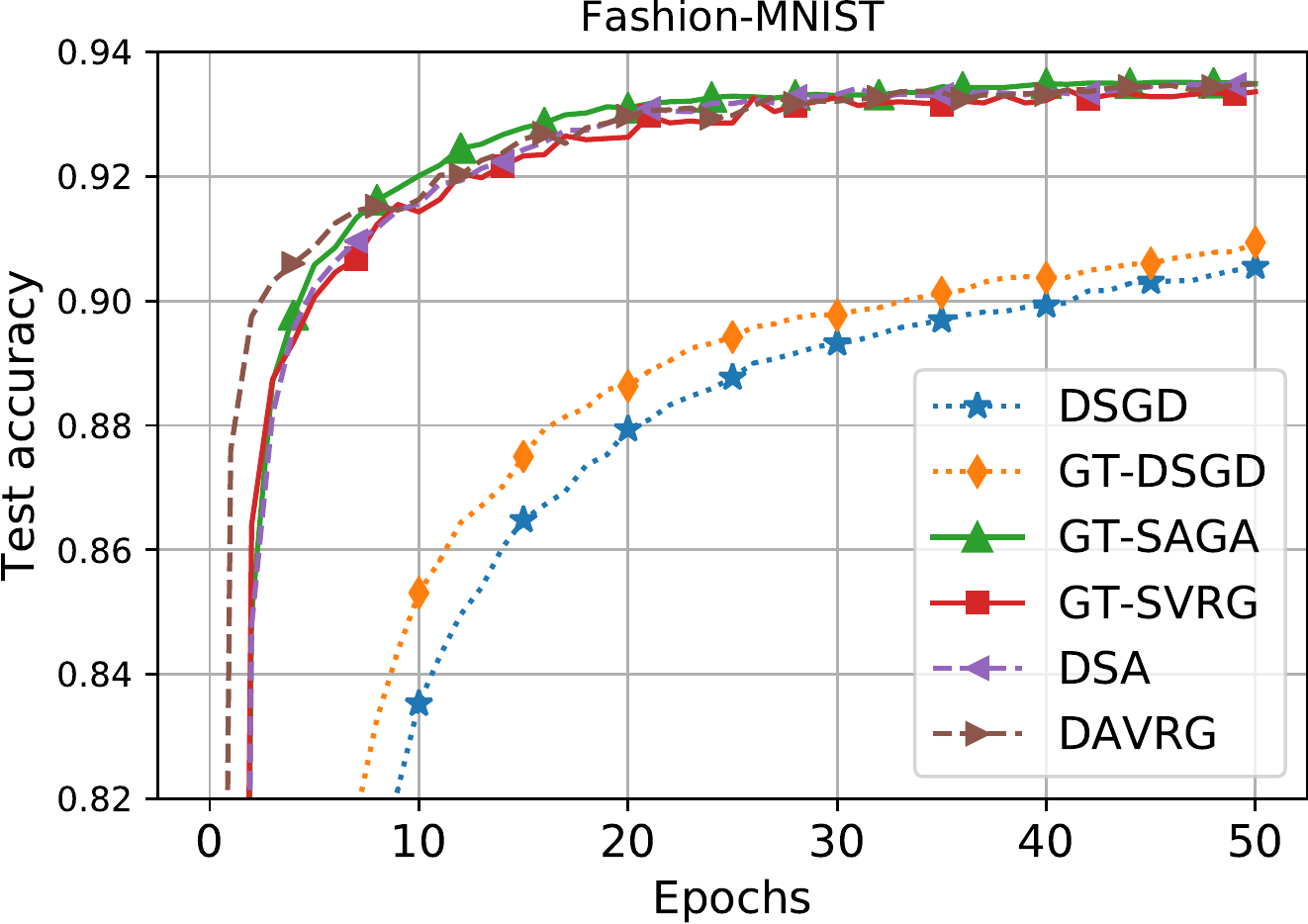}
\includegraphics[width=2.3in]{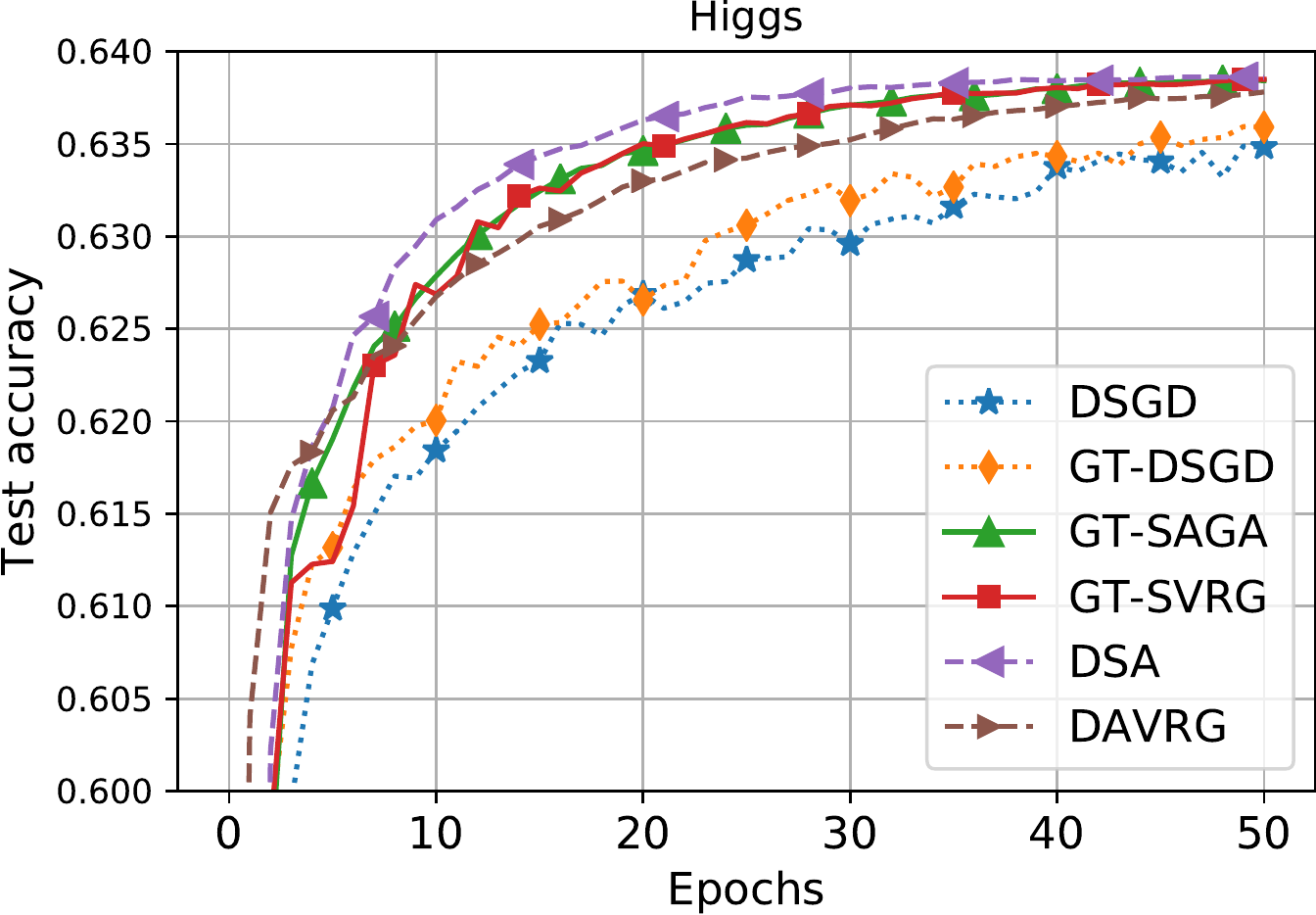}
\includegraphics[width=2.3in]{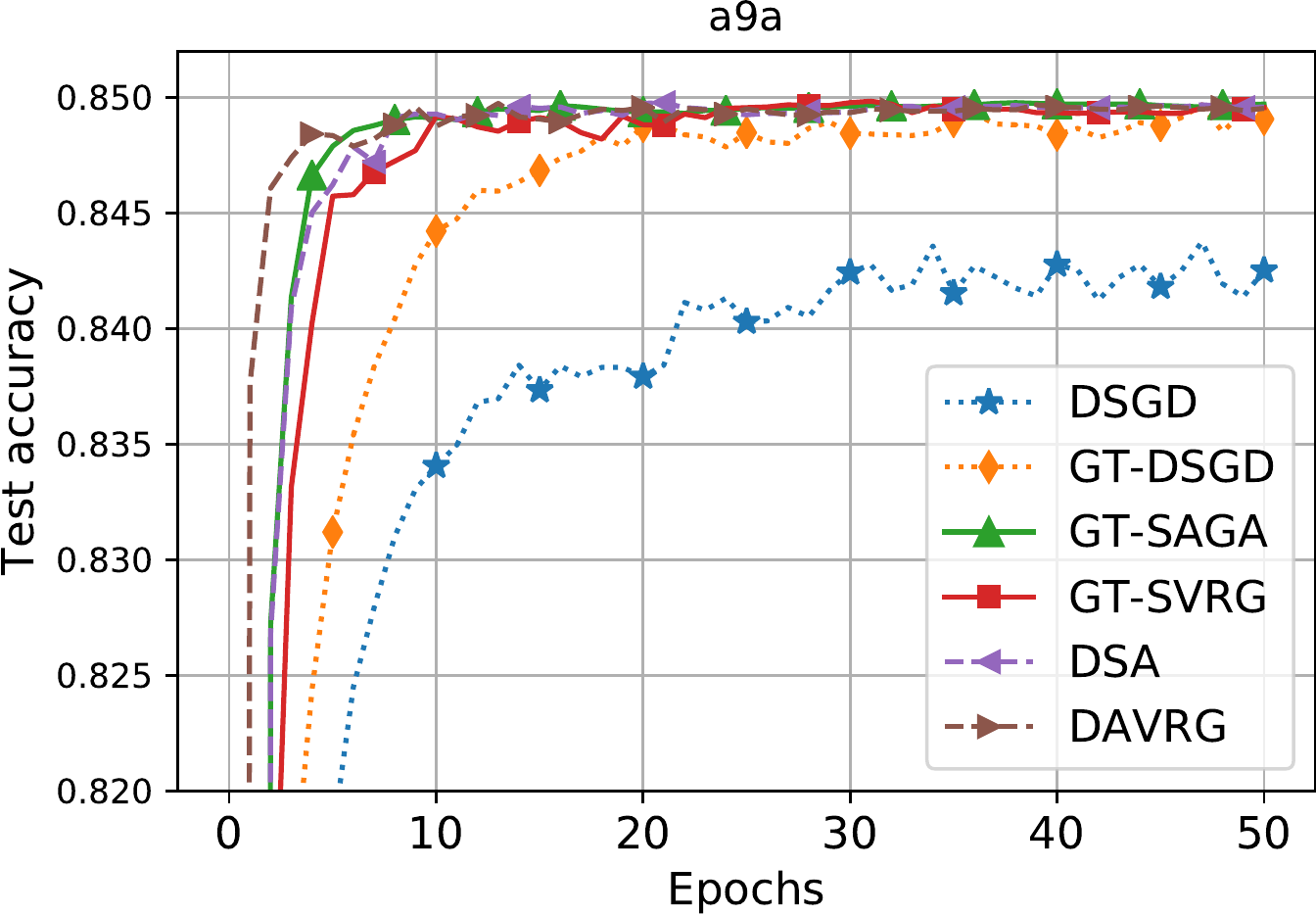}
\caption{Comparison of~\textbf{\texttt{GT-SAGA}} and~\textbf{\texttt{GT-SVRG}} with~\textbf{\texttt{DSGD}},~\textbf{\texttt{GT-DSGD}},~\textbf{\texttt{DSA}}, and~\textbf{\texttt{DAVRG}} on an undirected nearest-neighbor geometric graph with~${n=200}$ nodes over the Fashion-MNIST, Higgs, and a9a datasets. The top row shows the optimality gap, while the bottom row shows the corresponding test accuracy.}
\label{undir_all}
\end{figure*}

\subsection{Comparison with the state-of-the-art}
In this subsection, we compare the performances of the proposed \textbf{\texttt{GT-SAGA}} and \textbf{\texttt{GT-SVRG}} with the state-of-the-art decentralized stochastic first-order gradient algorithms over the datasets in Table~\ref{datasets}, i.e., DSGD, GT-DSGD, DSA, and DAVRG. We consider constant step-sizes for DSGD and GT-DSGD. Throughout this subsection, we set the regularization parameter as~$\lambda = (nm)^{-1}$ for better test accuracy~\cite{SAG,DAVRG}. 

We first consider the directed exponential graph with~${n=10}$ nodes that typically arise e.g., in data centers~\cite{SGP_ICML} where data is divided among a small number of very well-connected nodes. Note that DSA and DAVRG are not applicable to directed graphs since they require symmetric weight matrices. We thus compare the performances of \textbf{\texttt{GT-SAGA}}, \textbf{\texttt{GT-SVRG}}, DSGD and GT-DSGD, presented in Figs.~\ref{dir_all_1} and~\ref{dir_all_2}. It can be observed that the performances of DSGD and GT-DSGD are similar in this case, both of which linearly converge to a neighborhood of the optimal solution. On the other hand, \textbf{\texttt{GT-SAGA}} and \textbf{\texttt{GT-SVRG}} linearly converge to the \emph{exact} optimal solution and, moreover, achieve better test accuracy faster.

We next consider a large-scale undirected geometric graph with~${n=200}$ nodes that commonly arises e.g., in ad hoc network scenarios. The experimental result is presented in Fig.~\ref{undir_all}. We note that in this case GT-DSGD outperforms DSGD since the graph is not well-connected; this observation is consistent with~\cite{SED,SPM_XKK}. The performance of decentralized VR methods,~\textbf{\texttt{GT-SAGA}},~\textbf{\texttt{GT-SVRG}}, DSA and DAVRG are rather comparable, all of which significantly outperform DSGD and GT-DSGD in terms of both optimality gap and test accuracy. However, we note that the theoretical guarantees of DSA and DAVRG are relatively weak, compared with that of \textbf{\texttt{GT-SAGA}} and \textbf{\texttt{GT-SVRG}}.

Finally, we observe that across all experiments shown in Figs.~\ref{dir_all_1},~\ref{dir_all_2}, and~\ref{undir_all}, \textbf{\texttt{GT-SAGA}} exhibit faster convergence than \textbf{\texttt{GT-SVRG}}, at the expense of the storage cost of the gradient table at each node, demonstrating the space (storage) and time (convergence rate) tradeoffs of the \SAGA~and \SVRG~type variance reduction procedures.

\section{Convergence Analysis:\\ A General Dynamical System Approach}\label{general}
Our goal is to develop a unified analysis framework for the~\GTVR~family of algorithms. To this aim, we first present a dynamical system that unifies the~\GTVR~algorithms and develop the results that can be used in general; see~\cite{harnessing,DSGT_Pu,AB,push-pull} for similar approaches that do not involve local variance reduction schemes. Next, in Sections~\ref{sGTSAGA} and~\ref{sGTSVRG}, we specialize this dynamical system for \textbf{\texttt{GT-SAGA}} and \textbf{\texttt{GT-SVRG}} in order to formally derive the main results of Section~\ref{contribution}. 

Recall that~${\mb x_i^k\in\mbb{R}^p}$ denotes the~\GTVR~estimate of the optimal solution~$\mb x^*$ at node~$i$ and iteration~$k$, which iteratively descends in the direction of the global gradient tracker~${\mb y_i^k\in\mbb{R}^p}$. Concatenating~$\mb x_i^k$'s and~$\mb y_i^k$'s in column vectors~$\mb x^k,\mb y^k$, both in~$\mbb R^{pn}$, and defining~${W \coloneqq \ul{W}\otimes I_p}$, we can write the estimate update of~\GTVR~as
\begin{align}\label{alg_a}
\mb{x}^{k+1} &= W\mb{x}^{k} - \alpha\mb{y}^{k}, 
\end{align}
which is applicable to both~\GTSAGA~and~\GTSVRG. The gradient tracking step next is given by
\begin{align}
\mb{y}^{k+1} &= W\mb{y}^{k} +
\mb{r}^{k+1}-\mb{r}^{k}, \label{alg_b}
\end{align}
where~${\mb r^k\in\mbb R^{pn}}$ concatenates local variance-reduced gradient estimators~$\mb r_i^k$'s, all in~$\mathbb{R}^p$, which are given by~$\mb g_i^k$'s in \textbf{\texttt{GT-SAGA}} and by~$\mb v_i^k$'s in \textbf{\texttt{GT-SVRG}}. For the initial conditions, we have~${\mb{y}^0=\mb{r}^0}\in\mbb{R}^p$ and~$\mb x^0\in\mbb{R}^p$ is arbitrary. 

Clearly, \eqref{alg_a}-\eqref{alg_b} are applicable to the~\GTVR~framework in general and the specialized algorithm of interest from this family can be obtained by using the corresponding variance-reduced estimator. We therefore first analyze the dynamical system~\eqref{alg_a}-\eqref{alg_b}, on top of which the specialized results for \textbf{\texttt{GT-SAGA}} and \textbf{\texttt{GT-SVRG}} are derived subsequently. 




\vspace{-0.35cm}
\subsection{Preliminaries}
To proceed, we define several auxiliary variables that will aid the subsequent convergence analysis as follows.
\begin{align}
&\ol{\mb{x}}^{k} \coloneqq  \frac{1}{n}\left(\mb{1}_n^\top\otimes I_p\right)\mb{x}^{k},\qquad\ol{\mb{y}}^{k} \coloneqq  \frac{1}{n}\left(\mb{1}_n^\top\otimes I_p\right)\mb{y}^{k},\qquad \nonumber\\
&\ol{\mb{r}}^{k} \coloneqq  \frac{1}{n}\left(\mb{1}_n^\top\otimes I_p\right)\mb{r}^{k},
\nonumber\\
&\nabla\mb{f}(\mb{x}^{k})\coloneqq [\nabla f_1(\mb{x}_1^{k})^\top,\dots,\nabla f_n(\mb{x}_n^{k})^\top]^\top,\qquad \nonumber\\
&\ol{\nabla\mb{f}}(\mb{x}^{k})\coloneqq \frac{1}{n}\left(\mb{1}_n^\top\otimes I_p\right)\nabla\mb{f}(\mb{x}^{k}).
\nonumber 
\end{align}
We recall that~\eqref{alg_b} is a stochastic gradient tracking method~\cite{DSGT_Pu,DSGT_Xin,SGT_nonconvex_you} as an application of dynamic consensus~\cite{DAC}. It is straightforward to verify by induction that~\cite{DAC}:
$$\ol{\mb{r}}^k = \ol{\mb{y}}^k,\qquad\forall k\geq0.$$
Clearly, the randomness of both~\textbf{\texttt{GT-SAGA}} and~\textbf{\texttt{GT-SVRG}} lies in the set of independent random variables~$\{s_i^k\}_{i=\{1,\cdots,n\}}^{k\geq1}$. 
We denote~$\mc{F}^k$ as the history of the dynamical system generated by~$\{s_i^t\}^{t\leq k-1}_{i=\{1,\cdots,n\}}$. For both \textbf{\texttt{GT-SAGA}} and \textbf{\texttt{GT-SVRG}},~$\mb{r}_i^k$ is an unbiased estimator of~$\nabla f_i(\mb{x}_i^k)$ given~$\mc{F}^k$~\cite{SAGA,SVRG}, i.e.,
\begin{align*}
\mathbb{E}\left[\mb{r}^k | \mc{F}^k\right]
= \nabla\mb{f}(\mb{x}^k), \quad
\mathbb{E}\left[\ol{\mb{y}}^k | \mc{F}^k\right]
= \mathbb{E}\left[\ol{\mb{r}}^k | \mc{F}^k\right]
= \ol{\nabla\mb{f}}(\mb{x}^k).
\end{align*}  
In the following, we first present a few well-known results related to decentralized gradient tracking methods whose proofs can be found in, e.g.,~\cite{harnessing,DIGing,AB,push-pull}. 

\begin{lem}\label{descent}
Let~Assumptions~\ref{sc} and~\ref{smooth} hold. If~$0<\alpha\leq\frac{1}{L}$, we have~$\left\|\mb{x}-\alpha\nabla F(\mb{x}) -\mb{x}^*\right\|
\leq(1-\mu\alpha)\left\|\mb{x}-\mb{x}^*\right\|$,~$\forall\mb{x}\in\mathbb{R}^p$.
\end{lem}

\begin{lem}\label{L-bound}
Let Assumption~\ref{smooth} hold. Consider the iterates $\{\mb{x}^k\}$ generated by the dynamical system~\eqref{alg_a}-\eqref{alg_b}. We have that
$\left\|\ol{\nabla\mb{f}}(\mb{x}^{k})-\nabla F(\ol{\mb{x}}^{k})\right\|\leq\frac{L}{\sqrt{n}}\left\|\mb{x}^{k}-W_\infty\mb{x}^{k}\right\|,\forall k\geq 0$.
\end{lem}

\begin{lem}\label{W_contract}
Let Assumption~\ref{connect} hold. We have that~$\forall \mb{x}\in\mathbb{R}^{np}$,
$\left\|W\mb{x} - W_\infty\mb{x}\right\|\leq\sigma\left\|\mb{x} - W_\infty\mb{x}\right\|$, where~$W_\infty=\frac{\mb{1}_n\mb{1}_n^\top}{n}\otimes I_p$.
\end{lem}

\vspace{-0.35cm}
\subsection{Auxiliary Results}
\vspace{-0.15cm}
In this subsection, we analyze the general dynamical system~\eqref{alg_a}-\eqref{alg_b} by establishing the interrelationships between the mean-squared consensus error~$\mathbb{E}\left[\|\mb{x}^{k}-W_\infty\mb{x}^{k}\|^2\right]$,~network optimality gap $\mathbb{E}\left[\|\ol{\mb{x}}^{k}-\mb{x}^*\|^2\right]$ and gradient tracking error~$\mathbb{E}\left[\|\mb{y}^{k}-W_\infty\mb{y}^{k}\|^2\right]$.
\begin{lem}\label{consensus}
Let Assumption~\ref{connect} hold. Consider the iterates $\{\mb{x}^k\}$ generated by~\eqref{alg_a}-\eqref{alg_b}. We have the following hold:~$\forall k\geq0$,
{\small\begin{align}
\mathbb{E}\left[\left\|\mb{x}^{k+1}-W_\infty\mb{x}^{k+1}\right\|^2\right]\leq&~
	\frac{1+\sigma^2}{2}\mathbb{E}\left[\left\|\mb{x}^k-W_\infty\mb{x}^k\right\|^2\right] \nonumber\\
	&+ \frac{2\alpha^2}{1-\sigma^2}\mathbb{E}\left[\left\|\mb{y}^k-W_\infty\mb{y}^k\right\|^2\right]. \label{consensus1}\\
\mathbb{E}\left[\left\|\mb{x}^{k+1}-W_\infty\mb{x}^{k+1}\right\|^2\right]\leq&~
	2\mathbb{E}\left[\left\|\mb{x}^k-W_\infty\mb{x}^k\right\|^2\right] \nonumber\\
	&+ 2\alpha^2\mathbb{E}\left[\left\|\mb{y}^k-W_\infty\mb{y}^k\right\|^2\right] \label{consensus2}.
\end{align}}
\end{lem}
\begin{proof}
Using~\eqref{alg_a} and the fact that~$W_\infty W = W_\infty$, we have:
\begin{align}
&\left\|\mb{x}^{k+1}-W_\infty\mb{x}^{k+1}\right\|^2 \nonumber\\
=&~ \left\|W\mb{x}^k - W_\infty\mb{x}^k-\alpha\left(\mb{y}^k-W_\infty\mb{y}^k\right)\right\|^2
\label{consensus0}
\end{align}
Next, we use Young's inequality that~$\|\mb{a}+\mb{b}\|^2\leq(1+\eta)\|\mb{a}\|^2+(1+\frac{1}{\eta})\|\mb{b}\|^2,\forall\mb{a},\mb{b}\in\mathbb{R}^{np},\forall\eta>0$, and Lemma~\ref{W_contract} in~\eqref{consensus0} to obtain:~$\forall k\geq0$,
\begin{align}
\left\|\mb{x}^{k+1}-W_\infty\mb{x}^{k+1}\right\|^2
\leq&~\left(1+\eta\right)\sigma^2\left\|\mb{x}^{k}-W_\infty\mb{x}^{k}\right\|^2\nonumber\\
&~+ \left(1+\eta^{-1}\right)\alpha^2\left\|\mb{y}^k-W_\infty\mb{y}^k\right\|^2\nonumber
\end{align}
Setting~$\eta$ as $\frac{1-\sigma^2}{2\sigma^2}$ and~$1$ in the above inequality respectively leads to~\eqref{consensus1} and~\eqref{consensus2}. 
\end{proof}
Next, we establish an inequality for~$\mathbb{E}\left[\left\|\ol{\mb x}^{k+1}-\mb{x}^*\right\|^2\right]$. 
\begin{lem}
Let Assumptions~\ref{sc},~\ref{smooth} and~\ref{connect} hold. Consider the iterates $\{\mb{x}^k\}$ generated by~\eqref{alg_a}-\eqref{alg_b}. If~$0<\alpha\leq\frac{1}{L}$, we have the following inequalities hold:~$\forall k\geq0$,
\begin{align}
\mathbb{E}\left[n\left\|\ol{\mb x}^{k+1}-\mb{x}^*\right\|^2\right] 
\leq&~ \frac{L^2\alpha}{\mu}\mathbb{E}\left[\left\|\mb{x}^k-W_\infty\mb{x}^k\right\|^2\right] \nonumber\\
&+ (1-\mu\alpha)\mathbb{E}\left[n\|\ol{\mb{x}}^k-\mb{x}^*\|^2\right]  \nonumber\\
&+ \frac{\alpha^2}{n}\mathbb{E}\left[\left\|\mb{r}^k-\nabla\mb{f}(\mb{x}^k)\right\|^2\right].
\label{opt_1} \\
\mathbb{E}\left[n\left\|\ol{\mb x}^{k+1}-\mb{x}^*\right\|^2\right] 
\leq&~ 2L^2\alpha^2\mathbb{E}\left[\left\|\mb{x}^k-W_\infty\mb{x}^k\right\|^2\right] \nonumber\\
&+ 2\mathbb{E}\left[n\|\ol{\mb{x}}^k-\mb{x}^*\|^2\right]  \nonumber\\
&+ \frac{\alpha^2}{n}\mathbb{E}\left[\left\|\mb{r}^k-\nabla\mb{f}(\mb{x}^k)\right\|^2\right].
\label{opt_2}
\end{align}
\end{lem}
\begin{proof}
Multiplying~$\frac{\mb{1}_n^\top\otimes I_p}{n}$ to~\eqref{alg_a}, we have that~$\forall k\geq0$, $$\ol{\mb x}^{k+1}=\ol{\mb x}^{k}-\alpha\ol{\mb y}^{k}=\ol{\mb x}^{k}-\alpha\ol{\mb r}^{k}.$$ We expand~$\mathbb{E}\left[\|\ol{\mb{x}}^{k+1}-\mb{x}^*\|^2|\mc{F}^k\right]$ as follows. 
\begin{align}
&\mathbb{E}\left[\left\|\ol{\mb x}^{k+1}-\mb{x}^*\right\|^2|\mc{F}^k\right] \nonumber\\
=&~\mathbb{E}\left[\left\|\ol{\mb x}^{k}-\alpha\ol{\mb r}^{k}-\mb{x}^*\right\|^2|\mc{F}^k\right]\nonumber\\
=&~\mathbb{E}\left[\left\|\ol{\mb x}^{k}-\alpha\nabla F(\ol{\mb x}^{k})-\mb{x}^*+\alpha\left(\nabla F(\ol{\mb x}^{k})-\ol{\mb r}^{k}\right)\right\|^2|\mc{F}^k\right] \nonumber\\
=& \left\|\ol{\mb x}^{k}-\alpha\nabla F(\ol{\mb x}^{k})-\mb{x}^*\right\|^2 + \alpha^2\mathbb{E}\left[\left\|\nabla F(\ol{\mb x}^{k})-\ol{\mb r}^{k}|\right\|^2\mc{F}^k\right] \nonumber\\
&+ 2\alpha\Big\langle \ol{\mb x}^{k}-\alpha\nabla F(\ol{\mb x}^{k})-\mb{x}^*, \nabla F(\ol{\mb x}^{k})-\ol{\nabla\mb{f}}(\mb{x}^k)\Big\rangle, \label{opt0}
\end{align}
where in the last equality we used that~$\mathbb{E}\left[\ol{\mb{r}}^k|\mc{F}^k\right]=\ol{\nabla\mb{f}}(\mb{x}^k)$. Next, we expand and simplify~$\mathbb{E}\left[\left\|\nabla F(\ol{\mb x}^{k})-\ol{\mb r}^{k}\right\|^2|\mc{F}^k\right]$:
{\begin{align}
&\mathbb{E}\left[\left\|\nabla F(\ol{\mb x}^{k})-\ol{\mb r}^{k}\right\|^2|\mc{F}^k\right] 
\nonumber\\
=&~\left\|\nabla F(\ol{\mb x}^{k})-\ol{\nabla\mb{f}}(\mb{x}^k)\right\|^2
+ \mathbb{E}\left[\left\|\ol{\nabla\mb{f}}(\mb{x}^k)-\ol{\mb r}^{k}\right\|^2|\mc{F}^k\right] 
\label{var+consensus0}
\end{align}}where we used the fact that $$\Big\langle\nabla F(\ol{\mb x}^{k})-\ol{\nabla\mb{f}}(\mb{x}^k),\mathbb{E}\left[\ol{\nabla\mb{f}}(\mb{x}^k)-\ol{\mb r}^{k}|\mc{F}^k\right]\Big\rangle=0.$$ For the last term in~\eqref{var+consensus0}, we have that:
\begin{align}
&\mathbb{E}\left[\left\|\ol{\nabla\mb{f}}(\mb{x}^k)-\ol{\mb r}^{k}\right\|^2|\mc{F}^k\right] \nonumber\\
=&~\frac{1}{n^2}\mathbb{E}\left[\left\|\sum_{i=1}^{n}\left(\mb{r}_i^k-\nabla f_i(\mb{x}_i^k)\right)\right\|^2\Big|\mc{F}^k\right] \nonumber\\
=&~\frac{1}{n^2}\mathbb{E}\left[\left\|\mb{r}^k-\nabla\mb{f}(\mb{x}^k)\right\|^2|\mc{F}^k\right], \label{var+consensus1}
\end{align} 
where in the equality above we used the fact that~$\{\mb{r}_i^k\}_{i=1}^n$ are independent from each other and from~$\mc{F}^k$ and therefore~$\mathbb{E}\left[\sum_{i\neq j}\big\langle \mb{r}_i^k-\nabla f_i(\mb{x}_i^k), \mb{r}_j^k-\nabla f_j(\mb{x}_j^k)\big\rangle|\mc{F}^k\right]=0$. Now, we use~\eqref{var+consensus0},~\eqref{var+consensus1} and Lemma~\ref{descent} in~\eqref{opt0} to obtain:
\begin{align}
&\mathbb{E}\left[\left\|\ol{\mb x}^{k+1}-\mb{x}^*\right\|^2|\mc{F}^k\right]\nonumber\\
\leq&~ (1-\mu\alpha)^2\|\ol{\mb{x}}^k-\mb{x}^*\|^2  
+ \alpha^2\left\|\nabla F(\ol{\mb x}^{k})-\ol{\nabla\mb{f}}(\mb{x}^k)\right\|^2\nonumber\\
&+ 2\alpha(1-\mu\alpha)\left\|\ol{\mb x}^{k}-\mb{x}^*\right\|\left\|\nabla F(\ol{\mb x}^{k})-\ol{\nabla\mb{f}}(\mb{x}^k)\right\| \nonumber\\
&+ \frac{\alpha^2}{n^2}\mathbb{E}\left[\left\|\mb{r}^k-\nabla\mb{f}(\mb{x}^k)\right\|^2|\mc{F}^k\right].
\label{opt01}
\end{align}
Finally, we apply Young's inequality such that
{\begin{align}
&2\alpha\left\|\ol{\mb x}^{k}-\mb{x}^*\right\|\left\|\nabla F(\ol{\mb x}^{k})-\ol{\nabla\mb{f}}(\mb{x}^k)\right\| \nonumber\\
\leq&~\mu\alpha\left\|\ol{\mb x}^{k}-\mb{x}^*\right\|^2 + \mu^{-1}\alpha\left\|\nabla F(\ol{\mb x}^{k})-\ol{\nabla\mb{f}}(\mb{x}^k)\right\|^2 \nonumber
\end{align}}and~$\left\|\ol{\nabla\mb{f}}(\mb{x}^{k})-\nabla F(\ol{\mb{x}}^{k})\right\|\leq\frac{L}{\sqrt{n}}\left\|\mb{x}^{k}-W_\infty\mb{x}^{k}\right\|,\forall k\geq0$, from Lemma~\ref{L-bound} to~\eqref{opt01} and take the total expectation; the resulting inequality is exactly~\eqref{opt_1}. Similarly, using
{\begin{align}
&2\alpha\left\|\ol{\mb x}^{k}-\mb{x}^*\right\|\left\|\nabla F(\ol{\mb x}^{k})-\ol{\nabla\mb{f}}(\mb{x}^k)\right\| \nonumber\\
\leq& \left\|\ol{\mb x}^{k}-\mb{x}^*\right\|^2 + \alpha^2\left\|\nabla F(\ol{\mb x}^{k})-\ol{\nabla\mb{f}}(\mb{x}^k)\right\| \nonumber
\end{align}}
and Lemma~\ref{L-bound} in~\eqref{opt01} leads to~\eqref{opt_2}.
\end{proof}
Next, we derive an inequality for~$\mathbb{E}\left[\|\mb{y}^{k+1}-W_\infty\mb{y}^{k+1}\|^2\right]$.
\begin{lem}\label{general_3}
Let Assumption~\ref{smooth} and Assumption~\ref{connect} hold. Consider the iterates~$\{\mb{y}^k\}$ generated by~\eqref{alg_a}-\eqref{alg_b}. If~$0<\alpha\leq\frac{1}{4\sqrt{2}L}$, we have the following inequality hold:~$\forall k\geq0$,
{\begin{align}
&\mathbb{E}\left[\left\|\mb{y}^{k+1}-W_\infty\mb{y}^{k+1}\right\|^2\right]\nonumber\\
\leq&~\frac{33L^2}{1-\sigma^2}\mathbb{E}\left[\left\|\mb{x}^k-W_\infty\mb{x}^k\right\|^2\right] + \frac{L^2}{1-\sigma^2}\mathbb{E}\left[n\left\|\ol{\mb{x}}^k-\mb{x}^*\right\|^2\right] 
\nonumber\\
&+ \left(\frac{1+\sigma^2}{2}+\frac{32L^2\alpha^2}{1-\sigma^2}\right)\mathbb{E}\left[\left\|\mb{y}^k-W_\infty\mb{y}^k\right\|^2\right] \nonumber\\
&+ \frac{5}{1-\sigma^2}\mathbb{E}\left[\left\|\mb{r}^{k}-\nabla\mb{f}(\mb{x}^{k})\right\|^2\right]\nonumber\\
&+ \frac{4}{1-\sigma^2}\mathbb{E}\left[\left\|\mb{r}^{k+1}-\nabla\mb{f}(\mb{x}^{k+1})\right\|^2\right]
\nonumber.
\end{align}}
\end{lem}
\begin{proof}
Using~\eqref{alg_b} and the fact that~$W_\infty W= W_\infty$, we have:
{\begin{align}\label{y0}
&\left\|\mb{y}^{k+1}-W_\infty\mb{y}^{k+1}\right\|^2 \nonumber\\
=& \left\|W\mb{y}^{k} + \mb{r}^{k+1} - \mb{r}^{k} -W_\infty\left(W\mb{y}^{k} + \mb{r}^{k+1} - \mb{r}^{k}\right)\right\|^2 \nonumber\\
=& \left\|W\mb{y}^{k}-W_\infty\mb{y}^{k}+\left(I_{np}-W_\infty\right)\left(\mb{r}^{k+1} - \mb{r}^{k}\right)\right\|^2.
\end{align}}To proceed from~\eqref{y0}, we use Young's inequality that~$\|\mb{a}+\mb{b}\|^2\leq(1+\eta)\|\mb{a}\|^2+(1+\frac{1}{\eta})\|\mb{b}\|^2,\forall\mb{a},\mb{b}\in\mathbb{R}^{np}$ with~$\eta=\frac{2\sigma^2}{1-\sigma^2}$ and that~$\mn{I_{np}-W_\infty} = 1$ together with Lemma~\ref{W_contract} to obtain:
{\begin{align}
&\left\|\mb{y}^{k+1}-W_\infty\mb{y}^{k+1}\right\|^2 \nonumber\\
\leq&~\left(1+\frac{1-\sigma^2}{2\sigma^2}\right)\left\|W\mb{y}^{k}-W_\infty\mb{y}^{k}\right\|^2 \nonumber\\
&+ \left(1+\frac{2\sigma^2}{1-\sigma^2}\right)\left\|\left(I_{np}-W_\infty\right)\left(\mb{r}^{k+1}-\mb{r}^{k}\right)\right\|^2\nonumber\\
\leq&~\frac{1+\sigma^2}{2}\left\|\mb{y}^{k}-W_\infty\mb{y}^{k}\right\|^2 + \frac{2}{1-\sigma^2}\left\|\mb{r}^{k+1}-\mb{r}^{k}\right\|^2.
\end{align}}We then take the total expectation to obtain:
{
\begin{align}
&\mathbb{E}\left[\left\|\mb{y}^{k+1}-W_\infty\mb{y}^{k+1}\right\|^2\right] \nonumber\\
\leq~&\frac{1+\sigma^2}{2}\mathbb{E}\left[\left\|\mb{y}^{k}-W_\infty\mb{y}^{k}\right\|^2\right] + \frac{2}{1-\sigma^2}\mathbb{E}\left[\left\|\mb{r}^{k+1}-\mb{r}^{k}\right\|^2\right] \label{yk-contract}
\end{align}
}Now, we derive an upper bound for~$\mathbb{E}[\|\mb{r}^{k+1}-\mb{r}^{k}\|^2]$. Firstly,
{\small\begin{align}
&\mathbb{E}\left[\|\mb{r}^{k+1}-\mb{r}^{k}\|^2\right]\nonumber\\
\leq&~2\mathbb{E}\left[\|\mb{r}^{k+1}-\mb{r}^{k}-(\nabla\mb{f}(\mb{x}^{k+1})-\nabla\mb{f}(\mb{x}^{k}))\|^2\right] \nonumber\\
&+ 2\mathbb{E}\left[\|\nabla\mb{f}(\mb{x}^{k+1})-\nabla\mb{f}(\mb{x}^{k})\|^2\right]\nonumber\\
\leq&~2\mathbb{E}\left[\|\mb{r}^{k}-\nabla\mb{f}(\mb{x}^{k})\|^2\right]
+ 2\mathbb{E}\left[\|\mb{r}^{k+1}-\nabla\mb{f}(\mb{x}^{k+1})\|^2\right]
\nonumber\\
&+ 2L^2\mathbb{E}\left[\|\mb{x}^{k+1}-\mb{x}^{k}\|^2\right]\label{sg-diff}
\end{align}}where in the last inequality above we used that
\begin{align*}
&\mathbb{E}\left[\langle\mb{r}^{k+1}-\nabla\mb{f}(\mb{x}^{k+1}),\mb{r}^{k}-\nabla\mb{f}(\mb{x}^{k})\rangle\right] \nonumber\\
=~&\mathbb{E}\left[\mathbb{E}\left[\langle\mb{r}^{k+1}-\nabla\mb{f}(\mb{x}^{k+1}),\mb{r}^{k}-\nabla\mb{f}(\mb{x}^{k})\rangle|\mc{F}^{k+1}\right]\right]=0.
\end{align*}
We next bound $\mathbb{E}\left[\left\|\mb{x}^{k+1}-\mb{x}^{k}\right\|^2\right]$. Using~\eqref{alg_a} leads to: 
\begin{align}
\left\|\mb{x}^{k+1}-\mb{x}^{k}\right\|^2 
=& \left\|W\mb{x}^{k}-\alpha\mb{y}^k-\mb{x}^{k}\right\|^2\nonumber\\
=& \left\|\left(W-I_{np}\right)\left(\mb{x}^{k}-W_\infty\mb{x}^{k}\right)-\alpha\mb{y}^k\right\|^2 \nonumber\\
\leq&~8\left\|\mb{x}^{k}-W_\infty\mb{x}^{k}\right\|^2 + 2\alpha^2\left\|\mb{y}^k\right\|^2,\label{x-diff0}
\end{align}
where in~\eqref{x-diff0} we used the fact that~$\mn{W-I_{np}}\leq2$. We then denote~$\nabla\mb{f}(\mb{x}^*) \coloneqq  \left[\nabla f_1(\mb{x}^*)^\top,\cdots,\nabla f_n(\mb{x}^*)^\top\right]^\top$ and note that~$(\mb{1}_n^\top\otimes I_p)\nabla\mb{f}(\mb{x}^*) = \mb{0}_p$. We bound~$\|\mb{y}^k\|$ as follows.
\begin{align}
\left\|\mb{y}^k\right\| =&\Big\|\mb{y}^k-W_\infty\mb{y}^k+W_\infty\mb{r}^k-W_\infty\nabla\mb{f}(\mb{x}^k)\nonumber\\
&\qquad+W_\infty\nabla\mb{f}(\mb{x}^k)-W_\infty\nabla\mb{f}(\mb{x}^*)\Big\|\nonumber\\
\leq& \left\|\mb{y}^k-W_\infty\mb{y}^k\right\| + \left\|\mb{r}^k-\nabla\mb{f}(\mb{x}^k)\right\| \nonumber\\
&+ L\left\|\mb{x}^k-(\mb{1}_n\otimes I_p)\mb{x}^*\right\|\nonumber\\
\leq& \left\|\mb{y}^k-W_\infty\mb{y}^k\right\| + \left\|\mb{r}^k-\nabla\mb{f}(\mb{x}^k)\right\| \nonumber\\
&+ L\left\|\mb{x}^k-W_\infty\mb{x}^k\right\| + \sqrt{n}L\left\|\ol{\mb{x}}^k-\mb{x}^*\right\|,
\nonumber
\end{align}
where in the first equality we used~$\ol{\mb{y}}^k=\ol{\mb{r}}^k,\forall k\geq0$. Squaring the above inequality obtains the following:
\begin{align}
\left\|\mb{y}^k\right\|^2 
\leq&~ 4L^2\left\|\mb{x}^k-W_\infty\mb{x}^k\right\|^2 + 4nL^2\left\|\ol{\mb{x}}^k-\mb{x}^*\right\|^2 
\nonumber\\
&+ 4\left\|\mb{y}^k-W_\infty\mb{y}^k\right\|^2 + 4\left\|\mb{r}^k-\nabla\mb{f}(\mb{x}^k)\right\|^2.
\label{yk}
\end{align}
Using~\eqref{yk} in~\eqref{x-diff0} with the requirement that~$0<\alpha\leq\frac{1}{4\sqrt{2}L}$ and taking the total expectation, we have:
{\small\begin{align}
&\mathbb{E}\left[\left\|\mb{x}^{k+1}-\mb{x}^{k}\right\|^2\right] \nonumber\\
\leq&~  
8.25\mathbb{E}\left[\left\|\mb{x}^k-W_\infty\mb{x}^k\right\|^2\right] + 0.25\mathbb{E}\left[n\left\|\ol{\mb{x}}^k-\mb{x}^*\right\|^2\right] \nonumber\\
&+ 8\alpha^2\mathbb{E}\left[\left\|\mb{y}^k-W_\infty\mb{y}^k\right\|^2\right] 
+ 8\alpha^2\mathbb{E}\left[\left\|\mb{r}^k-\nabla\mb{f}(\mb{x}^k)\right\|^2\right]. 
\label{x-diff1}
\end{align}}
Finally, we apply~\eqref{x-diff1} in~\eqref{sg-diff} with~$0<\alpha\leq\frac{1}{4\sqrt{2}L}$ to obtain:
{\begin{align}
&\mathbb{E}\left[\left\|\mb{r}^{k+1}-\mb{r}^{k}\right\|^2\right] \nonumber\\
\leq&~ 16.5L^2\mathbb{E}\left[\left\|\mb{x}^k-W_\infty\mb{x}^k\right\|^2\right] + 0.5L^2\mathbb{E}\left[n\left\|\ol{\mb{x}}^k-\mb{x}^*\right\|^2\right] \nonumber\\
&+ 16\alpha^2L^2\mathbb{E}\left[\left\|\mb{y}^k-W_\infty\mb{y}^k\right\|^2\right]
\nonumber\\
&+ 2.5\mathbb{E}\left[\left\|\mb{r}^{k}-\nabla\mb{f}(\mb{x}^{k})\right\|^2\right]
\nonumber
+ 2\mathbb{E}\left[\left\|\mb{r}^{k+1}-\nabla\mb{f}(\mb{x}^{k+1})\right\|^2\right]
\end{align}}Using the above inequality in~\eqref{yk-contract} completes the proof.
\end{proof}

We finally present a general convergence result on a sequence of random variables that converge linearly in the mean-square sense. We note that this result is implied in the probability literature; see~\cite{probability_williams} for example. For the sake of completeness, we present its proof here.
\begin{lem}\label{BC}
Let~$\{X_k\}_{k\geq0}$ be a sequence of random variables such that~$\mathbb{E}[|X_k|] \leq \gamma^k$ for some~$0<\gamma<1$. Then we have
\begin{align*}
\mathbb{P}\left(\lim_{k\rightarrow\infty}(\gamma + \delta)^{-k}|X_k| = 0\right) = 1,
\end{align*}
where~$\delta > 0$ is an arbitrary positive constant.
\end{lem}
\begin{proof}
By Chebyshev's inequality, we have:~$\forall \epsilon>0,\forall\delta>0$,
\begin{align*}
\mathbb{P}\left((\gamma + \delta)^{-k}|X_k| > \epsilon \right)
\leq&~\epsilon^{-1}\mathbb{E}[(\gamma + \delta)^{-k}|X_k|] \nonumber\\
\leq&~\epsilon^{-1}(\gamma/(\gamma + \delta))^{k}.
\end{align*}
Summing the inequality above over~$k$, we obtain:
\begin{align*}
\sum_{k=0}^{\infty}\mathbb{P}\left((\gamma + \delta)^{-k}|X_k| > \epsilon \right)
\leq \epsilon^{-1}\sum_{k=0}^{\infty}\Big(\frac{\gamma}{\gamma + \delta}\Big)^{k} < \infty.
\end{align*}
By the Borel-Cantelli lemma,
\begin{align*}
\mathbb{P}\left((\gamma + \delta)^{-k}|X_k| > \epsilon~\text{for infinitely many}~k\right)
= 0,
\end{align*}
and the proof follows.
\end{proof}
We note that Lemma~\ref{BC} states that the non-asymptotic linear convergence of a sequence of random variables in the mean-square sense implies its asymptotic linear convergence in the almost sure sense. As a consequence, Corollaries~\ref{as_gtsaga} and~\ref{as_gtsvrg} will be immediately at hand once Theorems~\ref{main_gtsaga} and~\ref{main_gtsvrg} are established.
 
\noindent With the help of the auxiliary results on the general dynamical system~\eqref{alg_a}-\eqref{alg_b} established in this section, we now derive explicit convergence rates for the proposed algorithms, \textbf{\texttt{GT-SAGA}} and \textbf{\texttt{GT-SVRG}}, in the next sections.

\section{Convergence analysis of \textbf{\texttt{GT-SAGA}}}\label{sGTSAGA} 
In this section, we establish the mean-square linear convergence of \textbf{\texttt{GT-SAGA}} described in Algorithm~\ref{GT-SAGA}. Following the unified representation in~\eqref{alg_a}-\eqref{alg_b}, we recall that the local gradient estimator~$\mb{r}_i^k$ is given by~$\mb{g}_i^k$  in~\textbf{\texttt{GT-SAGA}}, i.e.,~$\forall i\in\{1,\cdots,n\},\forall k\geq1$,
\begin{align*}
\mb{g}_{i}^{k} = \nabla f_{i,s_i^{k}}\big(\mb{x}_{i}^{k}\big) - \nabla f_{i,s_i^{k}}\big(\mb{z}_{i,s_i^{k}}^{k}\big) + \tfrac{1}{m_i}\textstyle\sum_{j=1}^{m_i}\nabla f_{i,j}\big(\mb{z}_{i,j}^{k}\big),
\end{align*}
where~$s_i^k$ is selected uniformly at random from~$\{1,\cdots,m_i\}$ and the auxiliary variable~$\mb{z}_{i,j}^k$ is the most recent iterate where the component gradient~$\nabla f_{i,j}$ was evaluated up to time~$k$.

\subsection{Bounding the variance of the gradient estimator} We first derive an upper bound for~$\mathbb{E}\left[\|\mb{g}^k-\nabla\mb{f}(\mb{x}^k)\|^2\right]$ that is the variance of the gradient estimator~$\mb{g}^k$. To do this, we define~$t_i^k$ as the averaged optimality gap of the auxiliary variables of~$\{\mb{z}_{i,j}^k\}_{j=1}^{m_i}$ at node~$i$ as follows:  
\begin{align}
\textstyle t_i^k \coloneqq  \frac{1}{m_i}\sum_{j=1}^{m_i}\left\|\mb{z}_{i,j}^k-\mb{x}^*\right\|^2, \qquad
t^k \coloneqq  \sum_{i=1}^{n}t_i^k.
\end{align}
The following lemma shows that~$t^k$ has an intrinsic contraction property. Recall that~$M = \max_i{m_i}$ and~$m = \min_i{m_i}$.
\begin{lem}\label{t_contract}
Consider the iterates~$\{t^k\}$ generated by~\textbf{\texttt{GT-SAGA}}. We have the following holds:~$\forall k\geq1$,
\begin{align}
\mathbb{E}\left[t^{k+1}\right]
\leq&\left(1-\frac{1}{M}\right)\mathbb{E}\left[t^{k}\right]
+ \frac{2}{m}\mathbb{E}\left[\left\|\mb{x}^k-W_\infty\mb{x}^k\right\|^2\right] \nonumber\\
&+ \frac{2}{m}\mathbb{E}\left[n\left\|\ol{\mb{x}}^k-\mb{x}^*\right\|^2\right]. \nonumber
\end{align}
\end{lem}
\begin{proof}
Recall Algorithm~\ref{GT-SAGA} and note that~$\forall k\geq1,~\mb{z}_{i,j}^{k+1} = \mb{z}_{i,j}^{k}$ with probability~$1-\frac{1}{m_i}$ and~$\mb{z}_{i,j}^{k+1} = \mb{x}_{i}^{k}$ with probability~$\frac{1}{m_i}$ given~$\mc{F}^k$. Then we have the following holds:~$\forall i,\forall k\geq1$,
{\small\begin{align}
&\mathbb{E}\left[t_i^{k+1}|\mc{F}^k\right] \nonumber\\
=&~\frac{1}{m_i}\sum_{j=1}^{m_i}\mathbb{E}\left[\left\|\mb{z}_{i,j}^{k+1}-\mb{x}^*\right\|^2|\mc{F}^k\right]
\nonumber\\
=&~\frac{1}{m_i}\sum_{j=1}^{m_i}\mathbb{E}\left[\left(1-\frac{1}{m_i}\right)\left\|\mb{z}_{i,j}^{k}-\mb{x}^*\right\|^2 + \frac{1}{m_i}\left\|\mb{x}_{i}^{k}-\mb{x}^*\right\|^2\Big|\mc{F}^k\right]\nonumber\\
=&~\left(1-\frac{1}{m_i}\right)t_i^k + \frac{1}{m_i}\left\|\mb{x}_{i}^{k}-\mb{x}^*\right\|^2\nonumber\\
\leq&~\left(1-\frac{1}{M}\right)t_i^k + \frac{2}{m}\left\|\mb{x}_{i}^{k}-\ol{\mb{x}}^k\right\|^2
+ \frac{2}{m}\left\|\ol{\mb{x}}^k-\mb{x}^*\right\|^2 \label{t_contract_i}
\end{align}}The proof follows by summing~\eqref{t_contract_i} over~$i$ and taking the total expectation.
\end{proof}
In the next lemma, we bound the stochastic gradient variance~$\mathbb{E}\left[\|\mb{g}^k-\nabla\mb{f}(\mb{x}^k)\|^2\right]$ by the mean-square consensus error and the optimality gap of~$\mb{x}^k$ and~$t^k$.
\begin{lem}\label{var_bounnd_saga}
Let Assumption~\ref{smooth} hold. Consider the iterates~$\{\mb{g}^k\}$ generated by \textbf{\texttt{GT-SAGA}}. Then we have the following inequality hold:~$\forall k\geq1$, \begin{align}
\mathbb{E}\left[\left\|\mb{g}^k-\nabla\mb{f}(\mb{x}^k)\right\|^2\right] 
\leq&~4L^2\mathbb{E}\left[\left\|\mb{x}^k-W_\infty\mb{x}^k\right\|^2\right] \nonumber\\
+& 4L^2\mathbb{E}\left[n\left\|\ol{\mb{x}}^k-\mb{x}^*\right\|^2\right] + 2L^2 \mathbb{E}\left[t^k\right]. \nonumber
\end{align}
\end{lem}   
\begin{proof}
Recall the local gradient estimator~$\mb{g}_i^k$ from Algorithm~\ref{GT-SAGA} and proceed as follows.
\begin{align}
&\mathbb{E}\left[\left\|\mb{g}_i^k-\nabla f_i(\mb{x}_i^k)\right\|^2|\mc{F}^k\right] \nonumber\\
=&~\mathbb{E}\Big[\Big\|\nabla f_{i,s_i^{k}}\big(\mb{x}_{i}^{k}\big) - \nabla f_{i,s_i^{k}}\big(\mb{z}_{i,s_i^{k}}^{k}\big) \nonumber\\
&\qquad\qquad-\left(\nabla f_i(\mb{x}_i^k) 
- \tfrac{1}{m_i}\textstyle\sum_{j=1}^{m_i}\nabla f_{i,j}\big(\mb{z}_{i,j}^{k}\big)\right)\Big\|^2 \Big|\mc{F}^k\Big] \nonumber\\
\leq&~\mathbb{E}\left[\left\|\nabla f_{i,s_i^{k}}\big(\mb{x}_{i}^{k}\big) - \nabla f_{i,s_i^{k}}\big(\mb{z}_{i,s_i^{k}}^{k}\big) \right\|^2\Big|\mc{F}^k\right] \nonumber\\
=&~\textstyle{\frac{1}{m_i}\sum_{j=1}^{m_i}}\Big\|\left(\nabla f_{i,j}\big(\mb{x}_{i}^{k}\big) -\nabla f_{i,j}(\mb{x}^*)\right) \nonumber\\
&\qquad\qquad\qquad+ \left(\nabla f_{i,j}(\mb{x}^*)- \nabla f_{i,j}\big(\mb{z}_{i,j}^{k}\big)\right) \Big\|^2 \nonumber\\
\leq&~2L^2\left\|\mb{x}_i^k-\mb{x}^*\right\|^2 + 2L^2t_i^k\nonumber\\
\leq&~4L^2\left\|\mb{x}_i^k-\ol{\mb{x}}^k\right\|^2 + 4L^2\left\|\ol{\mb{x}}^k-\mb{x}^*\right\|^2 + 2L^2 t_i^k, \label{var_saga_i}
\end{align}
where the second inequality uses the standard conditional variance decomposition
\begin{align}\label{VRD}
&\mathbb{E}\left[\left\|\mb{a}_i^k-\mathbb{E}\left[\mb{a}_i^k|\mc{F}^k\right]\right\|^2|\mc{F}^k\right]\nonumber\\
=&~\mathbb{E}\left[\left\|\mb{a}_i^k\right\|^2|\mc{F}^k\right] -\left\|\mathbb{E}\left[\mb{a}_i^k|\mc{F}^k\right]\right\|^2  \leq\mathbb{E}\left[\left\|\mb{a}_i^k\right\|^2|\mc{F}^k\right],
\end{align}
with~$\mb{a}_i^k = \nabla f_{i,s_i^{k}}\big(\mb{x}_{i}^{k}\big) - \nabla f_{i,s_i^{k}}\big(\mb{z}_{i,s_i^{k}}^{k}\big)$. The proof follows by summing~\eqref{var_saga_i} over~$i$ and taking the total expectation.
\end{proof}  
Lemma~\ref{var_bounnd_saga} clearly shows that as~$\mb{x}_i^k$ and~$\mb{z}_{i,j}^k$ approach to an agreement on~$\mb{x}^*$, the variance of the gradient estimator decays to zero. We have the following corollary.
\begin{corollary}\label{var_bound_saga_2}
Let Assumption~\ref{smooth} and~\ref{connect} hold. Consider the iterates~$\{\mb{g}^k\}$ generated by \textbf{\texttt{GT-SAGA}}. If~$0<\alpha\leq\frac{1}{4\sqrt{2}L}$, then the following inequality holds~$\forall k\geq0$,
\begin{align}
&\mathbb{E}\left[\left\|\mb{g}^{k+1}-\nabla\mb{f}(\mb{x}^{k+1})\right\|^2\right] \nonumber\\
\leq&~12.75L^2\mathbb{E}\left[\left\|\mb{x}^k-W_\infty\mb{x}^k\right\|^2\right] + 12.5L^2\mathbb{E}\left[n\left\|\ol{\mb{x}}^k-\mb{x}^*\right\|^2\right]  \nonumber\\
& + 8L^2\alpha^2\mathbb{E}\left[\left\|\mb{y}^k-W_\infty\mb{y}^k\right\|^2\right] + 2.25L^2 \mathbb{E}\left[t^k\right]. \nonumber
\end{align}
\end{corollary}
\begin{proof}
Following directly from Lemma~\ref{var_bounnd_saga}, we have:~$\forall k\geq0$,
\begin{align}
&\mathbb{E}\left[\left\|\mb{g}^{k+1}-\nabla\mb{f}(\mb{x}^{k+1})\right\|^2\right]\nonumber\\
\leq&~4L^2\mathbb{E}\left[\left\|\mb{x}^{k+1}-W_\infty\mb{x}^{k+1}\right\|^2\right] + 4L^2\mathbb{E}\left[n\left\|\ol{\mb{x}}^{k+1}-\mb{x}^*\right\|^2\right] \nonumber\\
&+ 2L^2 \mathbb{E}\left[t^{k+1}\right]. \nonumber
\end{align}
Using~\eqref{consensus2},~\eqref{opt_2} and Lemma~\ref{t_contract} in the inequality above leads to the following: if~$0<\alpha\leq\frac{1}{4\sqrt{2}L}$,
\begin{align}
&\mathbb{E}\left[\left\|\mb{g}^{k+1}-\nabla\mb{f}(\mb{x}^{k+1})\right\|^2\right]\nonumber\\
\leq&~12.25L^2\mathbb{E}\left[\left\|\mb{x}^k-W_\infty\mb{x}^k\right\|^2\right]
+ 12L^2\mathbb{E}\left[n\|\ol{\mb{x}}^k-\mb{x}^*\|^2\right] \nonumber\\
&+ 8L^2\alpha^2\mathbb{E}\left[\left\|\mb{y}^k-W_\infty\mb{y}^k\right\|^2\right] \nonumber\\
&+ 2L^2\mathbb{E}\left[t^{k}\right] +0.125\mathbb{E}\left[\left\|\mb{g}^k-\nabla\mb{f}(\mb{x}^k)\right\|^2\right].
\nonumber
\end{align}
The proof follows by applying Lemma~\ref{var_bounnd_saga} in the above. 
\end{proof}
\vspace{-1cm}
\subsection{Main results for~\textbf{\texttt{GT-SAGA}}} With the bounds on the gradient variance for \textbf{\texttt{GT-SAGA}} derived in the previous subsection, we are now able to refine the inequalities obtained for the general dynamical system~\eqref{alg_a}-\eqref{alg_b} in Section~\ref{general} and derive the explicit convergence rates for~\textbf{\texttt{GT-SAGA}}. First, we apply the upper bound on~$\mathbb{E}[\|\mb{g}^k-\nabla\mb{f}(\mb{x}^k)\|^2]$ in Lemma~\ref{var_bounnd_saga} to~\eqref{opt_1} to obtain:~$\forall k\geq0$,
\begin{align}
&\mathbb{E}\left[n\left\|\ol{\mb x}^{k+1}-\mb{x}^*\right\|^2\right] \nonumber\\ 
\leq&~L^2\alpha\left(\frac{1}{\mu}+\frac{4\alpha}{n}\right)\mathbb{E}\left[\left\|\mb{x}^k-W_\infty\mb{x}^k\right\|^2\right] \nonumber\\
&+ \left(1-\mu\alpha + \frac{4L^2\alpha^2}{n}\right)\mathbb{E}\left[n\|\ol{\mb{x}}^k-\mb{x}^*\|^2\right] +\frac{2L^2\alpha^2}{n}\mathbb{E}\left[t^k\right]. \nonumber 
\end{align}
If~$0<\alpha\leq\frac{1}{4\mu}$, then~$\frac{1}{\mu}+\frac{4\alpha}{n}\leq\frac{2}{\mu}$; if~$0<\alpha\leq\frac{\mu n}{8L^2}$, then we have~$1-\mu\alpha + \frac{4L^2\alpha^2}{n}\leq1-\frac{\mu\alpha}{2}$. Therefore, if~$0<\alpha\leq\frac{\mu}{8L^2}$, we have the following:~$\forall k\geq0$, 
\begin{align}\label{saga2}
&\mathbb{E}\left[n\left\|\ol{\mb x}^{k+1}-\mb{x}^*\right\|^2\right] \nonumber\\
\leq&~ 
\frac{2L^2\alpha}{\mu}\mathbb{E}\left[\left\|\mb{x}^k-W_\infty\mb{x}^k\right\|^2\right]
+ \left(1-\frac{\mu\alpha}{2}\right)\mathbb{E}\left[n\|\ol{\mb{x}}^k-\mb{x}^*\|^2\right] \nonumber\\
&+ \frac{2L^2\alpha^2}{n}\mathbb{E}\left[t^k\right]
\end{align}
Second, we apply the upper bounds on~$\mathbb{E}[\|\mb{g}^k-\nabla\mb{f}(\mb{x}^k)\|^2]$ and~$\mathbb{E}[\|\mb{g}^{k+1}-\nabla\mb{f}(\mb{x}^{k+1})\|^2]$ in Lemma~\ref{var_bounnd_saga} and Corollary~\ref{var_bound_saga_2} to Lemma~\ref{general_3} to obtain the following:~$\forall k\geq0$,  
\begin{align}\label{saga4}
&\mathbb{E}\left[\left\|\mb{y}^{k+1}-W_\infty\mb{y}^{k+1}\right\|^2\right]\nonumber\\
\leq&~\frac{104L^2}{1-\sigma^2}\mathbb{E}\left[\left\|\mb{x}^k-W_\infty\mb{x}^k\right\|^2\right]
+ \frac{71L^2}{1-\sigma^2}\mathbb{E}\left[n\left\|\ol{\mb{x}}^k-\mb{x}^*\right\|^2\right]\nonumber\\
&+ \frac{19L^2}{1-\sigma^2}\mathbb{E}\left[t^k\right]+\frac{3+\sigma^2}{4}\mathbb{E}\left[\left\|\mb{y}^k-W_\infty\mb{y}^k\right\|^2\right],
\end{align}
if~$0<\alpha\leq\frac{1-\sigma^2}{16L}$.
To proceed, we write~\eqref{consensus1},~\eqref{saga2}, Lemma~\ref{t_contract} and~\eqref{saga4} jointly as a linear matrix inequality.
\begin{proposition}\label{prop_gtsaga}
Let Assumptions~\ref{sc},~\ref{smooth},~\ref{connect} hold and consider the iterates~$\{\mb{x}^k\},\{\mb{y}^k\},\{t^k\}$ generated by \textbf{\texttt{GT-SAGA}}. If the step-size~$\alpha$ follows~$0<\alpha\leq\frac{\mu(1-\sigma)}{16L^2}$, we have:~$\forall k\geq1$,
\begin{align}
\mb{u}^{k+1} \leq G_\alpha\mb{u}^{k}, 
\end{align}
where~$\mb{u}^k\in\mathbb{R}^4$ and~$G_\alpha\in\mathbb{R}^{4\times4}$ are defined as follows:
{\small\begin{align}
	\mb{u}^k&=
	\begin{bmatrix}
	\mathbb{E}\left[\left\|\mb{x}^{k}-W_\infty\mb{x}^{k}\right\|^2\right] \\[1ex]
	\mathbb{E}\left[n\left\|\ol{\mb x}^{k}-\mb{x}^*\right\|^2\right] \\[1ex]
	\mathbb{E}\left[t^{k}\right] \\[1ex]
	\mathbb{E}\left[L^{-2}\left\|\mb{y}^{k} - W_\infty\mb{y}^{k}\right\|^2\right]
	\end{bmatrix},\quad\nonumber\\
	G_\alpha&=
	\begin{bmatrix}
	\dfrac{1+\sigma^2}{2} & 0 & 0 &\dfrac{2\alpha^2L^2}{1-\sigma^2} \\[2ex]
	\dfrac{2L^2\alpha}{\mu}&  1-\dfrac{\mu\alpha}{2} & \dfrac{2L^2\alpha^2}{n} & 0\\[2ex]
	\dfrac{2}{m} & \dfrac{2}{m} & 1-\dfrac{1}{M} &0\\[2ex]
	\dfrac{104}{1-\sigma^2} & \dfrac{71}{1-\sigma^2} & \dfrac{19}{1-\sigma^2}  & \dfrac{3+\sigma^2}{4}
	\end{bmatrix}.\nonumber
\end{align}}
\end{proposition}Clearly, to show the linear convergence of~\textbf{\texttt{GT-SAGA}}, it suffices to derive the range of~$\alpha$ such that~$\rho(G_\alpha)<1$. To do this, we present a useful lemma from~\cite{matrix_analysis}. 
\begin{lem}\label{rho_bound}
		Let~$A\in\mathbb{R}^{d\times d}$ be non-negative and~$\mb{x}\in\mathbb{R}^d$ be positive. If~$A\mb{x}\leq\beta\mb{x}$ for~$\beta>0$, then~$\rho(A)\leq\mn{A}_{\infty}^{\mb{x}}\leq\beta.$
\end{lem}
We are ready to prove Theorem~\ref{main_gtsaga} based on Proposition~\ref{prop_gtsaga}.

\begin{P1}
Recall from Proposition~\ref{prop_gtsaga} that if~$0<\alpha\leq\frac{\mu(1-\sigma)}{16L^2}$, we have~$\mb{u}^{k+1}\leq G_\alpha\mb{u}^{k}$. In the light of Lemma~\ref{rho_bound}, we solve for the range of the step-size~$\alpha$ and a positive vector~$\bds\epsilon = \left[\epsilon_1,\epsilon_2,\epsilon_3,\epsilon_4\right]^\top$ such that the following (entry-wise) linear matrix inequality holds:
\begin{align}\label{G}
G_\alpha\bds{\epsilon} \leq \left(1-\frac{\mu\alpha}{4}\right)\bds\epsilon,
\end{align}
which can be written equivalently in the following form:
\begin{align}
&\frac{\mu\alpha}{4} + \frac{2L^2}{1-\sigma^2}\frac{\epsilon_4}{\epsilon_1}\alpha^2 \leq\frac{1-\sigma^2}{2} 
\label{r1'}\\
&\frac{2L^2}{n}\epsilon_3\alpha\leq\frac{\mu}{4}\epsilon_2
- \frac{2L^2}{\mu}\epsilon_1
\label{r2'}\\
&\frac{\mu\alpha}{4}\leq\frac{1}{M} - \frac{2}{m}\frac{\epsilon_1}{\epsilon_3}
- \frac{2}{m}\frac{\epsilon_2}{\epsilon_3} \label{r3'}\\
&\frac{\mu\alpha}{4}\leq
\frac{1-\sigma^2}{4}
-\frac{104}{1-\sigma^2}\frac{\epsilon_1}{\epsilon_4}
- \frac{71}{1-\sigma^2}\frac{\epsilon_2}{\epsilon_4}
- \frac{19}{1-\sigma^2}\frac{\epsilon_3}{\epsilon_4} \label{r4'}
\end{align} 
Clearly, that~\eqref{r2'}--\eqref{r4'} hold for some feasible range of~$\alpha$ is equivalent to the RHS of~\eqref{r2'}--\eqref{r4'} being positive. Based on this observation, we will next fix the values of~$\epsilon_1,\epsilon_2,\epsilon_3,\epsilon_4$ that are independent of~$\alpha$. First, for the RHS of~\eqref{r2'} to be positive, we set~$\epsilon_1 = 1, \epsilon_2 = 8.5Q^2,$ where~$Q = L/\mu$.
Second, the RHS of~\eqref{r3'} being positive is equivalent to
\begin{align}
\epsilon_3 > \frac{2M}{m}\epsilon_1 + \frac{2M}{m}\epsilon_2
		   = \frac{2M}{m} + \frac{17MQ^2}{m}.
\end{align}
We therefore set~$\epsilon_3 = \frac{20MQ^2}{m}$. Third, we note that the RHS of~\eqref{r4'} being positive is equivalent to the following:
\begin{align}
\epsilon_4 >&~ \frac{4}{(1-\sigma^2)^2}\left(104\epsilon_1+71\epsilon_2+19\epsilon_3\right) \nonumber\\
=&~ \frac{4}{(1-\sigma^2)^2}\left(104+603.5Q^2+\frac{380MQ^2}{m}\right) \nonumber
\end{align}
Note that~$104+603.5Q^2+\frac{380MQ^2}{m}\leq\frac{1087.5MQ^2}{m}$. We therefore set~$\epsilon_4 = \frac{8700}{\left(1-\sigma^2\right)^2}\frac{MQ^2}{m}$. 

We now solve for the range of~$\alpha$ from~\eqref{r1'}--\eqref{r4'} given the previously fixed~$\epsilon_1,\epsilon_2,\epsilon_3,\epsilon_4$. From~\eqref{r2'}, we have that
\begin{align}
\alpha \leq \frac{n}{2L^2\epsilon_3}\left(\frac{\mu}{4}\epsilon_2-\frac{2L^2}{\mu}\epsilon_1\right)	
		= \frac{m}{M}\frac{n}{320QL} \label{a2}.
\end{align}
Moreover, it is straightforward to verify that if~$\alpha$ satisfies
\begin{align}\label{a1}
0<\alpha\leq\frac{m}{M}\frac{(1-\sigma^2)^2}{320QL}
\end{align}
then~\eqref{r1'} holds. Next, to make~\eqref{r3'} hold, it suffices to make~$\alpha$:
\begin{align}
\alpha\leq\frac{1}{5\mu M}. \label{a3}
\end{align}
Finally, to make~\eqref{r4'} hold, it suffices to make
\begin{align}
\alpha\leq \frac{1-\sigma^2}{2\mu} \label{a4}.
\end{align}   
To summarize, combining~\eqref{a1}--\eqref{a4}, we conclude that if the step-size~$\alpha$ satisfies
\begin{align}
0<\alpha\leq\ol{\alpha}:=\min\left\{\frac{1}{5\mu M},\frac{m}{320M}\frac{(1-\sigma^2)^2}{LQ}\right\},
\end{align}
then~\eqref{G} holds with some~$\bds\epsilon>0$ and thus~$\rho\left(G_\alpha\right)\leq1-\frac{\mu\alpha}{4}$ according to Lemma~\ref{rho_bound}. Furhter if~$\alpha = \ol{\alpha}$, we have 
\begin{align*}
\rho\left(G_\alpha\right)\leq1-\min\left\{\frac{1}{20 M},\frac{m}{1280M}\frac{(1-\sigma^2)^2}{Q^2}\right\},
\end{align*}
which completes the proof. 
\end{P1}

\section{Convergence analysis of~\textbf{\texttt{GT-SVRG}}}\label{sGTSVRG} 

In this section, we conduct the complexity analysis of \textbf{\texttt{GT-SVRG}} in Algorithm~\ref{GT-SVRG} based on the auxiliary results derived for the general dynamical system~\eqref{alg_a}-\eqref{alg_b} in Section~\ref{general}. Recall from Algorithm~\ref{GT-SVRG} that the gradient estimator~$\mb{v}_i^k$ at each node~$i$ in \textbf{\texttt{GT-SVRG}} is given by the following: $\forall k\geq1$, choose~$s_i^k$ uniformly at random in~$\{1,\cdots,m_i\}$ and
\begin{align}
\mb{v}_{i}^{k} = \nabla f_{i,s_i^{k}}\big(\mb{x}_{i}^{k}\big) - \nabla f_{i,s_i^{k}}\big(\bds\tau_i^{k}\big) + \nabla f_i\big(\bds\tau_i^{k}\big)
\end{align}
where~$\bds\tau_i^k = \mb{x}_i^k$ if~$\bmod(k,T) = 0$, where~$T$ is the length of each inner loop iterations of \textbf{\texttt{GT-SVRG}}; otherwise~$\bds\tau_i^k = \bds\tau_i^{k-1}$. To facilitate the convergence analysis, we define an auxiliary variable~$\ol{\bds\tau}^k \coloneqq  \frac{1}{n}\sum_{i=1}^{n}\bds\tau_i^k$,~$\forall k\geq0$.

\subsection{Bounding the variance of the gradient estimator}  
We first bound the variance of the gradient estimator~$\mb{v}_i^k$, following a similar procedure as the proof of Lemma~\ref{var_bounnd_saga}.  
\begin{lem}\label{var_bound_svrg}
Let Assumption~\ref{smooth} hold and consider the iterates~$\{\mb{v}^k\}$ generated by \textbf{\texttt{GT-SVRG}} in Algorithm~\ref{GT-SVRG}. The following inequality holds~$\forall k\geq0$:
\begin{align}
&\mathbb{E}\left[\left\|\mb{v}^k-\nabla \mb{f}(\mb{x}^k)\right\|^2\right] \nonumber\\
\leq&~4L^2\mathbb{E}\left[\left\|\mb{x}^k-W_\infty\mb{x}^k\right\|^2\right] + 4L^2\mathbb{E}\left[n\left\|\ol{\mb{x}}^k-\mb{x}^*\right\|^2\right] \nonumber\\
&+ 4L^2\mathbb{E}\left[\left\|\bds\tau^k-W_\infty\bds\tau^k\right\|^2\right] + 4L^2\mathbb{E}\left[n\left\|\ol{\bds\tau}^k-\mb{x}^*\right\|^2\right]. \nonumber
\end{align}
\end{lem}
\begin{proof}
We recall from~Algorithm~\ref{GT-SVRG} the definition of each local gradient estimator~$\mb{v}_i^k$ in \textbf{\texttt{GT-SVRG}} and proceed as follows.
\begin{align}
&\mathbb{E}\left[\left\|\mb{v}_i^k-\nabla f_i(\mb{x}_i^k)\right\|^2|\mc{F}^k\right] \nonumber\\
=&~\mathbb{E}\Big[\Big\|\nabla f_{i,s_i^{k}}\big(\mb{x}_{i}^{k}\big) - \nabla f_{i,s_i^{k}}\big(\bds\tau_i^{k}\big) \nonumber\\
&\qquad\qquad\quad-\left(\nabla f_i(\mb{x}_i^k) - \nabla f_{i}\big(\bds\tau_i^{k}\big)\right)\Big\|^2 \Big|\mc{F}^k\Big] \nonumber\\
\leq&~\mathbb{E}\left[\left\|\nabla f_{i,s_i^{k}}\big(\mb{x}_{i}^{k}\big) - \nabla f_{i,s_i^{k}}\big(\bds\tau_i^{k}\big)\right\|^2\Big|\mc{F}^k\right] \nonumber\\
=&~\textstyle{\frac{1}{m_i}\sum_{j=1}^{m_i}}\Big\|\left(\nabla f_{i,j}\big(\mb{x}_{i}^{k}\big) -\nabla f_{i,j}(\mb{x}^*)\right) \nonumber\\
&\qquad\qquad\qquad+ \left(\nabla f_{i,j}(\mb{x}^*)- \nabla f_{i,j}\big(\bds\tau_i^{k}\big)\right) \Big\|^2 \nonumber\\
\leq&~2L^2\left\|\mb{x}_i^k-\mb{x}^*\right\|^2 + 2L^2\left\|\bds\tau_i^{k}-\mb{x}^*\right\|^2  \nonumber\\
\leq&~4L^2\left\|\mb{x}_i^k-\ol{\mb{x}}^k\right\|^2 + 4L^2\left\|\ol{\mb{x}}^k-\mb{x}^*\right\|^2 \nonumber\\
&+ 4L^2\left\|\bds\tau_i^{k}-\ol{\bds\tau}^{k}\right\|^2 + 4L^2\left\|\ol{\bds\tau}^{k}-\mb{x}^*\right\|^2, \label{var_bound_svrg_i}
\end{align}
where in the second inequality we used the standard conditional variance decomposition in~\eqref{VRD}.
The proof follows by summing~\eqref{var_bound_svrg_i} over~$i$ and taking the total expectation.
\end{proof}
Lemma~\ref{var_bound_svrg} shows that as~$\mb{x}^k$ and~$\bds\tau^k$ progressively approach the optimal solution~$\mb{x}^*$ of the Problem P1, the variance of the gradient estimator~$\mb{v}^k$ goes to zero.
We then immediately have the following corollary. 
\begin{corollary}\label{var_bound_svrg_2}
Let Assumption~\ref{smooth} hold and consider the iterates~$\{\mb{v}^k\}$ generated by \textbf{\texttt{GT-SVRG}}. If~$0<\alpha\leq\frac{1}{8L}$, then the following inequality holds~$\forall k\geq0$:
\begin{align}
&\mathbb{E}\left[\left\|\mb{v}^{k+1}-\nabla \mb{f}(\mb{x}^{k+1})\right\|^2\right] \nonumber\\
\leq&~16.75L^2\mathbb{E}\left[\left\|\mb{x}^k-W_\infty\mb{x}^k\right\|^2\right] \nonumber\\
&+ 16L^2\alpha^2\mathbb{E}\left[\left\|\mb{y}^k-W_\infty\mb{y}^k\right\|^2\right] 
+ 16.5L^2\mathbb{E}\left[n\|\ol{\mb{x}}^k-\mb{x}^*\|^2\right] \nonumber\\
&+ 4.5L^2\mathbb{E}\left[\left\|\bds\tau^{k}-W_\infty\bds\tau^{k}\right\|^2\right]  + 4.5L^2\mathbb{E}\left[n\left\|\ol{\bds\tau}^{k}-\mb{x}^*\right\|^2\right].   \nonumber
\end{align}
\end{corollary}
\begin{proof}
From Lemma~\ref{var_bound_svrg}, we have:~$\forall k\geq0$,
\begin{align}
&\mathbb{E}\left[\left\|\mb{v}^{k+1}-\nabla \mb{f}(\mb{x}^{k+1})\right\|^2\right]\nonumber\\ 
\leq&~4L^2\mathbb{E}\left[\left\|\mb{x}^{k+1}-W_\infty\mb{x}^{k+1}\right\|^2\right] + 4L^2\mathbb{E}\left[n\left\|\ol{\mb{x}}^{k+1}-\mb{x}^*\right\|^2\right] \nonumber\\
&+ 4L^2\mathbb{E}\left[\left\|\bds\tau^{k+1}-W_\infty\bds\tau^{k+1}\right\|^2\right] 
\nonumber\\
&+ 4L^2\mathbb{E}\left[n\left\|\ol{\bds\tau}^{k+1}-\mb{x}^*\right\|^2\right]. \label{var_bound_svrg_k+1_0}
\end{align}
Recall that~$\bds\tau^{k+1} = \mb{x}^{k+1}$ if~$\bmod(k+1,T)=0$; otherwise, $\bds\tau^{k+1} = \bds\tau^{k}$. We first derive upper bounds on the last two terms in~\eqref{var_bound_svrg_k+1_0} for these two cases seperately. On the one hand, if~$\bmod(k+1,T)\neq0$, we have that
\begin{align}
&4L^2\mathbb{E}\left[\left\|\bds\tau^{k+1}-W_\infty\bds\tau^{k+1}\right\|^2\right] + 4L^2\mathbb{E}\left[n\left\|\ol{\bds\tau}^{k+1}-\mb{x}^*\right\|^2\right]\nonumber\\
=&~4L^2\mathbb{E}\left[\left\|\bds\tau^{k}-W_\infty\bds\tau^{k}\right\|^2\right] + 4L^2\mathbb{E}\left[n\left\|\ol{\bds\tau}^{k}-\mb{x}^*\right\|^2\right]. \label{mod1}
\end{align}
On the other hand, if~$\bmod(k+1,T) = 0$, we have that
\begin{align}
&4L^2\mathbb{E}\left[\left\|\bds\tau^{k+1}-W_\infty\bds\tau^{k+1}\right\|^2\right] + 4L^2\mathbb{E}\left[n\left\|\ol{\bds\tau}^{k+1}-\mb{x}^*\right\|^2\right]\nonumber\\
=&~4L^2\mathbb{E}\left[\left\|\mb{x}^{k+1}-W_\infty\mb{x}^{k+1}\right\|^2\right] \nonumber\\
&\qquad+ 4L^2\mathbb{E}\left[n\left\|\ol{\mb{x}}^{k+1}-\mb{x}^*\right\|^2\right]. \label{mod2}
\end{align}
Therefore, combining~\eqref{mod1} and~\eqref{mod2}, we have that~$\forall k\geq0$:
{\small\begin{align}
&4L^2\mathbb{E}\left[\left\|\bds\tau^{k+1}-W_\infty\bds\tau^{k+1}\right\|^2\right] + 4L^2\mathbb{E}\left[n\left\|\ol{\bds\tau}^{k+1}-\mb{x}^*\right\|^2\right]\nonumber\\
\leq&~
4L^2\mathbb{E}\left[\left\|\mb{x}^{k+1}-W_\infty\mb{x}^{k+1}\right\|^2\right] + 4L^2\mathbb{E}\left[n\left\|\ol{\mb{x}}^{k+1}-\mb{x}^*\right\|^2\right] \nonumber\\
&~+ 4L^2\mathbb{E}\left[\left\|\bds\tau^{k}-W_\infty\bds\tau^{k}\right\|^2\right] + 4L^2\mathbb{E}\left[n\left\|\ol{\bds\tau}^{k}-\mb{x}^*\right\|^2\right] \label{mod}
\end{align}}Next, we apply~\eqref{mod} in~\eqref{var_bound_svrg_k+1_0} to obtain
{\small\begin{align}
&\mathbb{E}\left[\left\|\mb{v}^{k+1}-\nabla \mb{f}(\mb{x}^{k+1})\right\|^2\right]\nonumber\\ 
\leq&~8L^2\mathbb{E}\left[\left\|\mb{x}^{k+1}-W_\infty\mb{x}^{k+1}\right\|^2\right] + 8L^2\mathbb{E}\left[n\left\|\ol{\mb{x}}^{k+1}-\mb{x}^*\right\|^2\right] \nonumber\\
&+ 4L^2\mathbb{E}\left[\left\|\bds\tau^{k}-W_\infty\bds\tau^{k}\right\|^2\right] + 4L^2\mathbb{E}\left[n\left\|\ol{\bds\tau}^{k}-\mb{x}^*\right\|^2\right]. \label{var_bound_svrg_k+1_1}
\end{align}}The proof follows by using~\eqref{consensus2},~\eqref{opt_2} as well as Lemma~\ref{var_bound_svrg} in~\eqref{var_bound_svrg_k+1_1} and by simplifying the resulting inequality.  
\end{proof}

\subsection{Main results for~\textbf{\texttt{GT-SVRG}}} We now apply the upper bounds on the variance of the gradient estimator~$\mb{v}^k$ in \textbf{\texttt{GT-SVRG}} obtained in the previous subsection to refine the inequalities derived for the general dynamical system~\eqref{alg_a}-\eqref{alg_b} in Section~\ref{general} and establish the explicit complexity for~\textbf{\texttt{GT-SVRG}}. We first apply the upper bound on~$\mathbb{E}[\|\mb{v}^k-\nabla\mb{f}(\mb{x}^k)\|^2]$ in Lemma~\ref{var_bound_svrg} to~\eqref{opt_2} to obtain~$\forall k\geq0$:
\begin{align}
&\mathbb{E}\left[n\left\|\ol{\mb x}^{k+1}-\mb{x}^*\right\|^2\right] \nonumber\\
\leq&~L^2\alpha\left(\frac{1}{\mu}+\frac{4\alpha}{n}\right)\mathbb{E}\left[\left\|\mb{x}^k-W_\infty\mb{x}^k\right\|^2\right] \nonumber\\
&+ \left(1-\mu\alpha + \frac{4L^2}{n}\alpha^2 \right)\mathbb{E}\left[n\left\|\ol{\mb{x}}^k-\mb{x}^*\right\|^2\right] \nonumber\\
&+ \frac{4L^2\alpha^2}{n}\mathbb{E}\left[\left\|\bds\tau^k-W_\infty\bds\tau^k\right\|^2\right]  \nonumber\\
&+ \frac{4L^2\alpha^2}{n}\mathbb{E}\left[n\left\|\ol{\bds\tau}^k-\mb{x}^*\right\|^2\right].
\end{align}
If~$0<\alpha\leq\frac{1}{4\mu}$, we have~$(\frac{1}{\mu}+\frac{4\alpha}{n})\leq\frac{2}{\mu}$; if~$0<\alpha\leq\frac{n\mu}{8L^2}$, we have~$1-\mu\alpha + \frac{4L^2}{n}\alpha^2\leq1-\frac{\mu\alpha}{2}$. Therefore, if~$0<\alpha\leq\frac{\mu}{8L^2}$, we have~$k\geq0$:
\begin{align}\label{gtsvrg_2}
&\mathbb{E}\left[n\left\|\ol{\mb x}^{k+1}-\mb{x}^*\right\|^2\right] \nonumber\\
\leq&~\frac{2L^2\alpha}{\mu}\mathbb{E}\left[\left\|\mb{x}^k-W_\infty\mb{x}^k\right\|^2\right] \!+\! \left(1-\frac{\mu\alpha}{2} \right)\mathbb{E}\left[n\left\|\ol{\mb{x}}^k-\mb{x}^*\right\|^2\right] \nonumber\\
&+ \frac{4L^2\alpha^2}{n}\mathbb{E}\left[\left\|\bds\tau^k-W_\infty\bds\tau^k\right\|^2\right] \nonumber\\
&+ \frac{4L^2\alpha^2}{n}\mathbb{E}\left[n\left\|\ol{\bds\tau}^k-\mb{x}^*\right\|^2\right].
\end{align}
Next, we apply the upper bounds on~$\mathbb{E}[\|\mb{v}^k-\nabla\mb{f}(\mb{x}^k)\|^2]$ and~$\mathbb{E}[\|\mb{v}^{k+1}-\nabla\mb{f}(\mb{x}^{k+1})\|^2]$ in Lemma~\ref{var_bound_svrg} and Corollary~\ref{var_bound_svrg_2} to Lemma~\ref{general_3} and obtain:
$\forall k\geq0$,
\begin{align}\label{gtsvrg_3}
&\mathbb{E}\left[\left\|\mb{y}^{k+1}-W_\infty\mb{y}^{k+1}\right\|^2\right]\nonumber\\
\leq&~\frac{120L^2}{1-\sigma^2}\mathbb{E}\left[\left\|\mb{x}^k-W_\infty\mb{x}^k\right\|^2\right] + \frac{87L^2}{1-\sigma^2}\mathbb{E}\left[n\left\|\ol{\mb{x}}^k-\mb{x}^*\right\|^2\right] \nonumber\\
+& \frac{3+\sigma^2}{4}\mathbb{E}\left[\left\|\mb{y}^k-W_\infty\mb{y}^k\right\|^2\right]  \nonumber\\
+& \frac{38L^2}{1-\sigma^2}\mathbb{E}\left[\left\|\bds\tau^{k}-W_\infty\bds\tau^{k}\right\|^2\right] + \frac{38L^2}{1-\sigma^2}\mathbb{E}\left[n\left\|\ol{\bds\tau}^{k}-\mb{x}^*\right\|^2\right], 
\end{align}
if~$0<\alpha\leq\frac{1-\sigma^2}{14\sqrt{2}L}$.
Now, we write Lemma~\ref{consensus1},~\eqref{gtsvrg_2} and~\eqref{gtsvrg_3} jointly in an entry-wise linear matrix inequality that characterizes the evolution of \textbf{\texttt{GT-SVRG}} in the following proposition. 
\begin{proposition}\label{prop_gtsvrg}
Let Assumptions~\ref{sc},~\ref{smooth} and~\ref{connect} hold a nd Consider the iterates~$\{\mb{x}^k\},\{\mb{y}^k\},\{\mb{v}^k\}$ generated by \textbf{\texttt{GT-SVRG}}. If the step-size~$\alpha$ follows~$0<\alpha\leq\frac{\mu(1-\sigma^2)}{14\sqrt{2}L^2}$, then the following linear matrix inequality hold~$\forall k\geq0$:
\begin{align}\label{main0_svrg}
\mb{u}^{k+1} \leq J_\alpha\mb{u}^{k} + H_\alpha\wt{\mb{u}}^{k},
\end{align} 
where~$\mb{u}^k,\wt{\mb{u}}^k\in\mathbb{R}^3$ and~$J_\alpha,H_\alpha\in\mathbb{R}^{3\times3}$ are defined as
{\small\begin{align}
	\mb{u}^k=&
	\begin{bmatrix}
	\mathbb{E}\left[\left\|\mb{x}^{k}-W_\infty\mb{x}^{k}\right\|^2\right]\\ 
	\mathbb{E}\left[n\left\|\ol{\mb x}^{k}-\mb{x}^*\right\|^2\right]\\
	\mathbb{E}\Big[\frac{\left\|\mb{y}^{k} - W_\infty\mb{y}^{k}\right\|^2}{L^2}\Big]
	\end{bmatrix},\quad
	\wt{\mb{u}}^k=
	\begin{bmatrix}
	\mathbb{E}\left[\left\|\bds\tau^{k}-W_\infty\bds\tau^k\right\|^2\right] \\
	\mathbb{E}\left[n\left\|\ol{\bds\tau}^{k}-\mb{x}^*\right\|^2\right] \\
	0
	\end{bmatrix},
	\nonumber\\
	J_\alpha=&
	\begin{bmatrix}
	\dfrac{1+\sigma^2}{2} & 0 & \dfrac{2\alpha^2L^2}{1-\sigma^2} \\[1.5ex]
	\dfrac{2L^2\alpha}{\mu}&  1-\dfrac{\mu\alpha}{2} & 0\\[1.5ex]
	\dfrac{120}{1-\sigma^2} & \dfrac{87}{1-\sigma^2} & \dfrac{3+\sigma^2}{4}
	\end{bmatrix}, \nonumber\\
	H_\alpha =& 
	\begin{bmatrix}
	0 & 0 & 0 \\[1.5ex]
	\dfrac{4L^2\alpha^2}{n}&  \dfrac{4L^2\alpha^2}{n} & 0\\[1.5ex]
	\dfrac{38}{1-\sigma^2} & \dfrac{38}{1-\sigma^2} & 0 
	\end{bmatrix}.
	\nonumber
\end{align}}
\end{proposition} 
Note that~$T$ is the number of the inner loop iterations of \textbf{\texttt{GT-SVRG}}. We will show that the subsequence~$\{\mb{u}^{tT}\}_{t\geq0}$ of~$\{\mb{u}^k\}_{k\geq0}$, which corresponds to the outer loop updates of~\textbf{\texttt{GT-SVRG}}, converges to zero linearly, based on which the total complexity of~\textbf{\texttt{GT-SVRG}} will be established, in terms of the number of total component gradient computations required at each node to find the solution~$\mb{x}^*$. 
We now recall from Algorithm~\ref{GT-SVRG} that~$\forall k\geq0$,~$\bds\tau^{k+1} = \mb{x}_i^{k+1}$ if~$\bmod(k+1,T)=0$; else~$\bds\tau^{k+1} = \bds\tau^{k}$. Therefore,~$\forall t\geq0$ and $tT\leq k\leq(t+1)T-1$, we have~$\bds\tau^{k} = \mb{x}^{tT}$. Based on this discussion,~\eqref{main0_svrg} can be rewritten as the following dynamical system with delays: 
\begin{align*}
\mb{u}^{k+1} \leq J_\alpha\mb{u}^{k} + H_\alpha\mb{u}^{tT},~ \forall k\in[tT,(t+1)T-1],~ \forall t\geq0. 
\end{align*} 
We then recursively apply the above inequality over~$k$ to obtain the evolution of the outer loop iterations~$\{\mb{u}^{tT}\}_{t\geq0}$:
\begin{align}\label{main_2_svrg}
\mb{u}^{(t+1)T} 
\leq&\left(J_{\alpha}^T + \textstyle\sum_{l=0}^{T-1}J_{\alpha}^{l}H_\alpha\right) \mb{u}^{tT},\quad \forall t\geq0.  
\end{align}
Clearly, to show the linear decay of~$\{\mb{u}^{tT}\}_{t\geq0}$, it sufficies to find the range of~$\alpha$ such that~$\rho\big(J_{\alpha}^T + \textstyle\sum_{l=0}^{T-1}J_{\alpha}^{l}H_\alpha\big)<1$.
To this aim, we first derive the range of~$\alpha$ such that~$\rho(J_\alpha)<1$. 

\begin{lem}\label{J}
Let Assumptions~\ref{sc},~\ref{smooth},~\ref{connect} hold and consider the system matrix~$J_\alpha$ defined in Proposition~\ref{prop_gtsvrg}. If the step-size~$\alpha$ follows~$0<\alpha\leq\frac{(1-\sigma^2)^2}{187QL}$, then
\begin{align}
\rho(J_\alpha) \leq \mn{J_\alpha}_{\infty}^{\bds\delta} \leq 1-\tfrac{\mu\alpha}{4},
\end{align}
where~$\bds\delta = \left[1, 8Q^2, \frac{6528Q^2}{(1-\sigma^2)^2}\right]^\top$.
\end{lem}
\begin{proof}
In the light of Lemma~\ref{rho_bound}, we solve for the range of~$\alpha$ and a positive vector~$\bds\delta = [\delta_1,\delta_2,\delta_3]$ such that the following entry-wise linear matrix inequality holds:
$$J_\alpha\bds\delta\leq\left(1-\tfrac{\mu\alpha}{4}\right)\bds\delta,$$ which can be written equivalently as
\begin{align}
&\frac{\mu\alpha}{4} + \frac{2L^2\alpha^2}{1-\sigma^2}\frac{\delta_3}{\delta_1} \leq~\frac{1-\sigma^2}{2}, 
\label{v1'}\\
&8Q^2\delta_1 \leq~\delta_2,
\label{v2'}\\
&\frac{\mu\alpha}{4}\leq~\frac{1-\sigma^2}{4}-\frac{120}{1-\sigma^2}\frac{\delta_1}{\delta_3} 
-\frac{87}{1-\sigma^2}\frac{\delta_2}{\delta_3}.
 \label{v3'}
\end{align}
Based on~\eqref{v2'}, we set~$\delta_1 = 1$ and~$\delta_2 = 6Q^2$. With~$\delta_1$ and~$\delta_2$ being fixed, we next choose~$\delta_3>0$ such that the RHS of~\eqref{v3'} is positive, i.e,~$
\frac{1-\sigma^2}{4\delta_3}\big(\delta_3-\frac{480+2784Q^2}{(1-\sigma^2)^2}\big) >0.$
It suffices to set~$\delta_3 = \frac{6528Q^2}{(1-\sigma^2)^2}$. Now, with the previously fixed values of~$\delta_1,\delta_2,\delta_3$, in order to make~\eqref{v3'} hold, it suffices to choose~$\alpha$ such that~$0<\alpha\leq\tfrac{1-\sigma^2}{2\mu}$.
Similary, it can be verified that in order to make~\eqref{v1'} hold, it sufficies to make~$\alpha$ satisfy~$0< \alpha\leq \tfrac{(1-\sigma^2)^2}{187QL}$,
which completes the proof.
\end{proof}
We note that if the step-size~$\alpha$ satisfies the condition in Lemma~\ref{J}, we have~$\rho(J_\alpha)<1$. Moreover, since~$J_\alpha$ is nonnegative, we have that~$\sum_{l=0}^{T-1}J_\alpha^l\leq\sum_{l=0}^{\infty}J_\alpha^l = (I_3 - J_\alpha)^{-1}$. Therefore, following from~\eqref{main_2_svrg}, we have:
\begin{align}\label{main_3_svrg}
\mb{u}^{(t+1)T} \leq \left(J_{\alpha}^T + (I_3-J_\alpha)^{-1}H_\alpha\right) \mb{u}^{tT},\quad \forall t\geq0.
\end{align}
The rest of the convergence analysis is to derive the condition on the the number of each inner iterations~$T$ and the step-size~$\alpha$ of \textbf{\texttt{GT-SVRG}} such that the following inequality holds:
$$\rho(J_{\alpha}^T\mb{} + (I_3-J_\alpha)^{-1}H_\alpha)<1.$$
We first show that~$(I_3-J_\alpha)^{-1}H_\alpha$ is sufficiently small under an appropriate weighted matrix norm in the light of Lemma~\ref{rho_bound}.
\begin{lem}\label{JH}
Let Assumptions~\ref{sc},~\ref{smooth} and~\ref{connect} hold. Consider the system matrices~$J_\alpha,H_\alpha$ defined in Proposition~\ref{prop_gtsvrg}.
If the step-size~$\alpha$ follows~$0< \alpha \leq \frac{(1-\sigma^2)^2}{187QL}$, then
\begin{align*}
\mn{(I-J_\alpha)^{-1}H_\alpha}_{\infty}^{\mb{q}}\leq0.66,
\end{align*}
where~$\mb q = \big[1,1,\frac{1453}{(1-\sigma^2)^2}\big]^\top$.
\end{lem}
\begin{proof}
We start by deriving an entry-wise upper bound for the matrix~$(I-J_\alpha)^{-1}$. Note that the determinant of~$(I-J_\alpha)^{-1}$ is given by 
\begin{align*}
\det\left(I_3-J_\alpha\right) = \frac{(1-\sigma^2)^2\mu\alpha}{16}
- \frac{348L^4\alpha^3}{\mu(1-\sigma^2)^2} - \frac{120\alpha^3\mu L^2}{(1-\sigma^2)^2}.
\end{align*}
It can be verified that if~$0<\alpha\leq \frac{(1-\sigma^2)^2}{187QL}$,
\begin{align}
\det\left(I_3-J_\alpha\right) \geq \frac{(1-\sigma^2)^2\mu\alpha}{32}.
\end{align}
Then we derive an entry-wise upper bound for~$\operatorname{adj}(I_3-J_\alpha)$, where~$\mbox{adj}(\cdot)$ denotes the adjugate of the argument matrix and we denote~$\left[\operatorname{adj}\left(\cdot\right)\right]_{i,j}$ as its~$i,j$th entry: 
{\small\begin{align}
\left[\operatorname{adj}\left(I-J_\alpha\right)\right]_{1,2}
&= \frac{174L^2\alpha^2}{\left(1-\sigma^2\right)^2},
\quad
\left[\operatorname{adj}\left(I-J_\alpha\right)\right]_{1,3}
= \frac{\mu L^2\alpha^{3}}{1-\sigma^2}, \nonumber\\
\left[\operatorname{adj}\left(I-J_\alpha\right)\right]_{2,2}
&\leq\frac{(1-\sigma^2)^2}{8}, \quad
\left[\operatorname{adj}\left(I-J_\alpha\right)\right]_{2,3}
=\frac{4L^4\alpha^3}{\mu(1-\sigma^2)},\nonumber\\
\left[\operatorname{adj}\left(I-J_\alpha\right)\right]_{3,2}
&=\frac{87}{2}, \qquad\quad~~
\left[\operatorname{adj}\left(I-J_\alpha\right)\right]_{3,3}
=\frac{\mu\alpha(1-\sigma^2)}{4}.\nonumber
\end{align}}With the help of the above calculations, an entry-wise upper bound for~$\left(I_3-J_\alpha\right)^{-1}H_\alpha=\frac{\operatorname{adj}(I_3-J_\alpha)}{\det(I_3-J_\alpha)}H_\alpha$ can be obtained, i.e., if~$0<\alpha\leq\frac{(1-\sigma^2)^2}{187QL}$, we have
{\small\begin{align*}
\left(I_3-J_\alpha\right)^{-1}H_\alpha 
\leq
\begin{bmatrix}
0.039 & 0.039 & 0 \\[1.5ex]
0.23 &  0.23 & 0\\[1.5ex]
\dfrac{334}{(1-\sigma^2)^2} & \dfrac{334}{(1-\sigma^2)^2} & 0
\end{bmatrix}.
\end{align*}}Using Lemma~\ref{rho_bound} in a similar way as the proof of Lemma~\ref{J}, it can be verified that~$\big(\left(I_3-J_\alpha\right)^{-1}H_\alpha\big)\mb{q}\leq0.6 6\mb{q}$, where~$\mb q = [1,1,\frac{1453}{(1-\sigma^2)^2}]^\top$, which completes the proof. 
\end{proof}
Note that we use two different weighted matrix norms to bound~$J_\alpha$ and~$(I-J_\alpha)^{-1}H_\alpha$ respectively in Lemma~\ref{J} and~\ref{JH}, i.e.,~$\mn{\cdot}_\infty^{\bds\delta}$ and~$\mn{\cdot}_\infty^{\mb{q}}$, where~$\bds\delta = [1,8Q^2,\frac{6528Q^2}{(1-\sigma^2)^2}]^\top$ and~$\mb{q}=[1,1,\frac{1453}{(1-\sigma^2)^2}]^\top$. It can be verified that~\cite{matrix_analysis}:~$\forall X\in\mathbb{R}^{3\times3}$,
\begin{align}\label{equiv}
\mn{X}_\infty^{\mb{q}} \leq 8Q^2\mn{X}_\infty^{\mb{\bds\delta}}.
\end{align}
We next show the linear convergence of the outer loop of \textbf{\texttt{GT-SVRG}}, i.e., the linear decay of the subsequence~$\{\mb{u}^{tT}\}_{t\geq0}$ of~$\{\mb{u}^{k}\}_{k\geq0}$, where~$T$ is the number of inner loop iterations.
\begin{P2}
Consider the iterates~$\{\mb{u}^k\}$ generated by \textbf{\texttt{GT-SVRG}} (defined in Proposition~\ref{prop_gtsvrg}) and recall the recursion in~\eqref{main_3_svrg}:~$\forall t\geq0,\mb{u}^{(t+1)T} 
\leq\big(J_\alpha^T+\left(I-J_\alpha\right)^{-1}H_\alpha\big)\mb{u}^{tT}$.
Note that the weighted vector norm~$\|\cdot\|_\infty^{\mb{q}}$ induces the weighted matrix norm~$\mn{\cdot}_\infty^{\mb{q}}$~\cite{matrix_analysis}. Then using Lemma~\ref{J},~\ref{JH} and~\eqref{equiv}, If the step-size~$\alpha=\frac{(1-\sigma^2)^2}{187QL}$ and the number of inner loop iterations~$T = \frac{1496Q^2}{(1-\sigma^2)^2}\log(200Q)$, then we have:~$\forall t\geq0$,
\begin{align}\label{out}
\left\|\mb{u}^{(t+1)T}\right\|_\infty^{\mb{q}}
\leq&~\mn{J_\alpha^T+\left(I-J_\alpha\right)^{-1}H_\alpha}_{\infty}^{\mb{q}}\left\|\mb{u}^{tT}\right\|_\infty^{\mb{q}}\nonumber\\
\leq&~\left(\mn{J_\alpha^T}_{\infty}^{\mb{q}}+0.66\right)\left\|\mb{u}^{tT}\right\|_\infty^{\mb{q}}\nonumber\\
\leq&~\left(8Q^2\big(\mn{J_\alpha}_{\infty}^{\bds{\delta}}\big)^T+0.66\right)\left\|\mb{u}^{tT}\right\|_\infty^{\mb{q}}\nonumber\\
\leq&~0.7\left\|\mb{u}^{tT}\right\|_\infty^{\mb{q}},
\end{align}
Clearly,~\eqref{out} shows that the outer loop of \textbf{\texttt{GT-SVRG}}, i.e., $\{\mb{x}^{tT}\}_{t\geq0}$, converges to an~$\epsilon$-optimal solution with~$\mc{O}(\log\frac{1}{\epsilon})$ iterations. We further note that in each inner loop of~\textbf{\texttt{GT-SVRG}}, each node~$i$ computes~$(m_i + 2T)$ local component gradients. Therefore, the total number of component gradient computations at each node required is~$
\mc{O}\big(\big(M + \frac{Q^2\log Q}{(1-\sigma)^2}\big)\log\frac{1}{\epsilon}\big),
$ where~$M$ is the largest number of data points over all nodes and the proof follows. 
\end{P2}

\section{Conclusions}\label{conclusions}
In this paper, we have proposed a novel framework for constructing variance-reduced decentralized stochastic first-order methods over undirected and weight-balanced directed graphs that hinge on gradient tracking techniques. In particular, we derive decentralized versions of the centralized $\SAGA$ and \SVRG~algorithms, namely \textbf{\texttt{GT-SAGA}} and \textbf{\texttt{GT-SVRG}}, that achieve accelerated linear convergence for smooth and strongly convex functions compared with existing decentralized stochastic first-order methods. We have further shown that in the ``big data" regimes, \textbf{\texttt{GT-SAGA}} and \textbf{\texttt{GT-SVRG}} achieve non-asymptotic, linear speedups in terms of the number of nodes compared with centralized \SAGA~and \SVRG.


\bibliographystyle{IEEEbib}
\bibliography{DOPT,references}

\begin{IEEEbiography}[{\includegraphics[width=1in,height=1.2in,clip,keepaspectratio]{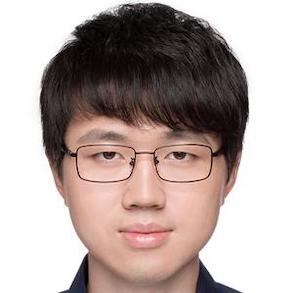}}]{Ran Xin} received his B.S. degree in Mathematics and Applied Mathematics from Xiamen University, China, in 2016, and M.S. degree in Electrical and Computer Engineering from Tufts University in 2018. Currently, he is a Ph.D. candidate in the Electrical and Computer Engineering Department at Carnegie Mellon University. His research interests include convex and nonconvex optimization, stochastic approximation and machine learning.
\end{IEEEbiography}

\begin{IEEEbiography}[{\includegraphics[width=1in,height=1.2in,clip,keepaspectratio]{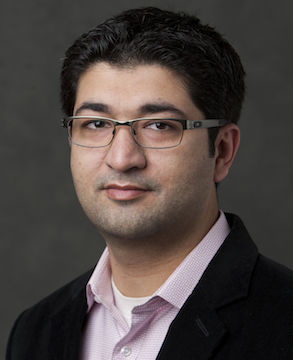}}]{Usman A. Khan} is an Associate Professor of Electrical and Computer Engineering (ECE) at Tufts University, Medford, MA, USA. His research interests include statistical signal processing, network science, and decentralized optimization over multi-agent systems. 
He received his B.S. degree in 2002 from University of Engineering and Technology, Pakistan, M.S. degree in 2004 from University of Wisconsin-Madison, USA, and Ph.D. degree in 2009 from Carnegie Mellon University, USA, all in ECE. 
He 
is currently an Associate Editor of the \textit{IEEE Control System Letters}, \textit{IEEE Transactions Signal and Information Processing over Networks}, and \textit{IEEE Open Journal of Signal Processing}. He is the Lead Guest Editor for the \textit{Proceedings of the IEEE Special Issue on Optimization for Data-driven Learning and Control} slated to publish in Nov. 2020. 
\end{IEEEbiography}

\begin{IEEEbiography}[{\includegraphics[width=1in,height=1.2in,clip,keepaspectratio]{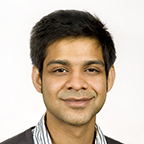}}]{Soummya Kar} received a B.Tech. in Electronics and Electrical Communication Engineering from the Indian Institute of Technology, Kharagpur, India, in May 2005 and a Ph.D. in electrical and computer engineering from Carnegie Mellon University, Pittsburgh, PA, in 2010. From June 2010 to May 2011 he was with the Electrical Engineering Department at Princeton University as a Postdoctoral Research Associate. He is currently a Professor of Electrical and Computer Engineering at Carnegie Mellon University. His research interests span several aspects of decision-making in large-scale networked dynamical systems with applications to problems in network science, cyber-physical systems and energy systems. 

\end{IEEEbiography}

\end{document}